%% file: main.tex
\documentclass[a4paper, 10pt]{amsproc}
\usepackage{defs}
\usepackage{orcidlink}
\usepackage{calc}
\usepackage{graphicx,wrapfig,lipsum}
\usepackage[dvipsnames]{xcolor}
\usepackage{tabularray}
\usepackage{empheq} 
\newtcolorbox{empheqboxed}{colback=white, 
    colframe=black,
    boxrule=0.25mm,
    width=\columnwidth,
    sharpish corners,
    top=-2mm, 
    left=2pt,
    bottom=5pt
}
\definecolor{metablue}{HTML}{0064E0}
\definecolor{metafg}{HTML}{1C2B33}
\definecolor{metabg}{HTML}{F1F4F7}
\definecolor{metabgdeep}{HTML}{D9EFFF}
\definecolor{metagreen}{HTML}{EAFFE8}
\definecolor{metagreen}{HTML}{FCFFEE}
\definecolor{metared}{HTML}{FFEAE8}

\newtheorem{theorem}{Theorem}

\newtheorem{remark}{Remark}

\newtheorem{proposition}{Proposition}
\newtheorem{lemma}{Lemma}
\newtheorem{corollary}{Corollary}

\DeclareSymbolFont{extraup}{U}{zavm}{m}{n}
\DeclareMathSymbol{\varheart}{\mathalpha}{extraup}{86}
\DeclareMathSymbol{\vardiamond}{\mathalpha}{extraup}{87}
\DeclareMathSymbol{\varclub}{\mathalpha}{extraup}{84}
\DeclareMathSymbol{\vardspade}{\mathalpha}{extraup}{85}

\usepackage[framemethod=tikz,xcolor=true]{mdframed}

\newmdenv[backgroundcolor=metabg, roundcorner=0pt, skipabove=4pt, linewidth=0pt, innertopmargin=8pt]{myframe}

\newmdenv[backgroundcolor=metabgdeep, roundcorner=10pt, skipabove=4pt, linewidth=1pt, innertopmargin=4pt]{myOCP}

\newmdenv[backgroundcolor=metared, roundcorner=10pt, skipabove=7pt, linewidth=0pt, innertopmargin=7pt]{myalgo}
\newmdenv[%
    leftmargin=0.5cm,
    backgroundcolor=yellow!10,%
    roundcorner=5pt,%
    tikzsetting={draw=red, line width=2.0pt}%
    ]{SpecialText}%

    \newmdenv[
  backgroundcolor=metabg,
  roundcorner=5pt,
  skipabove=4pt,
  linewidth=0.5pt,
  innertopmargin=1pt,
  tikzsetting={draw=black, line width=0.5pt},
  frametitlefont=\bfseries,
  frametitlebackgroundcolor=blue!10,
  frametitleaboveskip=0.5ex,
  frametitlebelowskip=0.5ex
]{myOCPboxTitle}

\usepackage{cleveref}

\newcounter{ocpproblem}

\crefname{ocpproblem}{Problem}{Problems}
\Crefname{ocpproblem}{Problem}{Problems}

\newenvironment{problembox}[1][]{%
  \refstepcounter{ocpproblem}%
  \begin{myOCPboxTitle}[frametitle={Problem~\theocpproblem#1}]%
}{%
  \end{myOCPboxTitle}%
}

\title[GMM-SB]{Lifted Schr\"{o}dinger Bridges for Gaussian Mixture Endpoints: Projection Gaps and Path-Space Obstructions}

\author[S. Ganguly]{Siddhartha Ganguly\,\orcidlink{0000-0003-2046-2061}}
\author[G. Rapakoulias]{George Rapakoulias\,\orcidlink{0009-0000-8572-710X}}
\author[P. Tsiotras]{Panagiotis Tsiotras\,\orcidlink{0000-0001-7563-4129}}

\thanks{%
S. Ganguly, G. Rapakoulias, and P. Tsiotras are with \faGroup\ Daniel Guggenheim School of Aerospace, \faUniversity\ Georgia Institute of Technology, \faMapMarker\  Atlanta, USA.
	Contact Information: (SG) \faHome\ \url{https://sites.google.com/view/siddhartha-ganguly}, \faEnvelope\ \texttt{sganguly41@gatech.edu}, (GR) \faEnvelope\ \texttt{grap@gatech.edu}, (PT) \faHome\ \url{https://dcsl.gatech.edu/tsiotras.html}, \faEnvelope\ \texttt{tsiotras@gatech.edu}.
}

\begin{document}
\subjclass[2020]{93E20, 93E03, 49K20, 49L99, 58E25, 65K10} 
\keywords{Stochastic optimal control, Schr\"{o}dinger bridges, optimal transport.}

\maketitle

\begin{abstract}
We study stochastic density control between Gaussian-mixture endpoint distributions under Brownian prior dynamics. Since the direct Schr\"{o}dinger bridge between Gaussian mixtures is generally not available in closed form, we introduce a \emph{lifted path-space construction} in which each trajectory is augmented with a source--target component label. 
Consequently, the problem decomposes into Gaussian component-to-component Schr\"{o}dinger bridges with explicit marginal, drift, and cost formulas, while the mixture-level assignment reduces to a finite-dimensional entropic coupling problem with a Sinkhorn scaling form. We then analyze the projection obtained by discarding or forgetting the label. 
By construction, the projected law satisfies the original Gaussian-mixture endpoint constraints, but its relative entropy generally differs from the lifted relative entropy by a nonnegative conditional label-information gap. This gap reveals a \emph{path-space obstruction}: the lifted optimizer cannot, in general, be identified with the direct unlabeled Schr\"{o}dinger bridge after projection.
We also derive the posterior-averaged Markov drift associated with the projected marginal flow, prove a kinetic-energy upper bound, and identify a common path-potential condition under which the projection gap vanishes. Several numerical illustrations showing density and shape control are recorded for a self-contained exposition.
\end{abstract}

\input{sections/1-Introduction}

\input{sections/2-Problem_Setup}

\input{sections/3-Main_Result}

\input{sections/4-Numerics}

\input{sections/Conclusion}

\input{sections/appendix}

\bibliographystyle{amsalpha}
\bibliography{refs}

\end{document}

%% file: sections/1-Introduction.tex
\section{Introduction}\label{sec:introduction}
The Schr\"{o}dinger bridge problem, originally posed by Erwin Schr\"{o}dinger as a thought experiment \cite{ref:SB:original:I,ref:SB:original:II}, provides a stochastic counterpart of optimal transport \cite{ref:OT:book:Villani,ref:Santam:OT:book,ref:santambrogio2023course}: among all path measures matching two prescribed endpoint distributions, one seeks the most likely evolution relative to a prior stochastic dynamics \cite{ref:Leonard:SB,ref:OT:Diff:GP,ref:GP:MC:OT:book,follmer1988random}.
In the Brownian case, and more generally for diffusion priors, this variational problem is equivalent, through Girsanov's theorem \cite[Chapter 3, \S 3.5]{ref:Shreve:Karatzas}, \cite{ref:YC:TG:MP:SB-and-OT,ref:SOC:SB:YC}, to a stochastic optimal control problem with a quadratic control energy cost.  
This viewpoint has led to a rich set of connections between entropy minimization, optimal transport, stochastic control \cite{ref:dai:pra:1990markov,ref:dai:pra:1991stochastic,ref:AH:Lie:Group:SB,ref:Adu:SB:SubRiemann,ref:Adu:SB:Average:Sys}, covariance steering \cite{chen2015optimal,chen2015optimal2,chenoptimalp3, ref:KI:KK:Disc:SB,ref:LQR:SB,ref:GR:AP:PS:MF-SB}, and computational algorithms based on Sinkhorn-type scaling \cite{ref:SB:3,ref:SB:1,ref:SB:4,ref:caluya2021wasserstein,ref:SB:halder:exact}.

In parallel, Schr\"{o}dinger bridge formulations have recently become increasingly visible in modern generative modeling \cite{ruthotto2021introduction}. 
Diffusion and score-based models can be interpreted as stochastic transports \cite{ref:OT:Diff:GP} between a simple prior distribution and a data distribution.  
Several recent works formulate or approximate this transport through Schr\"{o}dinger bridge, entropic optimal transport, score matching, or flow matching perspectives \cite{de2021diffusion,song2020score,shi2023diffusion,tong2023simulation,lipman2022flow,liu2022deep, liu2024generalized}. 
For example, diffusion Schr\"{o}dinger bridge methods connect score-based generative modeling with finite-time entropy-regularized transport, while bridge-matching and simulation-free formulations place diffusion and flow-matching models within a common stochastic transport framework. 
These developments show that SBs are not only a classical object in stochastic control and optimal transport \cite{ref:adj:match:soc}, but also provide a useful mathematical language for generative sampling, data-to-data translation, and learned distributional dynamics.

Despite this general theory, obtaining explicit and computationally viable density-steering laws remains difficult for non-Gaussian endpoint distributions.
Gaussian endpoints are, however, special: under Brownian or linear-Gaussian priors, the corresponding Schr\"{o}dinger bridge remains Gaussian, and its marginal flow, optimal drift, and energy cost admit closed-form or finite-dimensional descriptions \cite{ref:SB:closed:form,mallasto2022entropy,chenoptimalp3}. 
Related attempts to analytically verify tractable classes of optimal-transport or transport-type problems beyond the purely Gaussian setting include Gaussian-mixture and Wasserstein-type constructions ~\cite{chen2018optimal, delon_wasserstein-type_2020, ref:GoWithTheFlow}.
However, many density-control and generative-modeling problems of interest are intrinsically multimodal. A natural and expressive model class is given by \emph{Gaussian mixture models}
\[
    \rho_0 \Let \sum_{i=1}^{N_1}\alpha_i^0 p_i^0,\qquad
    \rho_1 \Let \sum_{j=1}^{N_2}\alpha_j^1 p_j^1,
\]
where each component $p_i^0$ and $p_j^1$ is Gaussian, but the full endpoint densities are generally \emph{non-Gaussian and multimodal}. 
From the viewpoint of generative modeling, such mixtures provide a simple, yet analytically revealing, model of multimodal distributions: they retain Gaussian component-wise structure while exhibiting the ambiguity in assigning mass between distinct modes. 
The direct Schr\"{o}dinger bridge between such mixtures is not, in general, known to remain within a finite Gaussian-mixture class with closed-form componentwise dynamics~\cite{ref:GoWithTheFlow}.
Moreover, the mixture representation introduces a structural ambiguity: the state process itself does not carry the component label and therefore does not specify which initial component should be paired with which terminal component.
Thus, while the Gaussian components suggest a tractable decomposition, the original stochastic dynamics remains unlabeled.

This observation motivates the following questions we address in this work:
\begin{myOCP}
 \begin{enumerate}[label={\textup{(\(\mathsf{Q}\)-\alph*)}}, leftmargin=*, widest=b, align=left]
\item \label{ques:1} Can the Gaussian-mixture structure be exploited without reducing the problem to a purely Gaussian one?
\item \label{ques:2} Can this be achieved by lifting the path space with an auxiliary source--target component label, so that the mixture-to-mixture construction decomposes into tractable Gaussian component-to-component subproblems together with a principled rule for assigning mass between components?
\item \label{ques:3}  After solving the resulting lifted, labeled problem on the augmented path space, what precisely is its relationship with the original unlabeled Schr\"{o}dinger bridge problem on the state path space?
\item \label{ques:4} Can the projected, state-only law be represented by a Markov feedback drift that steers the prescribed Gaussian-mixture endpoints?
\end{enumerate}   
\end{myOCP}

\begin{figure}[b]
\centering
\begin{tikzpicture}[
    >=Latex,
    node distance=8mm and 8mm,
    box/.style={
        draw,
        rounded corners=3pt,
        align=center,
        inner sep=4pt,
        text width=0.28\textwidth,
        minimum height=1.65cm,
        font=\scriptsize,
        line width=0.55pt,
        blur shadow={
            shadow xshift=1.4pt,
            shadow yshift=-1.4pt,
            shadow blur steps=9,
            opacity=.85
        }
    },
    bluebox/.style={
        box,
        fill=blue!5,
        draw=blue!55!black
    },
    tealbox/.style={
        box,
        fill=cyan!6,
        draw=cyan!45!black
    },
    orangebox/.style={
        box,
        fill=orange!8,
        draw=orange!60!black
    },
    greenbox/.style={
        box,
        fill=green!6,
        draw=green!45!black
    },
    purplebox/.style={
        box,
        fill=purple!6,
        draw=purple!50!black
    },
    redbox/.style={
        box,
        fill=red!4,
        draw=red!55!black,
        dashed
    },
    arrow/.style={
        -{Stealth[length=2.2mm,width=1.8mm]},
        line width=0.85pt,
        draw=black!70,
        line cap=round,
        shorten <=1pt,
        shorten >=2pt
    },
    dashedarrow/.style={
        -{Stealth[length=2.2mm,width=1.8mm]},
        line width=0.8pt,
        draw=red!65!black,
        dashed,
        line cap=round,
        shorten <=1pt,
        shorten >=2pt
    }
]

\node[bluebox] (data) {
\textbf{GMM endpoints}\\[0.5mm]
\[
\rho_0=\sum_i \alpha_i^0 p_i^0,\,
\rho_1=\sum_j \alpha_j^1 p_j^1
\]
Original dynamics is unlabeled on \(\Omega\)
};

\node[tealbox, right=of data] (pairwise) {
\textbf{Pairwise Gaussian SBs}\\[0.5mm]
For each \((i,j)\), compute
\[
\probmeas^{ij},\,\rho_t^{ij},\,u_t^{ij},\,C_{ij}
\]
Gaussian component-to-component layer (Prop. \ref{prop:gaussian_endpoint_coupling} and \ref{prop:pairwise_gaussian_bridge_formulas})
};

\node[orangebox, right=of pairwise] (coupling) {
\textbf{Lifted label coupling}\\[0.5mm]
Solve, for \(\pi\in\Pi(\alpha^0,\alpha^1)\),
\[
\min_{\pi}\;
\sum_{i,j}\pi_{ij}C_{ij}
+\varepsilon\KL(\pi\|\eta)
\]
Sinkhorn form when \(\eta_{ij}>0\)
};

\node[tealbox, below=of data] (lifted) {
\textbf{Lifted labeled law}\\[0.5mm]
Introduce \(\mathcal{Z} \Let I\times J\):
\[
\liftedmeas
=
\sum_{i,j}\pi_{ij}\delta_{(i,j)}
\otimes \probmeas^{ij}
\]
\[
\bar{\refmeas}^{\eta}
=
\sum_{i,j}\eta_{ij}\delta_{(i,j)}
\otimes \mathsf R_i
\]
Exact on \(\mathcal{Z}\times\Omega\)
};

\node[greenbox, below=12mm of pairwise] (projection) {
\textbf{Projection and entropy gap}\\[0.5mm]
Forget the label: \(\optmeas^\pi=(\proj_\Omega)_\#\liftedmeas\)
\[
\resizebox{0.96\linewidth}{!}{\(\displaystyle \KL(\liftedmeas\|\bar{\refmeas}^{\eta})=\KL(\optmeas^\pi\|\mathsf R_{\rho_0})+\gap(\pi,\eta)\)}
\]
(Thrm. \ref{thm:finite_dim_reduction} and \ref{thm:projection-gap-identity})
};

\node[greenbox, below=of coupling] (drift) {
\textbf{Projected Markov drift}\\[0.5mm]
\(\rho_t^\pi = \sum_{i,j}\pi_{ij}\rho_t^{ij}\)
\(\bar u_t^\pi(x) =\sum_{i,j}\gamma_{ij}(t,x)u_t^{ij}(x)\)
\(\gamma_{ij}(t,x) = \frac{\pi_{ij}\rho_t^{ij}(x)}{\rho_t^\pi(x)}\)
\[
\bar u_t^\pi(x_t)=\E^{\optmeas^\pi}[b_t^\pi\mid x_t]
\]
Markov projection of the hidden-label drift
};

\node[purplebox, below=15mm of projection, text width=0.4\textwidth, xshift=-10mm] (guarantees) {
\textbf{Feasibility, energy, and Markovization}\\[0.5mm]
\((\rho_t^\pi,\bar u_t^\pi)\) satisfies the Fokker--Planck equation and \(\rho_0^\pi=\rho_0,\, \rho_1^\pi=\rho_1,\) also we have: (a) energy bound (Prop. \ref{prop:projected_feasibility_energy}) and (b) entropy decomposition (Thrm. \ref{thrm:gaps:all:together}).
};

\node[redbox, right=of guarantees] (zerogap) {
\textbf{Special zero-gap case}\\[0.5mm]
Under certain structural conditions 
\[
\gap(\pi,\eta)=0.
\]
(Prop. \ref{prop:common-potential-zero-gap} and Rem. \ref{rem:common-potential-interpretation})
};

\draw[arrow] (data.east) -- (pairwise.west);
\draw[arrow] (pairwise.east) -- (coupling.west);

\draw[arrow] (pairwise.south west) -- (lifted.north east);
\draw[arrow] (coupling.south west) -- (lifted.north east);

\draw[arrow] (lifted.east) -- (projection.west);
\draw[arrow] (projection.east) -- (drift.west);

\draw[arrow] (projection.south) -- (guarantees.north);
\draw[arrow] (drift.south west) -- (guarantees.north east);

\draw[dashedarrow] (projection.south east) -- (zerogap.north west);

\end{tikzpicture}
\caption{Roadmap of our lifted Schr\"{o}dinger bridge construction.}
\label{fig:lifted-sb-roadmap}
\end{figure}
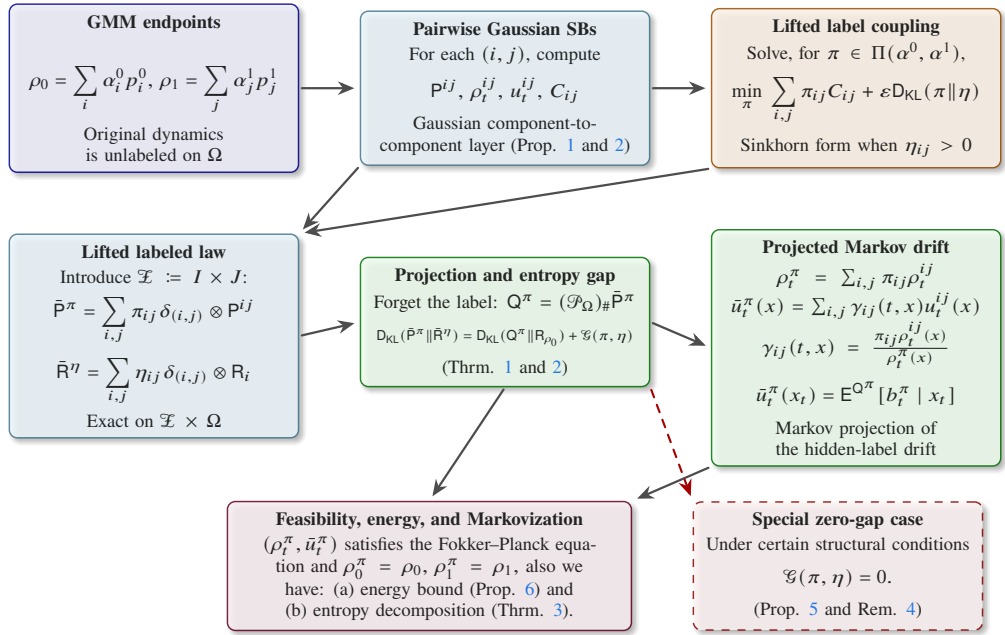

\textcolor{black}{These questions bring us to the main conceptual contribution of this manuscript, namely, recasting Gaussian-mixture Schr\"{o}dinger bridge constructions from a lifted path-space perspective. 
Rather than treating a Gaussian-mixture bridge only as a computational superposition of pairwise Gaussian bridges, we formulate an augmented entropy-minimization problem on a label--trajectory space, in which the source--target component assignment is itself a probabilistic object. 
In this formulation, the prior component coupling \(\eta\) specifies the reference assignment structure in the lifted space, while the optimized coupling \(\pi\) balances this prior assignment against the pairwise Gaussian bridge costs. 
This viewpoint separates three issues that are often conflated: the exact solution of the lifted labeled problem, the feasibility of its projection onto the original unlabeled path space, and the information gap that prevents the projected law from being identified, in general, with the direct unlabeled Schr\"{o}dinger bridge. 
The detailed contributions are as follows (a schematic is given in Figure \ref{fig:lifted-sb-roadmap}).}

\begin{itemize}[leftmargin=*]
    \item \textbf{Lifted formulation for Gaussian-mixture endpoints.}
    We address \ref{ques:1} by augmenting each trajectory with an auxiliary source--target component label. For a given label \((i,j)\), the problem becomes a Gaussian Schr\"{o}dinger bridge from the initial component \(p_i^0\) to the terminal component \(p_j^1\). 
    Thus, the Gaussian-mixture representation is used as a structural decomposition device. This is an exact construction on the augmented label--trajectory space; it is not, by itself, an assertion that the resulting projected law is the optimizer of the original unlabeled Schr\"{o}dinger bridge problem.

    \item \textbf{Finite-dimensional entropic assignment of component mass.}
    We address \ref{ques:2} by showing that, after the pairwise Gaussian bridge costs have been computed, the remaining lifted problem reduces to a finite-dimensional entropic coupling problem over the mixture weights. 
    When the prior coupling \(\eta\) has full support, the optimizer has a Sinkhorn scaling form. This separates the componentwise continuous bridge calculations from the discrete component-assignment problem. 
  
    \item \textbf{Projection gap between the lifted and unlabeled problems.}
    We address \ref{ques:3} by projecting the lifted law back to the original path space, i.e., by forgetting the auxiliary label.
    The projected law is feasible for the prescribed Gaussian-mixture endpoints. However, its relative entropy with respect to the original unlabeled Brownian reference generally differs from the lifted relative entropy. 
    We prove a projection identity showing that the lifted entropy decomposes into the projected path-space entropy plus a nonnegative conditional label-information gap. This gap identifies precisely the obstruction to interpreting the lifted optimizer as the optimizer of the direct unlabeled Schr\"{o}dinger bridge problem. 
    We also give a structural common path-potential condition under which this projection gap vanishes, showing that the gap is not an unavoidable artifact of lifting.

    \item \textbf{Posterior-averaged Markov density control.}
    We address \ref{ques:4} by constructing a state-only Markov drift obtained by posterior averaging of the pairwise Gaussian bridge drifts. 
    The resulting density--drift pair satisfies the Fokker--Planck equation, matches the prescribed endpoint densities, and obeys a kinetic-energy upper bound. 
    Posterior-averaged Gaussian-mixture bridge feedback laws have appeared in GMMflow-type constructions \cite{ref:GoWithTheFlow}; our contribution is to identify such feedback as the Eulerian Markov projection of a lifted labeled Schr\"{o}dinger bridge and to separate \emph{three distinct objects}: 
    optimality in the augmented labeled space, feasibility after projection to the original state space, and the information gap between the lifted and unlabeled path-space laws. 
    Numerical examples illustrate the lifted construction and projected Markov feedback for Gaussian-mixture steering.
\end{itemize}

\textbf{Notation:} We employ standard notation throughout the article. 
Let \(\N \Let \aset{1,2,\ldots}\) 
denote the set of positive integers. Given \(d\in\N\), we denote the \(d\)-dimensional Euclidean space by \(\R[d]\), equipped with the standard inner product \((x,y)\mapsto x^\top y\) and norm \(x\mapsto \norm{x}\). For \(p\in\N\), \(\mathrm I_p\) denotes the \(p\times p\) identity matrix. For \(q\in\N\), \(\mathbb S^q\), \(\mathbb S^q_+\), and \(\mathbb S^q_{++}\) denote the sets of \(q\times q\) real symmetric, real symmetric positive semidefinite, and real symmetric positive definite matrices, respectively. 
For \(A\in\mathbb S^q_{++}\), \(A^{1/2}\) denotes the unique positive definite square root of the matrix $A$. 
The trace and determinant of a square matrix \(A\) are denoted by \(\tr(A)\) and \(\det(A)\), respectively. Given a metric space \(X\), we write \(\Borelsigalg(X)\) for its Borel sigma-algebra and \(\polish(X)\) for the set of probability measures on \((X,\Borelsigalg(X))\). If \(X\) and \(Y\) are metric spaces, \(T:X\to Y\) is measurable, and \(\mu\in\polish(X)\), then the pushforward of \(\mu\) by \(T\), denoted by \(T_\#\mu\in\polish(Y)\), is defined by
\(\Borelsigalg(Y) \ni A \mapsto (T_\#\mu)(A)\Let \mu(T^{-1}(A))\). 
For \(z\in X\), \(\delta_z\) denotes the Dirac measure at \(z\). For two probability measures \(\mu,\nu\) on a common measurable space \(X\), we write \(\mu\ll\nu\) to mean that \(\mu\) is absolutely continuous with respect to \(\nu\).
The support of a measure is denoted by \(\operatorname{supp}(\cdot)\). For a nonnegative matrix \(\pi \Let (\pi_{ij})\) indexed by finite sets \(\indexsource\) and \(\indexsink\), we write \(\operatorname{supp}\pi \Let \aset[]{(i,j)\in\indexsource\times\indexsink \suchthat \pi_{ij}>0}\). We use \(\mathcal{N}(m,\Sigma)\) to denote the Gaussian law with mean \(m\) and covariance \(\Sigma\), and also its density when the meaning is clear from context.
The relative entropy of \(\mu\) with respect to \(\nu\) is denoted by \(\KL(\mu\|\nu) \Let
\int_{X}\log\left(\frac{\dd\mu}{\dd\nu}\right)\dd\mu\) when \(\mu\ll\nu\) and \(+\infty\), otherwise. For a random variable \(X\) defined on a probability space with law \(Q\), we write
\(\Law_Q(X)\) for its distribution and \(\E^Q[\cdot]\) for expectation under \(Q\). When the underlying law is clear from context, we  write \(\E[\cdot]\).

%% file: sections/2-Problem_Setup.tex

\section{The stochastic optimal control problem}\label{sec:problem}

Fix \(d\in\N\) and a noise level \(\eps>0\).
Let \(\Omega \Let \pathsp(\lcrc{0}{1};\R[d])\) be the space of continuous paths equipped with the Borel \(\sigma\)-algebra, where \(\Omega\) is endowed with the uniform topology. Let \(\Omega \ni \omega\mapsto x_t(\omega)\Let \omega(t) \in \R[d]\) be the canonical process on \(\Omega\).\footnote{All stochastic processes below may be realized on rich enough complete filtered probability spaces with suitably defined filtrations; however, since the analysis is carried out at the level of induced path laws, we work on the canonical path space \(\Omega\) and identify each controlled diffusion with its law on \(\Omega\).} 
We consider controlled diffusions on \(\R[d]\) of the form
\begin{equation}\label{eq:dynamics}
\dd x_t = u_t(x_t)\,\dd t + \sqrt{\eps}\,\dd w_t, \quad t\in\lcrc{0}{1},
\end{equation}
with the following data: 
\begin{enumerate}[label={\textup{(\ref{eq:dynamics}-\alph*)}}, leftmargin=*, widest=b, align=left]
\item \label{prob:data:1} \((w_t)_{t\in \lcrc{0}{1}}\) is an \(\R[d]\)-valued standard Brownian motion;
\item \label{prob:data:2} The drift \(u\) is taken from the regular feedback class \(\mathcal U_{\mathrm{reg}}\), consisting of Borel measurable maps
\begin{align}
    u:\lcrc{0}{1}\times\R[d] \lra \R[d], \quad \lcrc{0}{1} \times \R[d] \ni (\tau,\xi) \mapsto u_{\tau}(\xi) \in \R[d],  \nn
\end{align}
such that \(u_t(\cdot)\) is locally Lipschitz in \(x\), uniformly on compact sets with time-integrable local Lipschitz constants, and has at most linear growth in \(x\). More precisely, for every \(R>0\), there exists \(L_R\in L^1(0,1)\) such that, for a.e. \(t\in\lcrc{0}{1}\),
\begin{align}
    \norm{u_t(x)-u_t(y)}\le L_R(t)\norm{x-y}\quad \text{for all }\norm{x},\norm{y}\le R, \nn
\end{align}
 and there exists \(a\in L^2(0,1)\) such that,
for all \(x\in\R[d]\) and for a.e. \(t\in\lcrc{0}{1}\),
\begin{align}
    \norm{u_t(x)} \le a(t)(1+\norm{x}). \nn
\end{align}
\end{enumerate}
Under the regularity assumptions \ref{prob:data:1}--\ref{prob:data:2}, for every initial law with finite second moment, \eqref{eq:dynamics} admits a \emph{unique strong solution} \cite[Chapter 5, Theorem 5.2.9]{ref:Shreve:Karatzas}, \cite[Chapter 1, \S 6.4]{ref:Stoc:Control:book}, \cite{ref:VSB:Diff:1}.
In particular, the induced law of the process on \(\Omega\) is well-defined.

Fix \(N_1,N_2 \in \N\). Let \(m_i^0,m_j^1\in\R[d]\) and \(\Sigma_i^0,\Sigma_j^1\in\mathbb S_{++}^d\).
Consider the following weights that satisfy
\begin{align}
    \alpha_i^0,\alpha_j^1 > 0, \qquad \sum_{i=1}^{N_1}\alpha_i^0=1, \qquad
    \sum_{j=1}^{N_2}\alpha_j^1=1. \nn
\end{align} 
Let \(\indexsource \Let \aset[]{1,\dots,N_1}\) and \(\indexsink \Let \aset[]{1,\dots,N_2},\) and define the Gaussian densities
\begin{align}
    p_i^0 \Let \mathcal N(m_i^0,\Sigma_i^0),\qquad
    p_j^1 \Let \mathcal N(m_j^1,\Sigma_j^1),\qquad i \in \indexsource , j \in \indexsink. \nn
\end{align}
The initial and terminal densities are \(\dummyx\mapsto \rho_0(\dummyx)\) and \(\dummyx\mapsto \rho_1(\dummyx)\), respectively, which are defined on \(\R[d]\), and have the Gaussian-mixture specification
\begin{equation}\label{eq:mixture_endpoints}
   \rho_0(\dummyx)
    \Let
    \sum_{i=1}^{N_1}\alpha_i^0\,p_i^0(\dummyx),
    \qquad
 \rho_1(\dummyx)
    \Let
    \sum_{j=1}^{N_2}\alpha_j^1\,p_j^1(\dummyx).
\end{equation}
We define the admissible control set associated with the endpoint pair
\((\rho_0,\rho_1)\) by
\begin{align}
\mathcal U(\rho_0,\rho_1)  \Let \left\{  u\in\mathcal U_{\mathrm{reg}} \;\middle\vert\;  
\begin{array}{@{}l@{}}
\text{the solution of \eqref{eq:dynamics} with }x_0\sim\rho_0
    \text{ satisfies }x_1\sim\rho_1,\\[0.35em]
    \displaystyle
    \E\left[\int_0^1\norm{u_t(x_t)}^2\,\dd t\right]<+\infty
\end{array}
\right\}.\nn
\end{align}
The stochastic density-control problem is
\begin{problembox}\label{prob:density-control}
\begin{equation*}
    \inf_{u\in\mathcal U(\rho_0,\rho_1)}
    \E\left[\int_0^1 \frac{1}{2}\norm{u_t(x_t)}^2\,\dd t
    \right] =
    \begin{aligned}
    & \inf_{u\in\mathcal U_{\mathrm{reg}}}
    &&  \E\left[\int_0^1 \frac{1}{2}\norm{u_t(x_t)}^2\,\dd t
    \right]  \\
    & \sbjto
    &&  \begin{cases}
    \text{dynamics }\eqref{eq:dynamics},\,
    x_0\sim\rho_0,\, x_1\sim\rho_1,\\
    \E\int_0^1\norm{u_t(x_t)}^2\,\dd t<+\infty.
    \end{cases}
    \end{aligned}
\end{equation*}
\end{problembox}

\begin{remark}[Extension to linear stochastic control systems]
The lifted Schr\"{o}dinger bridge construction which we will present in the sequel is not limited to the Brownian prior
\begin{align}
    \dd x_t=u_t(x_t)\,\dd t+\sqrt{\eps}\,\dd w_t, \nn
\end{align}
and extends, in the same spirit, to linear stochastic systems of the form  \cite{chen2016optimal, mei_flow_2024}
\begin{align}
\dd x_t=A(t)x_t\,\dd t+B(t)v_t(x_t)\,\dd t+\sqrt{\eps}\,B(t)\,\dd w_t, \nn
\end{align}
with the usual problem data. 
In that case, the reference process is the uncontrolled linear diffusion, and the control cost in \Cref{prob:density-control} is measured in terms of the input \(v_t\), not the total drift. 
The pairwise component bridges remain linear-Gaussian, and hence the Brownian pairwise formulas in the sequel are simply replaced by their corresponding linear-system analogs. 
At the structural level, however, the rest of the theory in the sequel will remain unchanged: the lifted KL decomposition, the finite-dimensional entropic coupling problem, the Sinkhorn scaling form, the projected density construction, the energy bound, and the projection-gap identity all continue to hold in the same form.
\end{remark}

%% file: sections/3-Main_Result.tex

\section{Main results}

\Cref{prob:density-control} is posed on the original state space \(\R[d]\) and is therefore unlabeled: the component indices \(i\in\indexsource\) and \(j\in\indexsink\) in the Gaussian-mixture representation are not observed state variables of the controlled diffusion. 
They only encode a chosen decomposition of the endpoint densities. In this section, we use these component labels as auxiliary variables to construct a lifted path-space formulation. 
This lifted formulation decomposes the mixture-to-mixture problem into pairwise component Schr\"{o}dinger bridges coupled through a finite-dimensional transport plan.

We first recall the path-space relative-entropy representation of the stochastic control cost, then define the pairwise component bridges, and finally show that the lifted problem reduces exactly to a finite-dimensional entropic coupling problem over the mixture weights. 
This construction should be interpreted as an exact reduction of a labeled, lifted problem. Its relation to the original unlabeled density-control problem is addressed afterward through a projection identity.


\subsection{Uncontrolled reference law}

Fix an initial law \(\nu \in \polish(\R[d])\) and let \(\mathsf R^\nu \in \polish(\Omega)\) denote the path law on \(\Omega\) of the uncontrolled diffusion
\begin{equation}\label{eq:uncontrolled}
\dd x_t =\sqrt{\eps}\,\dd w_t \quad \text{with }\Law(x_0) \Let \nu.
\end{equation}
For probability measures \(\optmeas\) and \(\refmeas\) on \((\Omega,\borel(\Omega))\), recall the definition of KL divergence \(\KL(\optmeas\|\refmeas)\). 
Let \(u\in\mathcal U_{\mathrm{reg}}\) and let \(\optmeas^u\) denote the path
law induced by the controlled diffusion
\[
    \dd x_t = u_t(x_t)\,\dd t + \sqrt{\eps}\,\dd w_t \quad \text{with }\Law(x_0)\Let \nu .
\]
Whenever \(\optmeas^u\ll \mathsf R^\nu\), Girsanov's theorem \cite{ref:YC:TG:MP:SB-and-OT} gives
\begin{equation}\label{eq:kl_energy}
    \KL(\optmeas^u\|\mathsf R^\nu)=\frac{1}{2\eps}\E^{\optmeas^u}\left[\int_0^1
    \norm{u_t(x_t)}^2\,\dd t\right].
\end{equation}
Thus, the stochastic density-control cost in \S\ref{sec:problem} can be
represented as a path-space relative-entropy cost. 
In the sequel, this representation will be applied componentwise with \(\nu=p_i^0\), yielding the pairwise Schrödinger bridges between the Gaussian components of the initial and terminal mixtures.

\begin{remark}[Well-posedness of Problem \ref{prob:density-control}]\label{rem:on:well-posedness}
Problem \ref{prob:density-control} is a stochastic optimal control problem with fixed initial and terminal marginals.
Existence and duality results for such fixed-marginal stochastic control problems, including convex running costs, are classical; see, for example, \cite{ref:mikami2006duality,ref:mikami2008optimal,ref:Mikami:OT:book}. 
In the present Brownian setting, the Girsanov identity above identifies the quadratic cost with the relative entropy with respect to the uncontrolled Brownian reference law, and the corresponding path-space formulation is the Brownian Schr\"{o}dinger bridge problem; see also \cite{ref:cattiaux1994minimization,ref:Leonard:SB,ref:YC:TG:MP:SB-and-OT}. 
Since the endpoint densities considered here are finite mixtures of nondegenerate Gaussian densities, they are smooth, strictly positive, and have finite second moments. 
\end{remark}


\subsection{Pairwise component SBs}
In the mixture setting \eqref{eq:mixture_endpoints}, each initial component \(p_i^0\) gives rise to its own uncontrolled reference law. For each \(i\in\indexsource\), define \(\refmeas^i \Let \refmeas^{p_i^0},\) i.e., \(\refmeas^i\) is the law of the uncontrolled diffusion \eqref{eq:uncontrolled} with \(\nu=p_i^0\).
For each \((i,j)\in\indexsource\times\indexsink\), define the pairwise component Schr\"{o}dinger bridge by
\begin{equation}\label{eq:pairwise_SB_def}
    \probmeas^{ij}\in \argmin_{\aset[\big]{\optmeas \in \polish(\Omega) \suchthat \Law_{\optmeas}(x_0)=p_i^0,\,
    \Law_{\optmeas}(x_1)=p_j^1}} \KL(\optmeas\|\refmeas^i).
\end{equation}
Since \(p_i^0\) and \(p_j^1\) are nondegenerate Gaussian densities, this bridge is well defined. 
We denote by \(u^{ij}\) its optimal drift and by \(\rho_t^{ij}\) its time-\(t\) marginal density. Define the KL cost \(\kappa_{ij}\Let \KL(\probmeas^{ij}\|\refmeas^i),\) and the corresponding kinetic-energy cost by \(C_{ij}\Let \eps \kappa_{ij}.\) We also have by the Girsanov identity
\begin{equation}\label{eq:pairwise_energy}
    C_{ij}=\E^{\probmeas^{ij}}\left[\int_0^1\frac{1}{2}
    \norm{u_t^{ij}(x_t)}^2\,\dd t\right],
\end{equation}
i.e., \(C_{ij}\) is the minimum kinetic energy required to steer the initial Gaussian component \(p_i^0\) to the terminal Gaussian component \(p_j^1\), relative to the uncontrolled Brownian reference initialized at \(p_i^0\).


\subsection{Closed-form pairwise Gaussian Schr\"{o}dinger bridges}\label{subsec:Pairwise:formulas}

Closed-form Gaussian Schr\"{o}dinger bridges under linear stochastic dynamics were first established by Chen, Georgiou, and Pavon; see
\cite[Proposition 4 and Theorem 8]{chen2015optimal}
and \cite{ref:YC:TG:MP:SB-and-OT}.
Propositions~\ref{prop:gaussian_endpoint_coupling}
and~\ref{prop:pairwise_gaussian_bridge_formulas} below
recall the Brownian component formulas using the notation
of this manuscript. 
These formulas will be used to
compute the pairwise costs \(C_{ij}\) in \eqref{eq:pairwise_energy}.


Fix \((i,j)\in \indexsource \times \indexsink\), and recall that
\[
p_i^0 \Let \mathcal N(m_i^0,\Sigma_i^0),\qquad p_j^1 \Let \mathcal N(m_j^1,\Sigma_j^1),
\]
with \(m_i^0,m_j^1\in\R[d]\) and \(\Sigma_i^0,\Sigma_j^1\in\mathbb S_{++}^d\). 
Let \(\probmeas^{ij}\) 
denote the pairwise Schr\"{o}dinger bridge defined in \eqref{eq:pairwise_SB_def}, and let \(\rho_t^{ij}\) and \(u_t^{ij}\) denote its time-\(t\) marginal density and optimal drift, respectively. 
We first state the Gaussian endpoint coupling associated with \(\probmeas^{ij}\).

\begin{myframe}
\begin{proposition}[Gaussian endpoint coupling for the pairwise bridge]
\label{prop:gaussian_endpoint_coupling}
For each \((i,j)\in \indexsource\times \indexsink\), the pairwise bridge \(\probmeas^{ij}\) is Gaussian. In particular, the endpoint pair \((x_0,x_1)\) under \(\probmeas^{ij}\) is jointly Gaussian:
\[
\mathrm{Law}_{\probmeas^{ij}}(x_0,x_1)=\mathcal N\!\left(\begin{bmatrix} m_i^0\\ m_j^1 \end{bmatrix},\begin{bmatrix} \Sigma_i^0 & \Sigma_{01}^{ij}\\ (\Sigma_{01}^{ij})^{\top} & \Sigma_j^1 \end{bmatrix}\right),
\]
where
\[
\Sigma_{01}^{ij}\Let (\Sigma_i^0)^{1/2}\Xi_{ij}(\Sigma_i^0)^{-1/2}-\frac{\eps}{2}I,\qquad \Xi_{ij}\Let \left((\Sigma_i^0)^{1/2}\Sigma_j^1(\Sigma_i^0)^{1/2}+\frac{\eps^2}{4}I\right)^{1/2}.
\]
\end{proposition}
\end{myframe}

Proposition \ref{prop:gaussian_endpoint_coupling}
gives the Gaussian endpoint coupling.
The next result records the corresponding marginal flow,
affine drift, and component control energy; see also
\cite[Theorem~3]{ref:SB:closed:form}
for explicit marginal and drift formulas.
A proof is included in Appendix~\ref{appen:proofs}.

\begin{myframe}
\begin{proposition}[Closed-form pairwise Gaussian bridge]
\label{prop:pairwise_gaussian_bridge_formulas}
Fix \((i,j)\in \indexsource\times \indexsink\), and let \(\Sigma_{01}^{ij}\) be as in Proposition~\ref{prop:gaussian_endpoint_coupling}.
Then, \(\Sigma_t^{ij}\in\mathbb S_{++}^d\) for all \(t\in\lcrc{0}{1}\), and the following hold.
\begin{enumerate}[label={\textup{(\ref{prop:pairwise_gaussian_bridge_formulas}-\alph*)}}, leftmargin=*, widest=b, align=left]

\item \label{prop:GSB:1} For every \(t\in \lcrc{0}{1}\), the time-\(t\) marginal under \(\probmeas^{ij}\) is Gaussian \(\rho_t^{ij}=\mathcal N(m_t^{ij},\Sigma_t^{ij}),\) with \(m_t^{ij} \Let (1-t)m_i^0+t\,m_j^1\) and \(\Sigma_t^{ij} \Let (1-t)^2\Sigma_i^0+t^2\Sigma_j^1+t(1-t)\big(\Sigma_{01}^{ij}+(\Sigma_{01}^{ij})^{\top}+\eps I\big).\)

\item \label{prop:GSB:2} For every \(t\in[0,1)\),
the bridge drift admits the affine representation \(x \mapsto u_t^{ij}(x) \Let A_t^{ij}(x-m_t^{ij})+c_{ij}\) with \(c_{ij}\Let m_j^1-m_i^0,\) where \(A_t^{ij}\Let S_t^{ij}(\Sigma_t^{ij})^{-1}\) and \(S_t^{ij}\Let (1-t)\big((\Sigma_{01}^{ij})^{\top}-\Sigma_i^0\big)+t\big(\Sigma_j^1-\Sigma_{01}^{ij}-\eps I\big).\)

\item \label{prop:GSB:3} 
The pairwise kinetic-energy cost is
\begin{align}
C_{ij} \Let \frac{1}{2}\int_{\lcrc{0}{1}\times \R[d]}\hspace{-2mm}\rho_t^{ij}(x)\,\|u_t^{ij}(x)\|^2\,\dd x\ \dd t=\frac{1}{2}\|m_j^1-m_i^0\|^2+\frac{1}{2}\int_0^1\hspace{-3mm}\operatorname{tr}\!\Big(A_t^{ij}\Sigma_t^{ij}(A_t^{ij})^{\top}\Big)\,\dd t. \nn
\end{align}
\end{enumerate}
\end{proposition}
\end{myframe}


\subsection{Lifting the mixture problem}

We first define a few objects needed in the sequel. 
A matrix \(\pi \Let (\pi_{ij})_{i\in \indexsource,j\in \indexsink}\) is a \emph{component coupling} for the mixture weights \((\alpha^0,\alpha^1)\) if
\begin{equation}\label{eq:pi_constraints}
\pi_{ij}\ge 0, \quad \sum_{j\in\indexsink}\pi_{ij}=\alpha_i^0 \;\;\text{for all } \, i\in\indexsource \text{  and  } \sum_{i\in\indexsource}\pi_{ij}=\alpha_j^1 \;\;\text{for all }\, j\in\indexsink.
\end{equation}
We denote the set of all such couplings by \(\Pi(\alpha^0,\alpha^1)\). Note that, \(\pi_{ij}\) is the fraction of total mass assigned to start from initial component \(i\) and end at terminal component \(j\).
The constraints \eqref{eq:pi_constraints} ensure that, aggregating over all terminal components, the total mass leaving component \(i\) equals \(\alpha_i^0\), and aggregating over all initial components, the total mass arriving at component \(j\) equals \(\alpha_j^1\).

Define a discrete label space \(\labelspace \Let \indexsource \times \indexsink\). Let \(Z\) be a \(\mathcal Z\)-valued random variable with \(\probmeas\bigl(Z=(i,j)\bigr)=\pi_{ij}.\) 
We lift the state space from trajectories \(\Omega\) to the product space
\begin{align}\label{eq:lifted}
\liftedspace \Let \mathcal Z\times \Omega = \bigl( \aset[]{1,\ldots,N_1} \times \aset[]{1,\ldots,N_2} \bigr) \times \pathsp\bigl(\lcrc{0}{1};\R[d]\bigr),\nn
\end{align}
and equip \(\liftedspace\) with \(\Borelsigalg(\liftedspace)\).
Observe that, a point in \(\liftedspace\) is a pair \((z,\omega)\) with \(z=(i,j)\) and \(\lcrc{0}{1} \ni t \mapsto \omega(t) \in \R[d]\) a continuous trajectory. We now define one of the main objects: For \(\pi\in\Pi(\alpha^0,\alpha^1)\), define a probability measure on \(\liftedspace\) by
\begin{equation}\label{eq:barP_def}
\liftedmeas \Let \sum_{i\in\indexsource}\sum_{j\in\indexsink}
\pi_{ij}\,\delta_{(i,j)}\otimes \probmeas^{ij} \in \polish(\liftedspace).
\end{equation}
\textcolor{black}{In \eqref{eq:barP_def}, \(\delta_{(i,j)}\) denotes the Dirac measure on the discrete label space \(\mathcal{Z} \Let \indexsource \times \indexsink\) concentrated at the label \((i,j)\). Thus, \(\delta_{(i,j)}\otimes \probmeas^{ij}\) is the product probability measure on \(\bigl(\mathcal{Z}\times\Omega,\Borelsigalg(\mathcal{Z})\otimes\Borelsigalg(\Omega)\bigr)\) satisfying
\begin{align}
A\times B \mapsto (\delta_{(i,j)}\otimes \probmeas^{ij})(A\times B) \Let \delta_{(i,j)}(A)\probmeas^{ij}(B) \quad\text{for all } A \in \Borelsigalg(\mathcal{Z}) \text{ and } B \in \Borelsigalg(\Omega), \nn
\end{align}
where \(\delta_{(i,j)}(A)=1\) if \((i,j)\in A\), and \(\delta_{(i,j)}(A)=0\) otherwise.} Consequently, if \(x_{0:1} \Let \aset[]{t \mapsto x_t \suchthat t \in \lcrc{0}{1}}\), to sample \((Z,x_{0:1})\sim \bar{\probmeas}^{\pi}\) one needs to follow the two-step procedure: (i) draw \(Z=(i,j)\) with probability \(\pi_{ij}\); (ii) conditional on \(Z=(i,j)\), draw an entire trajectory \(x_{0:1}\) from the pairwise bridge law \(\probmeas^{ij}\). 
Our first observation is that the endpoint constraints are satisfied by construction.

\begin{myframe}
\begin{proposition}\label{prop:endpoints_auto}
Let \(\pi\in\Pi(\alpha^0,\alpha^1)\) and let \(\rho_t^\pi(\cdot)\) denote the time-\(t\) marginal density of \(x_t\) under \(\liftedmeas\).
Then,
\begin{equation}\label{eq:marginals_mix}
x \mapsto \rho_t^\pi(x) \Let \sum_{i\in\indexsource}\sum_{j\in\indexsink}\pi_{ij}\,\rho_t^{ij}(x),
\end{equation}
and, in particular, \(\rho_0^\pi=\rho_0\) and \(\rho_1^\pi=\rho_1.\)
\end{proposition}
\end{myframe}

\begin{proof}
By Definition~\ref{eq:barP_def}, conditioning on \(Z=(i,j)\) yields the pairwise law \(\probmeas^{ij}\). 
Therefore, the law of \(x_t\) under \(\bar{\probmeas}^{\pi}\) is the mixture of the laws under \(\probmeas^{ij}\) with weights \(\pi_{ij}\), giving \ref{eq:marginals_mix}.
At \(t=0\), \(\rho_0^{ij}=p_i^0\), hence
\begin{align}
\rho_0^\pi=\sum_{(i,j) \in \indexsource \times \indexsink}\pi_{ij}p_i^0=\sum_{i\in \indexsource}\Big(\sum_{j\in \indexsink}\pi_{ij}\Big)p_i^0=\sum_{i \in \indexsource}\alpha_i^0 p_i^0=\rho_0. \nn
\end{align}
Similarly, at \(t=1\), \(\rho_1^{ij}=p_j^1\), hence,
\begin{align}
\rho_1^\pi=\sum_{(i,j) \in \indexsource \times \indexsink}\pi_{ij}p_j^1=\sum_{j \in \indexsink}\Big(\sum_{i\in \indexsource}\pi_{ij}\Big)p_j^1=\sum_{j\in \indexsink}\alpha_j^1 p_j^1=\rho_1. \nn    
\end{align}
Thus, the lifted law \(\bar{\probmeas}^{\pi}\) satisfies the endpoints by construction. 
\end{proof}
We now endow the lifted construction with an entropic optimality criterion, which selects the component coupling \(\pi\) by comparing \(\liftedmeas\) against a labeled reference law on \(\liftedspace\).


\subsubsection{Augmented entropic formulation}

Let \(\pi,\eta \in \Pi(\alpha^0,\alpha^1)\) with \(\eta_{ij}>0\). Define the standard KL divergence \(\KL(\pi\|\eta)\Let\sum_{i,j}\pi_{ij}\log\!\left(\frac{\pi_{ij}}{\eta_{ij}}\right).\)
We now construct an augmented reference law in the lifted space \(\liftedspace\). Fix a strictly positive \emph{prior coupling} \(\eta\in\Pi(\alpha^0,\alpha^1)\) with \(\eta_{ij}>0\).
Define the augmented (uncontrolled) measure
\begin{equation}\label{eq:barR_def}
\augmentmeas \Let \sum_{i\in\indexsource}\sum_{j\in\indexsink}
\eta_{ij}\,\delta_{(i,j)}\otimes \refmeas^i \in \polish(\liftedspace).
\end{equation}
The coupling \(\eta\) plays the role of a prior component assignment before accounting for the pairwise bridge costs. We now come to one of our main results.

Define the admissible lifted class
\begin{align}\label{eq:feas:set}
\admclass \Let \left\{ \bar{\optmeas} \Let \sum_{i,j}\pi_{ij}\delta_{(i,j)}\otimes \optmeas^{ij} \in \polish(\liftedspace)  \;\middle\vert\;  
\begin{array}{@{}l@{}}
\pi\in\Pi(\alpha^0,\alpha^1), \,\optmeas^{ij} \in \polish(\Omega),\\
\mathrm{Law}_{\optmeas^{ij}}(x_0)=p_i^0,\;
\mathrm{Law}_{\optmeas^{ij}}(x_1)=p_j^1.
\end{array}
\right\}. 
\end{align}

\begin{myframe}
\begin{theorem}[An exact lifted reduction]\label{thm:finite_dim_reduction}
Let \(\eta\in\Pi(\alpha^0,\alpha^1)\) satisfy \(\eta_{ij}>0\) for all
\((i,j)\in\indexsource\times\indexsink\), and define \(\bar{\refmeas}^{\eta}\)
by \eqref{eq:barR_def}. 
Consider the lifted problem
\begin{align}
    \inf_{\bar{\optmeas}\in \admclass}
    \KL\bigl(\bar{\optmeas}\|\bar{\refmeas}^{\eta}\bigr).
\end{align}
Then, we have the following finite-dimensional form
\begin{align}\label{eq:lifted:KL}
    \inf_{\bar{\optmeas}\in \admclass}
    \KL\bigl(\bar{\optmeas}\|\bar{\refmeas}^{\eta}\bigr)
    = \inf_{\pi\in\Pi(\alpha^0,\alpha^1)}
    \aset[\Big]{  \sum_{i,j}\pi_{ij}\kappa_{ij} +    \KL(\pi\|\eta)      }.
\end{align}
Equivalently, in the kinetic-energy scaling representation
\begin{align}\label{eq:kinetic:form}
    \inf_{\bar{\optmeas}\in \admclass}
    \eps\KL\bigl(\bar{\optmeas}\|\bar{\refmeas}^{\eta}\bigr)
    = \inf_{\pi\in\Pi(\alpha^0,\alpha^1)}
    \aset[\Big]{\sum_{i,j}\pi_{ij}C_{ij}+\eps\KL(\pi\|\eta)}.
\end{align}
Moreover, for every fixed \(\pi\in\Pi(\alpha^0,\alpha^1)\), the optimal
conditional laws on every active pair \((i,j)\) with \(\pi_{ij}>0\) are the pairwise bridges \(\optmeas^{ij}=\probmeas^{ij}\). Hence, for any optimal coupling \(\pi^\star\), an optimal lifted law is given by
\begin{align}
    \bar{\probmeas}^{\pi^\star} =
    \sum_{i,j}\pi^\star_{ij}\delta_{(i,j)}\otimes \probmeas^{ij}. \nn
\end{align}
If, in addition, \(\pi^\star_{ij}>0\) for all \((i,j)\), then the optimal
conditional laws are unique for all component pairs.
\end{theorem}
\end{myframe}

\begin{proof}
Let
\[
\bar{\optmeas}=\sum_{i,j}\pi_{ij}\delta_{(i,j)}\otimes \optmeas^{ij}\in\admclass
\]
be arbitrary. 
We first prove the inequality ``\(\geq\)'' in \eqref{eq:lifted:KL}.
If \(\KL\bigl(\bar{\optmeas}\|\bar{\refmeas}^{\eta}\bigr)=+\infty,\)
then the desired lower bound is immediate. 
Hence, assume that \(\KL\bigl(\bar{\optmeas}\|\bar{\refmeas}^{\eta}\bigr)<+\infty.\)
In particular, \(\bar{\optmeas}\ll\bar{\refmeas}^{\eta}\). 
Since \(\bar{\refmeas}^{\eta}=\sum_{i,j}\eta_{ij}\delta_{(i,j)}\otimes \mathsf R^i\) and \(\eta_{ij}>0\) for every \((i,j)\), it follows that, for every active pair \((i,j)\) with \(\pi_{ij}>0\), one has
\(\optmeas^{ij}\ll \mathsf R^i.\) Indeed, if \(B\in\borel(\Omega)\) satisfies \(\mathsf R^i(B)=0\), then \(\bar{\refmeas}^{\eta}\bigl(\{(i,j)\}\times B\bigr)=\eta_{ij}\mathsf R^i(B)=0.\) 
Since \(\bar{\optmeas}\ll\bar{\refmeas}^{\eta}\), this implies \(0=
\bar{\optmeas}\bigl(\{(i,j)\}\times B\bigr)=\pi_{ij}\optmeas^{ij}(B).\) Thus, if \(\pi_{ij}>0\), then \(\optmeas^{ij}(B)=0\), proving \(\optmeas^{ij}\ll \mathsf R^i\).

Consequently, the standard chain rule for relative 
entropy~\cite[Theorem B.2.1]{ref:Dupuis:LargeDev}  applies on the disjoint label-slice decomposition of \(\bar\Omega=\mathcal Z\times\Omega\), and gives
\begin{align}\label{KL:bound}
  \KL\bigl(\bar{\optmeas}\|\bar{\refmeas}^{\eta}\bigr)=
    \KL(\pi\|\eta)+\sum_{i,j}\pi_{ij}\KL(\optmeas^{ij}\|\mathsf{R}^i).   
\end{align}
For each pair \((i,j)\), admissibility gives \(\Law_{\optmeas^{ij}}(x_0)=p_i^0\) and \(\Law_{\optmeas^{ij}}(x_1)=p_j^1.\) Hence, by the definition of the pairwise Schr\"{o}dinger bridge \(\probmeas^{ij}\), \(\KL(\optmeas^{ij}\|\mathsf R^i)\ge
\KL(\probmeas^{ij}\|\mathsf R^i)=\kappa_{ij}.\) 
Therefore, using \eqref{KL:bound}, we have
\(\KL\bigl(\bar{\optmeas}\|\bar{\refmeas}^{\eta}\bigr)\ge\KL(\pi\|\eta)+\sum_{i,j}\pi_{ij}\kappa_{ij}.\)
Since this holds for every \(\bar{\optmeas}\in\admclass\), we obtain
\begin{align}\label{eq:ineq:side:1}
    \inf_{\bar{\optmeas}\in\admclass}
    \KL\bigl(\bar{\optmeas}\|\bar{\refmeas}^{\eta}\bigr)
    \ge
    \inf_{\pi\in\Pi(\alpha^0,\alpha^1)}
    \aset[\Big]{\KL(\pi\|\eta)+\sum_{i,j}\pi_{ij}\kappa_{ij}}.
\end{align}

We now prove the reverse inequality. 
Let \(\pi\in\Pi(\alpha^0,\alpha^1)\) be arbitrary, and recall the definition of \(\liftedmeas\) from \eqref{eq:barP_def}. Since each \(\probmeas^{ij}\) satisfies \(\Law_{\probmeas^{ij}}(x_0)=p_i^0\) and \(\Law_{\probmeas^{ij}}(x_1)=p_j^1,\) we have \(\liftedmeas\in\admclass\).
Applying the preceding KL decomposition \eqref{KL:bound} with \(\optmeas^{ij}=\probmeas^{ij}\) gives \(\KL\bigl(\bar{\probmeas}^{\pi}\|\bar{\refmeas}^{\eta}\bigr)=
 \KL(\pi\|\eta)+\sum_{i,j}\pi_{ij}\kappa_{ij}.\) 
 Taking the infimum over \(\pi\in\Pi(\alpha^0,\alpha^1)\), we obtain
\begin{align}\label{eq:ineq:side:2}
    \inf_{\bar{\optmeas}\in\admclass}
    \KL\bigl(\bar{\optmeas}\|\bar{\refmeas}^{\eta}\bigr)
    \le
    \inf_{\pi\in\Pi(\alpha^0,\alpha^1)}
    \aset[\Big]{\KL(\pi\|\eta)+\sum_{i,j}\pi_{ij}\kappa_{ij}}.
\end{align}
Combining inequalities \eqref{eq:ineq:side:1} and \eqref{eq:ineq:side:2} yields
\[
    \inf_{\bar{\optmeas}\in\admclass}
    \KL\bigl(\bar{\optmeas}\|\bar{\refmeas}^{\eta}\bigr)
    =
    \inf_{\pi\in\Pi(\alpha^0,\alpha^1)}
    \aset[\Big]{\KL(\pi\|\eta)+\sum_{i,j}\pi_{ij}\kappa_{ij}}.
\]
Multiplying both sides by \(\eps\), and using \(C_{ij}=\eps\kappa_{ij}\), yields \eqref{eq:kinetic:form}.

Finally, fix \(\pi\in\Pi(\alpha^0,\alpha^1)\).
The term \(\KL(\pi\|\eta)\) is fixed, and the only part depending on the conditional path laws is \(\sum_{i,j}\pi_{ij}\KL(\optmeas^{ij}\|\mathsf R^i).\) 
For every active pair \((i,j)\) with \(\pi_{ij}>0\), the conditional path law \(\optmeas^{ij}\) must minimize the pairwise problem
\[
    \inf \aset[\big]{\KL(\optmeas\|\mathsf R^i) \suchthat \optmeas \in \polish(\Omega),\;
    \Law_{\optmeas}(x_0)=p_i^0,\;
    \Law_{\optmeas}(x_1)=p_j^1},
\]
where \(\optmeas\) is a generic path-space optimization variable\footnote{By the definition of the pairwise component Schr\"{o}dinger bridge in \eqref{eq:pairwise_SB_def}, the minimizer of this problem is precisely \(\probmeas^{ij}\).
Therefore, on every active pair, the optimal conditional law is \(\optmeas^{ij}=\probmeas^{ij}\).}

This pairwise minimization problem has a unique solution. Indeed, its admissible set is convex, and \(\optmeas\mapsto \KL(\optmeas\|\mathsf{R}^i)\) is strictly convex on the set of measures absolutely continuous with respect to \(\mathsf{R}^i\).
Since the Gaussian bridge \(\probmeas^{ij}\) has finite entropy, every minimizer has finite entropy and hence is absolutely continuous with respect to \(\mathsf{R}^i\). 
Therefore, two distinct minimizers cannot exist. The unique minimizer is precisely the pairwise Schrödinger bridge \(\probmeas^{ij}\). Hence, for any optimal coupling \(\pi^\star\), the optimal lifted law is
\[
    \bar{\probmeas}^{\pi^\star}= \sum_{i,j}\pi^\star_{ij}\delta_{(i,j)}\otimes \probmeas^{ij}.
\]
If, in addition, \(\pi^\star_{ij}>0\) for all \((i,j)\), then every conditional law is active, and therefore the optimal conditional laws are unique for all component pairs. The proof is complete.
\end{proof}

Consequently, minimizing the lifted KL cost over \(\pi\in\Pi(\alpha^0,\alpha^1)\) is equivalently written, in the kinetic-energy scaling of Theorem \ref{thm:finite_dim_reduction}, as the finite-dimensional strictly convex program
\begin{equation}\label{eq:pi_opt}
\pi^\star \in \argmin_{\pi\in\Pi(\alpha^0,\alpha^1)} \left\{
\sum_{i,j}\pi_{ij}\,C_{ij} + \eps\,\KL(\pi\|\eta) \right\} = \argmin_{\pi\in\Pi(\alpha^0,\alpha^1)} \left\{
\sum_{i,j}\pi_{ij}\,\kappa_{ij} + \KL(\pi\|\eta) \right\}.
\end{equation}
Theorem \ref{thm:finite_dim_reduction} reduces the lifted entropy minimization problem to a finite-dimensional entropic transport problem over \(\Pi(\alpha^0,\alpha^1)\); we now record its optimality conditions and the corresponding Sinkhorn scaling form.


\subsubsection{Optimality conditions and Sinkhorn form}
We can solve for \(\pi\) in \eqref{eq:pi_opt} using Sinkhorn-type scaling. 
Recall that \(\eta\in\Pi(\alpha^0,\alpha^1)\) is the fixed strictly positive prior coupling used to define the lifted reference law \(\augmentmeas \) in \eqref{eq:barR_def}.
Define the Gibbs kernel \cite{ref:GP:MC:OT:book}
\begin{equation}\label{eq:K_matrix}
K_{ij}\Let \eta_{ij}\exp\!\left(-\frac{C_{ij}}{\eps}\right).
\end{equation}
Since \(C_{ij}=\eps\kappa_{ij}\), one has \(K_{ij}=\eta_{ij}\exp(-\kappa_{ij})\). We have the following scaling form.

\begin{myframe}
\begin{proposition}\label{prop:sinkhorn_form}
The unique optimizer \(\pi^\star\) of \eqref{eq:pi_opt} has the form
\begin{equation}\label{eq:sinkhorn_form}
\pi_{ij}^\star = a_i\,K_{ij}\,b_j,
\end{equation}
for some positive vectors \(a\in\R[N_1]\), \(b\in\R[N_2]\) chosen so that \(\sum_j\pi_{ij}^\star=\alpha_i^0,\) and \(\sum_i\pi_{ij}^\star=\alpha_j^1.\)
\end{proposition}
\end{myframe}

\begin{proof}
Consider the Lagrangian for Problem~\eqref{eq:pi_opt}
\begin{align}
 \mathcal L(\pi,\lambda,\mu)
&=\sum_{i,j}\pi_{ij}C_{ij}+\eps\sum_{i,j}\pi_{ij}\log\!\Big(\frac{\pi_{ij}}{\eta_{ij}}\Big)
+\sum_i\lambda_i\Big(\alpha_i^0-\sum_j\pi_{ij}\Big)
+\sum_j\mu_j\Big(\alpha_j^1-\sum_i\pi_{ij}\Big). \nn
\end{align}
Since \(\eta_{ij}>0\) and the retained mixture weights are strictly positive, the minimizer is strictly positive. 
Hence, the first-order stationarity condition gives
\begin{align}
0=\frac{\partial \mathcal L}{\partial \pi_{ij}}=C_{ij}+\eps\Big(\log(\pi_{ij}/\eta_{ij})+1\Big)-\lambda_i-\mu_j. \nn
\end{align}
Rearranging, we get \(\log(\pi_{ij}/\eta_{ij})=(\lambda_i+\mu_j-C_{ij})/{\eps}-1.\) 
Exponentiating and absorbing constants into \(a_i\) and \(b_j\), we obtain
\begin{align}
\pi_{ij}
=\eta_{ij}\exp\!\left(-\frac{C_{ij}}{\eps}\right)
\exp\!\left(\frac{\lambda_i}{\eps}-\frac{1}{2}\right)
\exp\!\left(\frac{\mu_j}{\eps}-\frac{1}{2}\right)
=a_iK_{ij}b_j. \nn
\end{align}
This yields \eqref{eq:sinkhorn_form}.
The marginal constraints determine the scaling vectors \(a,b\), up to the usual reciprocal scaling ambiguity \(a\mapsto ca\), \(b\mapsto c^{-1}b\). 
\end{proof}

Note that the scaling vectors \((a,b)\) can be computed by alternating normalization:
\begin{align}
a_i \leftarrow \frac{\alpha_i^0}{\sum_j K_{ij}b_j},\qquad
b_j \leftarrow \frac{\alpha_j^1}{\sum_i K_{ij}a_i}, \nn
\end{align}
iterated to convergence. See \cite[Chapter 3]{ref:GP:MC:OT:book} for more details.

The preceding results solve the lifted labeled problem, but the original SB/density-control problem is posed on the unlabeled path space; we therefore next project the lifted law by discarding/forgetting the component label and quantify the resulting entropy discrepancy.


\subsection{Quantifying the gap}\label{subsec:gap:quantification}

Theorem \ref{thm:finite_dim_reduction} is exact for the lifted problem. 
However, the original problem is unlabeled and is posed on the path space \(\Omega\), not on \(\liftedspace\). To this end, let us define
\begin{align}
\proj_{\Omega}: \mathcal{Z} \times \Omega \to \Omega \quad  \text{with }\mathcal{Z} \times \Omega \ni (z,\omega) \mapsto \proj_{\Omega}(z,\omega) \Let \omega \in \Omega, \nn   
\end{align}
which is the map that projects label-trajectory pairs onto the original path space. 
Define the path-space marginal of the lifted law associated with \(\pi\) by
\begin{align}\label{eq:proj:path:law}
\optmeas^\pi \Let (\proj_{\Omega})_{\#}\liftedmeas \in \polish(\Omega) \quad \text{where }\, (\proj_{\Omega})_{\#}:\polish(\liftedspace) \lra \polish(\Omega). 
\end{align}
Since \(\liftedmeas = \sum_{i,j}\pi_{ij}\delta_{(i,j)}\otimes \probmeas^{ij}\), we have \(\optmeas^\pi = \sum_{i,j}\pi_{ij} \probmeas^{ij}.\) 
Indeed, for any \(B\in\Borelsigalg(\Omega)\), the definition \(\optmeas^\pi=(\proj_\Omega)_{\#} \liftedmeas\) in \eqref{eq:proj:path:law} gives
\begin{align*}
\optmeas^\pi(B)=\liftedmeas(Z\times B)=\sum_{i,j}\pi_{ij}\bigl(\delta_{(i,j)}\otimes \probmeas^{ij}\bigr)(Z\times B)=\sum_{i,j}\pi_{ij}\probmeas^{ij}(B).
\end{align*}
Similarly, for any \(B\in\Borelsigalg(\Omega)\), the path-space marginal of the augmented reference law is
\begin{align}\label{proj:ref:law}
(\proj_\Omega)_{\#}\augmentmeas(B)=\augmentmeas(Z\times B)=\sum_{i,j}\eta_{ij}\mathsf{R}^i(B)=\sum_i\alpha_i^0 \mathsf{R}^i(B) \teL \mathsf{R}^{\rho_0} \in \polish(\Omega),
\end{align}
where we used \(\sum_j\eta_{ij}=\alpha_i^0\).

Thus, projecting \(\augmentmeas\) onto the original path space removes the auxiliary label and recovers the uncontrolled Brownian reference law initialized from the mixture \(\rho_0\).
Similarly, the projected law \(\optmeas^\pi\Let (\proj_\Omega)_\#\liftedmeas\) has endpoint marginals \(\rho_0\) and \(\rho_1\) by Proposition~\ref{prop:endpoints_auto}. 
Hence, \(\optmeas^\pi\) is a feasible path-space law for the original unlabeled endpoint problem. However, feasibility after projection does not imply equality of the corresponding relative entropies. 
%
%
The lifted law retains the auxiliary label \(Z=(i,j)\), whereas \(\optmeas^\pi\) is obtained after forgetting this label. 
Consequently, in general
\begin{align}
\KL(\liftedmeas \| \augmentmeas) \neq \KL(\optmeas^\pi \| \mathsf{R}^{\rho_0}). \nn
\end{align}
The discrepancy between these two quantities is precisely the conditional label-information term quantified below.

\begin{myframe}
\begin{theorem}[Projection-gap identity for the lifted bridge]
\label{thm:projection-gap-identity}
Let \(\pi,\eta\in\Pi(\alpha^0,\alpha^1)\) and \(\supp \pi \subseteq \supp \eta\).
Recall the definitions \(\liftedmeas
\Let\sum_{i,j}\pi_{ij}\delta_{(i,j)}\otimes \probmeas^{ij}\) and \(\augmentmeas\Let
\sum_{i,j}\eta_{ij}\delta_{(i,j)}\otimes \mathsf{R}^i.\) 
Let \(\optmeas^\pi \Let (\proj_\Omega)_\#\liftedmeas\) and \(\mathsf{R}^{\rho_0}\Let (\proj_\Omega)_\#\augmentmeas\) be defined as in \eqref{eq:proj:path:law} and \ref{proj:ref:law}. 
Define the conditional label-information gap
\begin{align}
\gap(\pi,\eta)\Let\E^{\optmeas^\pi}\left[\KL\bigl(\liftedmeas(Z\mid x_{0:1})
    \|\augmentmeas(Z\mid x_{0:1})\bigr)\right].    
\end{align} 
Then, \(\gap(\pi,\eta)\ge 0\), and \(\KL(\liftedmeas\|\augmentmeas)=\KL(\optmeas^\pi\|\mathsf R^{\rho_0})+\gap(\pi,\eta).\)
Consequently,
\begin{align}\label{eq:gap:main:q}
 \KL(\optmeas^\pi\|\mathsf R^{\rho_0}) =\sum_{i,j}\pi_{ij}\kappa_{ij} +\KL(\pi\|\eta)- \gap(\pi,\eta),
\end{align}
which, since \(C_{ij}=\eps\kappa_{ij}\), is equivalent to
\begin{align}
\eps\KL(\optmeas^\pi\|\mathsf R^{\rho_0}) = \sum_{i,j}\pi_{ij}C_{ij}+\eps\KL(\pi\|\eta) -\eps\gap(\pi,\eta).    
\end{align}    
\end{theorem}
\end{myframe}

\begin{proof}
Since each pairwise bridge satisfies \(\probmeas^{ij}\ll\mathsf R^i\) on every active label pair and \(\supp \pi \subseteq \supp \eta\), we have \(\liftedmeas\ll\augmentmeas.\) Applying the standard relative-entropy chain rule \cite[Theorem B.2.1]{ref:Dupuis:LargeDev} under the projection \(\proj_\Omega:\mathcal Z\times\Omega\to\Omega\) gives
\begin{align}
\KL(\liftedmeas\|\augmentmeas) =\KL((\proj_\Omega)_\#\liftedmeas\|(\proj_\Omega)_\#\augmentmeas) + \E^{\optmeas^\pi} \left[
\KL\bigl(\liftedmeas(Z\mid x_{0:1}) \| \augmentmeas(Z\mid x_{0:1})\bigr) \right]. \nn
\end{align}
By definition, \((\proj_\Omega)_\#\liftedmeas=\optmeas^\pi\) and \((\proj_\Omega)_\#\augmentmeas=\mathsf R^{\rho_0},\) and hence
\begin{align}
\KL(\liftedmeas\|\augmentmeas) = \KL(\optmeas^\pi\|\mathsf R^{\rho_0}) + \gap(\pi,\eta). \nn
\end{align}
The nonnegativity of \(\gap(\pi,\eta)\) follows from the nonnegativity of the relative entropy.

Finally, applying the label-sliced KL decomposition from \eqref{KL:bound} in Theorem~\ref{thm:finite_dim_reduction}, with \(\mathsf{Q}^{ij}=\probmeas^{ij}\), we obtain \(\KL(\liftedmeas\|\augmentmeas) = \KL(\pi\|\eta)+\sum_{i,j}\pi_{ij}\kappa_{ij}.\) Substituting this into the projection identity gives
\begin{align}\label{eq:projection_gap_energy}
 \KL(\optmeas^\pi\|\mathsf R^{\rho_0}) = \sum_{i,j}\pi_{ij}\kappa_{ij}+ \KL(\pi\|\eta)  - \gap(\pi,\eta).   
\end{align}
Multiplying by \(\eps\), and using \(C_{ij}=\eps\kappa_{ij}\), gives the kinetic-energy-scaled identity. 
\end{proof}

Thus, the lifted objective equals the unlabeled path-space objective only up to the nonnegative gap term \(\gap(\pi,\eta)\). 
In general, this term does not vanish, and therefore, one cannot identify the lifted labeled problem with the original unlabeled problem without additional structure.

\begin{remark}
Since \(\gap(\pi,\eta)\) is an expected conditional relative entropy, it is nonnegative.
Moreover, \(\gap(\pi,\eta)=0\) if and only if the two posterior distributions of the label \(Z=(i,j)\), conditioned on the full path \(x_{0:1}\) (recall that \(x_{0:1} \Let \aset[]{t \mapsto x_t \suchthat t \in \lcrc{0}{1}}\)), coincide \(\optmeas^\pi\)-almost surely. More precisely, \(\gap(\pi,\eta)=0\) if and only if, for \(\optmeas^{\pi}\)-a.e. path \(\omega\in\Omega\),
\begin{align}
\bar{\probmeas}^{\pi}\!\left(Z=(i,j)\mid x_{0:1}=\omega\right)
= \bar{\refmeas}^{\eta}\!\left(Z=(i,j)\mid x_{0:1}=\omega\right) \nn 
\end{align}
for every \((i,j)\in\indexsource\times\indexsink\). This translates, after choosing Radon--Nikodym versions of the posterior densities, to
\begin{align}
    \frac{\dd(\pi_{ij}\probmeas^{ij})}{\dd \optmeas^\pi}(\omega)
    = \frac{\dd(\eta_{ij}\refmeas^i)}{\dd \refmeas^{\rho_0}}(\omega)
    \quad \text{for all }(i,j)\text{ and for }\optmeas^\pi\text{-a.e. }\omega. \nn
\end{align}
Under \(\bar{\probmeas}^{\pi}\), the label \((i,j)\) determines both the initial component and the terminal component of the pairwise bridge. Under \(\bar{\refmeas}^{\eta}\), the path is generated by the uncontrolled law \(\refmeas^i\), which depends on \(i\) but not dynamically on \(j\). 
Therefore, the posterior label distribution inferred from a full controlled bridge path need not agree with the posterior label distribution inferred from an uncontrolled reference path. 
Thus, except for specially aligned cases (see Proposition \ref{prop:common-potential-zero-gap}), one expects \(\gap(\pi,\eta)>0.\) Consequently, one should be careful while identifying the lifted labeled objective with the original unlabeled path-space objective.
\end{remark}

\begin{myframe}
\begin{proposition}[A common path-potential condition for zero projection gap]
\label{prop:common-potential-zero-gap}
Let \(\pi,\eta\in\Pi(\alpha^0,\alpha^1)\), and suppose that they have the same support \(\mathcal{A}\Let \supp \pi=\supp \eta =\aset[]{(i,j)\in\indexsource\times\indexsink \suchthat \pi_{ij}>0}
= \aset[]{(i,j)\in\indexsource\times\indexsink \suchthat \eta_{ij}>0}.\) 
Assume that, for every \((i,j)\in \mathcal{A}\), one has \(\probmeas^{ij}\ll \mathsf{R}^i\) and that there exist a nonnegative measurable function \(h:\Omega\to [0,+\infty)\) and constants \(c_{ij}>0\), \((i,j)\in \mathcal{A}\), such that
\begin{align}\label{gap:assum:1}
\frac{\dd \probmeas^{ij}}{\dd \mathsf{R}^i}(\omega)=c_{ij}h(\omega)\quad \text{for  }\, \mathsf{R}^i\text{-a.e. }\omega \,\text{ with }(i,j)\in \mathcal{A}.
\end{align}
Assume, moreover, that the constants are compatible with the lifted prior in the sense that there exists $\lambda>0$ such that \(\pi_{ij}c_{ij}=\lambda\eta_{ij}\) for \((i,j)\in \mathcal{A}
\). 
Then, the projection gap vanishes, i.e., \(\gap(\pi,\eta)=0.\)
\end{proposition}
\end{myframe}

\begin{proof}
Recall Definitions~\eqref{eq:barP_def} and \eqref{eq:barR_def}. 
Since \(\supp\pi=\supp\eta=\mathcal{A}\) and
\(\probmeas^{ij}\ll\mathsf R^i\) for every \((i,j)\in\mathcal A\), we have \(\liftedmeas\ll \bar{\mathsf{R}}^{\eta}\). 
Also recall the projections \eqref{eq:proj:path:law} and \eqref{proj:ref:law}. 
From the projection-gap identity, we have that
\begin{align}
    \gap(\pi,\eta)=
    \E^{\optmeas^\pi}\left[
    \KL\left(\liftedmeas(Z\mid x_{0:1})
   \|\bar{\mathsf{R}}^{\eta}(Z\mid x_{0:1})\right)
    \right]. \nn 
\end{align}
It is therefore enough to show that the two posterior label distributions coincide \(\optmeas^\pi\)-a.e. Let \(M\Let \sum_{(i,j)\in\mathcal A}\eta_{ij}\mathsf R^i\). For each \(i\), write \(r_i\Let \frac{\dd\mathsf R^i}{\dd M}.\) 
Using \(\frac{\dd\probmeas^{ij}}{\dd\mathsf R^i}=c_{ij}h,\)
the posterior label distribution under the lifted controlled law is, for \(\optmeas^\pi\)-a.e. \(\omega\),
\begin{align}
    \liftedmeas(Z=(i,j)\mid \omega) =
    \frac{\pi_{ij}c_{ij}h(\omega)r_i(\omega)}
    {\sum_{(k,\ell)\in\mathcal A}\pi_{k\ell}c_{k\ell}h(\omega)r_k(\omega)}. \nn
\end{align}
Similarly, the posterior label distribution under the lifted reference law is
\begin{align}
    \bar{\mathsf R}^{\eta}(Z=(i,j)\mid \omega) = \frac{\eta_{ij}r_i(\omega)}
    {\sum_{(k,\ell)\in\mathcal A}\eta_{k\ell}r_k(\omega)}.\nn
\end{align}
By the compatibility condition \(\pi_{ij}c_{ij}=\lambda\eta_{ij}\) for \((i,j)\in\mathcal{A}\), the numerator and denominator in the first expression are both multiplied by the same factor \(\lambda h(\omega)\). 
Hence,
\begin{align}
\bar{\probmeas}^{\pi}\left(Z=(i,j)\mid x_{0:1}=\omega\right)
= \bar{\refmeas}^{\eta}\left(Z=(i,j)\mid x_{0:1}=\omega\right), \nn
\end{align}
for every \((i,j)\in\mathcal{A}\), \(\optmeas^\pi\)-a.e. \(\omega\). 
Therefore, 
the conditional relative entropy inside the definition of \(\gap(\pi,\eta)\) is zero \(\optmeas^\pi\)-a.e., and hence \(\gap(\pi,\eta)=0\).
\end{proof}

\begin{remark}[Interpretation and a non-vacuous zero-gap class]
\label{rem:common-potential-interpretation}
Proposition \ref{prop:common-potential-zero-gap} is a path-space statement.
Assumption \ref{gap:assum:1} is on the likelihood ratios of the pairwise Schr\"{o}dinger bridge laws. 
The reason the projection gap vanishes in that case is that, after observing the projected path \(\omega=x_{0:1}\), the same factor
\(h(\omega)\) appears in all active posterior label weights and cancels during normalization.
The compatibility condition \(\pi_{ij}c_{ij}=\lambda\eta_{ij}\) for \((i,j)\in\mathcal{A}\),
then makes the posterior label distribution under the lifted controlled law coincide with the posterior label distribution under the lifted reference law.
Hence, the conditional label-information term is zero.

The condition \eqref{gap:assum:1}, together with the compatibility condition \(\pi_{ij}c_{ij}=\lambda\eta_{ij}\) for \((i,j)\in \mathcal{A}\), is not vacuous.
A simple Gaussian-mixture class satisfying this condition is obtained from a common translation drift. 
For example,
let \(\eps>0\) and \(N_1=N_2=N\), let the active assignment graph be diagonal, \(\mathcal A=\aset[]{(i,i) \suchthat i=1,\ldots,N},\) and take \(\pi_{ii}=\eta_{ii}=\alpha_i^0\) for all \(i=1,\ldots,N.\)
Fix \(v\in \R[d]\), and let \(p_i^0=\mathcal N(m_i^0,\Sigma_i^0),\) and \(p_i^1=\mathcal N(m_i^0+v,\Sigma_i^0+\eps I).\)
Consider the controlled Brownian dynamics
\begin{align}
    \dd x_t=v\,\dd t+\sqrt{\eps}\,\dd w_t\quad \text{with }\,x_0\sim p_i^0. \nn
\end{align}
Its path law will be denoted by \(\probmeas^{ii}\).
Since the drift is constant, Girsanov's formula gives, with respect to the uncontrolled reference law \(\mathsf R^i\),
\begin{align}
  \frac{\dd\probmeas^{ii}}{\dd\mathsf R^i}(\omega)=\exp\left(
        \frac{1}{\eps}v^\top(x_1(\omega)-x_0(\omega)) -
        \frac{1}{2\eps}\|v\|^2 \right).   \nn
\end{align}
Therefore, Proposition~\ref{prop:common-potential-zero-gap} applies with
\begin{align}\label{eq:h:expression}
    \omega \mapsto  h(\omega) \Let\exp\left(
        \frac{1}{\eps}v^\top(x_1(\omega)-x_0(\omega)) - \frac{1}{2\eps}\|v\|^2\right)\quad\text{with }\,c_{ii}=1.
\end{align}
Since \(\pi_{ii}=\eta_{ii}\), the compatibility condition holds with \(\lambda=1\). 
Note that \(\probmeas^{ii}\) is indeed the pairwise Schr\"{o}dinger bridge from \(p_i^0\) to \(p_i^1\). 
Under the above controlled dynamics, \(x_1\sim \mathcal N(m_i^0+v,\Sigma_i^0+\eps I)=p_i^1.\) 
Moreover, for any admissible drift \(u\) steering \(p_i^0\) to \(p_i^1\), the mean displacement satisfies \(\E\left[\int_0^1 u_t(x_t)\,\dd t\right]=v.\) 
Hence, Jensen's inequality gives
\begin{align}
    \E\left[\int_0^1 \frac{1}{2}\|u_t(x_t)\|^2\,\dd t\right]
    \geq \frac{1}{2} \left\| \E\left[\int_0^1 u_t(x_t)\,\dd t\right] \right\|^2 = \frac{1}{2}\|v\|^2. \nn
\end{align}
The constant drift \(u_t\equiv v\) attains this lower bound. 
Therefore, it is optimal, and the corresponding path law is the pairwise Schr\"{o}dinger bridge \(\probmeas^{ii}\). 
Consequently, for this prescribed-assignment Gaussian-mixture transformation, \(\gap(\pi,\eta)=0.\) This example does not collapse the mixture to a single Gaussian: the source means \(m_i^0\), and covariances \(\Sigma_i^0\) may still vary with \(i\). 
In this example, Assumption \eqref{gap:assum:1} simply says that all active subpopulations undergo the same Brownian translation mechanism, so their pairwise bridge likelihood ratios share the common path-dependent factor \(h\) given in \eqref{eq:h:expression}.
\end{remark}

\begin{remark}[A canonical choice of the prior coupling]\label{rem:eta:interp}
The choice of the prior coupling \(\eta\) should respect the prescribed
marginals \(\alpha^0\) and \(\alpha^1\).
In particular, the entrywise uniform matrix \(\eta_{ij}={1}/{N_1N_2}\) belongs to \(\Pi(\alpha^0,\alpha^1)\) only in the special case where both
mixture weights are uniform, i.e., \( \alpha_i^0={1}/{N_1},\) and \(\alpha_j^1={1}/{N_2}.\) For general mixture weights, a natural neutral choice is, instead, the product coupling
\[
    \eta_{ij} \Let \alpha_i^0\alpha_j^1.
\]
This corresponds to an independent prior between the source label
\(I\in\indexsource\) and the target label \(J\in\indexsink\). 
This choice has a useful \emph{information-theoretic interpretation}. 
Since
\(\pi\in\Pi(\alpha^0,\alpha^1)\), its marginals are fixed and equal to
\(\alpha^0\) and \(\alpha^1\). 
Therefore,
\begin{align}
\KL(\pi\|\alpha^0\otimes\alpha^1)=\sum_{i,j}\pi_{ij}
    \log\left(\frac{\pi_{ij}}{\alpha_i^0\alpha_j^1}\right)
    \teL \muinf_\pi(I;J),   \nn 
\end{align}
where \(\muinf_\pi(I;J)\) denotes the mutual information between the initial component label \(I\) and the terminal component label \(J\) under the joint law \(\pi\). 
Hence, for the independent prior \(\eta=\alpha^0\otimes\alpha^1\), the finite-dimensional problem becomes
\begin{align}
\min_{\pi\in\Pi(\alpha^0,\alpha^1)} \left\{\sum_{i,j}\pi_{ij}C_{ij}+ \eps \, \muinf_\pi(I;J)\right\}. \nn 
\end{align}
Thus, the first term favors low-energy component-to-component bridges, while the second term penalizes excessive dependence between source and target labels. 
The optimizer creates a structured assignment only when the reduction in pairwise bridge cost justifies the associated information cost. 
In this sense, \(\eta=\alpha^0\otimes\alpha^1\) is a canonical default choice when no prior component-assignment information is available. \textcolor{black}{Moreover, the lifted formulation is not restricted to this neutral product prior --- it is a natural benefit of our approach. 
This feature is illustrated numerically in \S\ref{subsubsec:shape:matching}, 
Figure~\ref{fig:silhouette_eta_prior_ablation}.}
\end{remark}

The identity \eqref{eq:projection_gap_energy} above quantifies the loss incurred when the label coordinate is removed from the lifted law.
It compares the labeled law \(\liftedmeas\) on \(\mathcal{Z}\times\Omega\) with the unlabeled path law \(\optmeas^\pi\) on \(\Omega\). 
However, \(\optmeas^\pi\) is still a path-space object and, in general, retains \emph{hidden memory} of the source--target label through the observed trajectory. 
Thus, \(\optmeas^\pi\) should not automatically be identified with the law of a Markov diffusion driven by a state-feedback drift.
Following \cite{ref:GoWithTheFlow}, we next construct such a state-feedback drift at the Eulerian level and later quantify the additional passage from the generally non-Markov path law \(\optmeas^\pi\) to the Markov law generated by this feedback.


\subsection{Projecting back: regularity and feasibility}\label{subsec:project:back}
We now show that the standard posterior-averaged mixture feedback \cite{ref:GoWithTheFlow} arises canonically from the lifted labeled law. 
Thus, the following construction should not be viewed as an independent ansatz; it is the Eulerian Markov projection of \(\liftedmeas\). 
Fix \(\pi\in\Pi(\alpha^0,\alpha^1)\), and consider the lifted law \(\liftedmeas\). 
For every Borel set \(A\in \Borelsigalg(\R[d])\), the joint law of the label \(Z\) and the state \(x_t\) satisfies
\begin{align}
\liftedmeas\bigl(Z=(i,j),\,x_t\in A\bigr)=\pi_{ij}\int_A \rho_t^{ij}(x)\,\dd x. \nn
\end{align}
Consequently, after forgetting the label, the time-\(t\) state density is
\begin{align}\label{eq:proj:density}
x \mapsto \rho_t^\pi(x)\Let \sum_{i,j}\pi_{ij}\,\rho_t^{ij}(x).
\end{align}
Moreover, the conditional probability of the hidden label given the current state is
\begin{equation}\label{eq:gamma_ij}
\gamma_{ij}(t,x) \Let \liftedmeas\bigl(Z=(i,j)\mid x_t=x\bigr)=\frac{\pi_{ij}\rho_t^{ij}(x)}{\rho_t^\pi(x)}.
\end{equation}
Since the drift under the conditional bridge \(\probmeas^{ij}\) is \(u_t^{ij}\), the state-only Markov drift induced by the lifted construction is obtained by posterior averaging over the unobserved label 
\begin{equation}\label{eq:projected_drift}
x \mapsto \bar u_t^\pi(x) \Let \sum_{i,j}\gamma_{ij}(t,x)\,u_t^{ij}(x) =
\frac{\sum_{i,j}\pi_{ij}\rho_t^{ij}(x)u_t^{ij}(x)}{\rho_t^\pi(x)}.
\end{equation}
Equivalently, we have 
\[
\rho_t^\pi(x)\bar u_t^\pi(x)=\sum_{i,j}\pi_{ij}\rho_t^{ij}(x)u_t^{ij}(x).
\]
This identity is the key reason why the pair \((\rho_t^\pi,\bar u_t^\pi)\) satisfies the Fokker--Planck equation after summing the componentwise equations, which we will show later in Proposition \ref{prop:projected_feasibility_energy}. 
We also note that the objects in \eqref{eq:proj:density}--\eqref{eq:projected_drift} are algebraically identical to those employed in the Gaussian-mixture bridge constructions in \cite{ref:GoWithTheFlow}. 
The contribution here is not the formulas themselves, but their derivation from the lifted path-space law and its role in the projection analysis developed above.

\begin{myframe}
\begin{lemma}[Regularity of the projected drift]
\label{lem:projected_drift_regularity}
Let \(\pi\in\Pi(\alpha^0,\alpha^1)\), and recall the expressions \eqref{eq:proj:density}, \eqref{eq:gamma_ij}, and \eqref{eq:projected_drift}.
Then, \(\bar u^\pi\in\mathcal{U}_{\rm reg}\). 
Specifically, for every \(R>0\), there exists \(L_R(\cdot)\in L^1(0,1)\) such that
\begin{align}
\norm{\bar u_t^\pi(x)-\bar u_t^\pi(y)} \leq L_R(t)\norm{x-y}\qquad \text{with } \norm{x},\norm{y}\leq R, \nn 
\end{align}
for a.e. \(t\in\lcrc{0}{1}\), and \(a_\pi(\cdot)\in L^2(0,1)\) such that
\begin{align}
\norm{\bar u_t^\pi(x)} \leq a_\pi(t)(1+\norm{x}) \qquad \text{for }x\in\R[d], \nn
\end{align}
for a.e. \(t\in\lcrc{0}{1}\).
\end{lemma}
\end{myframe}

\begin{proof}
Recall from Proposition~\ref{prop:pairwise_gaussian_bridge_formulas} that, for each \((i,j)\in\indexsource\times\indexsink\), the density \(x\mapsto\rho_t^{ij}(x)\) is smooth and strictly positive, while the drift has the affine form
\begin{align}
x \mapsto u_t^{ij}(x)=A_t^{ij}(x-m_t^{ij})+c^{ij}
\quad \text{for }t\in [0,1). \nn
\end{align}
Moreover, by the explicit formulas for the pairwise Gaussian bridge in Section \ref{subsec:Pairwise:formulas},
%
the maps
\(t\mapsto m_t^{ij}\), \(t\mapsto S_t^{ij}\), and \(t\mapsto\Sigma_t^{ij}\) are continuous on \(\lcrc{0}{1}\), with \(\Sigma_t^{ij}\in\mathbb S_{++}^d\) for every \(t\in\lcrc{0}{1}\). 
Hence, \(t\mapsto(\Sigma_t^{ij})^{-1}\) is continuous and uniformly bounded on \(\lcrc{0}{1}\). Since \(A_t^{ij}=S_t^{ij}(\Sigma_t^{ij})^{-1}\), it follows that \(t\mapsto A_t^{ij}\) is uniformly bounded on \(\lcrc{0}{1}\).

We first prove the linear-growth estimate.
For each \((i,j)\), the affine representation gives
\begin{align}
\| {u_t^{ij}(x)} \| \leq \| A_t^{ij}\| \norm{x} + \| A_t^{ij}m_t^{ij}\| +
\| c^{ij}\|. \nn
\end{align}
Thus, setting \(a_{ij}(t) \Let \| A_t^{ij} \| +\| A_t^{ij}m_t^{ij}\| +\| c^{ij}\|,\) we have
\begin{align}
\| u_t^{ij}(x) \| \leq a_{ij}(t)(1+\norm{x}). \nn
\end{align}
By the boundedness of \(A_t^{ij}\), \(m_t^{ij}\), and \(c^{ij}\), each \(a_{ij}\) belongs to
\(L^\infty(0,1)\), hence also to \(L^2(0,1)\). Since there are only finitely many component pairs, \(a_\pi(t) \Let \max_{i,j}a_{ij}(t)\) is bounded and thus belongs to \(L^2(0,1)\). 
Using the facts that \(\gamma_{ij}^\pi(t,x)\geq 0\) and \(\sum_{i,j}\gamma_{ij}^\pi(t,x)=1\), we obtain
\begin{align}
\| \bar u_t^\pi(x) \| \leq \sum_{i,j}\gamma_{ij}^\pi(t,x)\| u_t^{ij}(x) \| \leq a_\pi(t)(1+\norm{x}). \nn 
\end{align}
This proves the linear-growth condition.

It remains to prove the local Lipschitz condition. For a fixed \(R>0\), we denote by \(\Ball_R\Let\aset[]{x\in\R[d]\suchthat \norm{x}\le R}\), the closed Euclidean ball of radius \(R\) centered at the origin. 
Observe that: 
\begin{itemize}[leftmargin=*, label=\(\circ\)]
    \item 
    Since each \(\rho_t^{ij}\) is a smooth positive Gaussian density and \(\rho_t^\pi=\sum_{i,j}\pi_{ij}\rho_t^{ij}\), the function \((t,x)\mapsto\rho_t^\pi(x)\) is continuous and strictly positive on the compact set \(\lcrc{0}{1}\times \Ball_R\). 
    Therefore, \(\rho_t^\pi\) has a positive minimum on this compact set \cite[Chapter 1]{ref:santambrogio2023course}. 
    Consequently, each posterior weight \(x \mapsto \gamma_{ij}^\pi(t,x)\) is continuously differentiable in \(x\), and \(\nabla_x\gamma_{ij}^\pi\) is uniformly bounded on \(\lcrc{0}{1}\times  \Ball_R\).
    
    \item 
    Since \(u_t^{ij}\) is affine with uniformly bounded affine coefficients, both \(u_t^{ij}\) and \(\nabla_xu_t^{ij}=A_t^{ij}\) are uniformly bounded on \(\lcrc{0}{1}\times \Ball_R\). 
    Differentiating \(\bar u_t^\pi\) with respect to \(x\), we get \[\nabla_x\bar u_t^\pi(x) = \sum_{i,j} \left[ u_t^{ij}(x)\otimes\nabla_x\gamma_{ij}^\pi(t,x) + \gamma_{ij}^\pi(t,x)A_t^{ij} \right].\]
    The right-hand side is uniformly bounded on \(\lcrc{0}{1}\times  \Ball_R\). 
    Hence, there exists a constant \(C_R<+\infty\) such that \[ \sup_{t\in\lcrc{0}{1}}\sup_{x\in \Ball_R} \norm{\nabla_x\bar u_t^\pi(x)} \leq C_R. \]
\end{itemize} 
From the mean-value theorem~\cite[Chapter 5]{ref:Rud-Analysis}
\begin{align}
\norm{\bar u_t^\pi(x)-\bar u_t^\pi(y)} \leq C_R\norm{x-y} \quad\text{with } x,y \in \Ball_R.\nn
\end{align}
Thus, we may choose \(t \mapsto L_R(t)\Let C_R\). 
Therefore, \(\bar u^\pi\) satisfies the local Lipschitz condition required in \(\mathcal U_{\rm reg}\). Combining the preceding linear-growth estimate and the local-Lipschitz estimate, we conclude that \(\bar u^\pi\in\mathcal U_{\rm reg}\). The proof is complete.
\end{proof}

\subsubsection{Feasibility of the projected drift}

Having constructed the posterior-averaged drift \(\bar u^\pi\), we state the corresponding feasibility and energy properties. 
These are direct consequences of the componentwise Fokker--Planck equations and the convexity of the kinetic energy, and are included to make the connection with Problem \ref{prob:density-control} explicit.
\begin{myframe}
\begin{proposition}[Projected feasibility and energy bound]
\label{prop:projected_feasibility_energy}
Let \(\pi\in\Pi(\alpha^0,\alpha^1)\), and suppose that the posterior-averaged drift
\(\bar u^\pi\) belongs to \(\mathcal U_{\rm reg}\). 
Then, the pair
\((\rho_t^\pi,\bar u_t^\pi)\) satisfies
\begin{align}
\partial_t\rho_t^\pi+\nabla\cdot(\rho_t^\pi\bar u_t^\pi)-\frac{\eps}{2}\Delta\rho_t^\pi=0, \nn
\end{align}
on \((0,1)\times\R[d]\). 
Moreover, \(\rho_0^\pi=\rho_0\), \(\rho_1^\pi=\rho_1.\) Futhermore, the projected Markov feedback satisfies the kinetic-energy bound
\begin{align}
    \mathcal{J}_{\rm proj}(\pi)\Let \int_0^1\int_{\R[d]}
    \frac{1}{2}\rho_t^\pi(x) \| \bar u_t^\pi(x)\|^2\,\dd x\,\dd t
    \le
    \mathcal{J}_{\rm lift}(\pi) \Let \sum_{i,j}\pi_{ij}C_{ij}. \nn 
\end{align}
\end{proposition}
\end{myframe}

\begin{proof}
By definition,
for each \((i,j)\), the pairwise bridge density \(\rho_t^{ij}\) and  corresponding drift \(u_t^{ij}\) satisfy
\begin{align}
    \partial_t\rho_t^{ij}+\nabla\cdot(\rho_t^{ij}u_t^{ij}) -\frac{\eps}{2}\Delta\rho_t^{ij}= 0. \nn
\end{align}
Multiplying by \(\pi_{ij}\) and summing over \((i,j)\), we obtain
\begin{align}
\partial_t\Big(\sum_{i,j}\pi_{ij}\rho_t^{ij}\Big)+\nabla\cdot\Big(\sum_{i,j}\pi_{ij}\rho_t^{ij}u_t^{ij}\Big)-\frac{\eps}{2}\Delta\Big(\sum_{i,j}\pi_{ij}\rho_t^{ij}\Big)
    =0. \nn
\end{align}
By definition, \(\rho_t^\pi=\sum_{i,j}\pi_{ij}\rho_t^{ij}\). 
Moreover, from~\eqref{eq:projected_drift}, \(\rho_t^\pi\bar u_t^\pi=\sum_{i,j}\pi_{ij}\rho_t^{ij}u_t^{ij}.\) 
Substituting these identities yields
\(\partial_t\rho_t^\pi+\nabla\cdot(\rho_t^\pi\bar u_t^\pi)-\frac{\eps}{2}\Delta\rho_t^\pi=0.\) 
The endpoint identities follow from Proposition~\ref{prop:endpoints_auto}. It remains to prove the energy bound. Applying Jensen's inequality on \(\bar u_t^\pi(x)\) gives
\begin{align}
\norm{\bar u_t^\pi(x)}^2\le \sum_{i,j}\gamma_{ij}^\pi(t,x)\| u_t^{ij}(x)\|^2. \nn
\end{align}
Multiplying by \(\rho_t^\pi(x)\) and using \(\rho_t^\pi(x)\gamma_{ij}^\pi(t,x) = \pi_{ij}\rho_t^{ij}(x),\) we obtain
\begin{align}
\rho_t^\pi(x)\norm{\bar u_t^\pi(x)}^2 \le \sum_{i,j}\pi_{ij}\rho_t^{ij}(x)\| u_t^{ij}(x)\|^2. \nn
\end{align}
Finally, integrating over \(\lcrc{0}{1}\times\R[d]\) gives \(\mathcal{J}_{\rm proj}(\pi)\le\mathcal{J}_{\rm lift}(\pi)\).
\end{proof}

We have the following immediate corollary.
\begin{myframe}
\begin{corollary}[Admissibility for Problem \ref{prob:density-control}]\label{cor:projected_drift_admissible}
For every \(\pi\in\Pi(\alpha^0,\alpha^1)\), the posterior-averaged drift \(\bar u^\pi\) belongs to \(\mathcal U(\rho_0,\rho_1)\). 
In particular, the admissible set in Problem~1 is nonempty.
\end{corollary}
\end{myframe}

\begin{proof}
By Lemma~\ref{lem:projected_drift_regularity}, \(\bar u^\pi\in\mathcal U_{\rm reg}\), so the SDE \(\d x_t=\bar u_t^\pi(x_t)\,\dd t+\sqrt{\eps}\,\dd w_t,\) with \(x_0\sim\rho_0\) admits a unique strong solution. 
Its one-time marginals solve the Fokker--Planck equation associated with \(\bar u^\pi\). 
By Proposition~\ref{prop:projected_feasibility_energy}, the explicitly constructed density \(\rho_t^\pi\) solves the same Fokker--Planck equation and satisfies \(\rho_0^\pi=\rho_0\). 
By uniqueness of the marginal/Fokker--Planck flow for drifts in \(\mathcal U_{\rm reg}\), the law of \(x_t\) has density \(\rho_t^\pi\). 
In particular, \(x_1\sim\rho_1\). 
Finally, Proposition~\ref{prop:projected_feasibility_energy} gives
\begin{align}
\E\left[\int_0^1\norm{\bar u_t^\pi(x_t)}^2\dd t\right]
= \int_0^1\int_{\R[d]}\rho_t^\pi(x)\norm{\bar u_t^\pi(x)}^2 \dd x \ \dd t <+\infty.  \nn
\end{align}
Hence, \(\bar u^\pi\in\mathcal U(\rho_0,\rho_1)\).
\end{proof}

Consequently, Problem \ref{prob:density-control} is feasible, and the projected construction gives an explicit finite-energy admissible feedback.

\subsection{Mean-Field Extension}

The lifted construction also provides a natural way to formulate the mean-field Schr\"{o}dinger bridge problem~\cite{ref:GR:AP:PS:MF-SB} with Gaussian-mixture endpoints.
To this end, consider the linear McKean--Vlasov dynamics
\[\dd x_t=A_t x_t\,\dd t+\bar A_t\bar x_t\,\dd t+B_tu_t(x_t)\,\dd t+D_t\,\dd w_t,\quad \text{with }\bar x_t\Let \E[x_t],\]
with the same prescribed Gaussian-mixture boundary distributions \(\rho_0\) and \(\rho_1\).
Let
\begin{align}
\hat x_0\Let \sum_{i\in \indexsource}\alpha_i^0m_i^0,\quad \hat x_1\Let \sum_{j\in \indexsink}\alpha_j^1m_j^1, \nn 
\end{align}
denote the endpoint means. 
In the unconstrained case, the mean-field term can be separated by writing \(x_t=\bar x_t+\tilde x_t,\) and \(u_t=\bar u_t+\tilde u_t,\) where \(\bar x_t \Let \E[x_t]\), \(\bar u_t \Let \E[u_t(x_t)]\), and \(\E[\tilde x_t]=\E[\tilde u_t]=0\). 
This yields the deterministic mean-steering problem
\[
\dot{\bar x}_t=(A_t+\bar A_t)\bar x_t+B_t\bar u_t\qquad \,\bar x_0=\hat x_0,\,\,\bar x_1=\hat x_1, 
\]
together with the centered linear Schr\"{o}dinger bridge
\begin{align} \label{eq:PT100}
\dd \tilde x_t=A_t\tilde x_t\,\dd t+B_t\tilde u_t(\tilde x_t)\,\dd t+D_t\,\dd w_t,
\end{align}
whose boundary distributions are the shifted Gaussian mixtures
\[
\tilde\rho_0 \Let \sum_{i\in I}\alpha_i^0\mathcal N(m_i^0-\hat x_0,\Sigma_i^0),\qquad \tilde\rho_1 \Let \sum_{j\in J}\alpha_j^1\mathcal N(m_j^1-\hat x_1,\Sigma_j^1).
\]
The lifted construction can then be applied directly to this centered problem. Namely, for each source--target component pair \((i,j)\), one solves the Gaussian Schr\"{o}dinger bridge, equivalently, the linear covariance-steering problem,
\[\mathcal N(m_i^0-\hat x_0,\Sigma_i^0)\longrightarrow \mathcal N(m_j^1-\hat x_1,\Sigma_j^1),\]
under the centered linear dynamics~\eqref{eq:PT100}.
Let the resulting pairwise law, density, drift, and kinetic cost be denoted by \(\tilde \probmeas^{ij}\), \(\tilde\rho_t^{ij}\), \(\tilde u_t^{ij}\), and \(C_{ij}\), respectively. 
Given a component coupling \(\pi\in\Pi(\alpha^0,\alpha^1)\), define the lifted centered law \(\bar{\probmeas}^{\pi}_{\mathsf{MF}}=\sum_{i\in \indexsource}\sum_{j\in \indexsink}\pi_{ij}\,\delta_{(i,j)}\otimes \tilde \probmeas^{ij}.\) 
Applying the same posterior-averaging construction used in \eqref{eq:proj:density}--\eqref{eq:projected_drift}, now to the centered variables, define the quantities
\begin{align}
\tilde\rho_t^\pi(z)\Let \sum_{i\in \indexsource}\sum_{j\in \indexsink}\pi_{ij}\tilde\rho_t^{ij}(z),\,\, 
\tilde u_t^\pi(z)\Let\sum_{i\in \indexsource}\sum_{j\in \indexsink}\tilde\gamma_{ij}(t,z)\tilde u_t^{ij}(z),\,\,
\tilde\gamma_{ij}(t,z)\Let \frac{\pi_{ij}\tilde\rho_t^{ij}(z)}{\tilde\rho_t^\pi(z)}. \nn
\end{align} 
The corresponding mean-field feedback in the original coordinates is then
\(u_t^\pi(x)=\bar u_t+\tilde u_t^\pi(x-\bar x_t),\) and the induced density is \(\rho_t^\pi(x)=\tilde\rho_t^\pi(x-\bar x_t).\) 
Consequently, the unconstrained mean-field problem can be viewed as the composition of a deterministic mean-steering problem and a lifted centered Gaussian-mixture Schr\"{o}dinger bridge. 
In this sense, the mixture-policy construction used for unconstrained mean-field Schr\"{o}dinger bridges with Gaussian-mixture endpoints is precisely the projected Markov feedback associated with a component-labeled lifted path law.


\subsection{Hidden-label memory and Markovian projection}

In this section, we distinguish the projected path law \(\optmeas^{\pi}\) in \eqref{eq:proj:path:law} and the Markov law generated by the posterior-averaged feedback \(\bar u^\pi\) in \ref{eq:projected_drift}. 
The former is obtained by forgetting the auxiliary label in the lifted law, whereas the latter is obtained by driving the Brownian diffusion with the state-feedback drift \(\bar u^\pi\). 
We denote by \(\widehat{\optmeas}^\pi \in \polish(\Omega)\) the path law induced by
\begin{align}
\dd x_t=\bar u_t^\pi(x_t)\,\dd t+\sqrt{\eps}\,\dd w_t,\qquad x_0\sim\rho_0. \nn
\end{align}
By Corollary \ref{cor:projected_drift_admissible}, this law is well defined, has finite kinetic energy, and satisfies \(\operatorname{Law}_{\widehat \optmeas^\pi}(x_t)=\rho_t^\pi(x)\,\dd x\) for \(t\in \lcrc{0}{1}\). 
On the other hand, \(\optmeas^\pi\) is generally a mixture of bridge laws with a hidden source--target label fixed for the whole trajectory. 
Therefore, even though each \(\probmeas^{ij}\) is Markov, the mixture \(\optmeas^\pi\) need not be Markov with respect to the observed state process. To this end, let
\begin{align}  \label{eqn:Fx}
\mathcal{F}_t^x \Let \sigma(x_s:0\le s\le t), 
\end{align}
denote the observed state filtration. 
Under the lifted law \(\liftedmeas\), define the pathwise posterior label probabilities \(\Gamma_{ij}^\pi(t) \Let \liftedmeas\bigl(Z=(i,j)\mid\mathcal F_t^x\bigr).\) 
This should be distinguished from \(\gamma_{ij}^\pi\) in \eqref{eq:gamma_ij}. The former conditions on the whole observed trajectory up to time \(t\), whereas the latter conditions only on the current state. 
Define the \(\mathcal{F}_t^{x}\)-adapted \cite[Chapter I, \S 1]{ref:Protter:StochInt} hidden-label drift \(b_t^\pi \Let \sum_{i,j}\Gamma_{ij}^\pi(t)u_t^{ij}(x_t).\)
This drift is generally non-Markovian, because it may depend on the observed past through the posterior \(\Gamma_{ij}^\pi(t)\).

\begin{myframe}
\begin{proposition}[Hidden-label drift and Markovian projection]\label{prop:markov:proj}
For every \(\pi\in\Pi(\alpha^0,\alpha^1)\), under the projected path law \(\optmeas^\pi=(\proj_\Omega)_\#\liftedmeas\), there exists an \((\mathcal{F}_t^x)_{t\in \lcrc{0}{1}}\)-Brownian motion \(\widetilde{w}\) such that we have the semimartingale representation
\begin{align}
    x_t= x_0+\int_0^t b_s^\pi\,\dd s+\sqrt{\eps}\,\widetilde w_t
    \quad\, t\in\lcrc{0}{1}. \nn
\end{align}
Moreover, for a.e. \(t\in\lcrc{0}{1}\), the posterior-averaged drift is
the conditional expectation
\begin{align}
    \bar u_t^\pi(x)=\E^{\optmeas^\pi}\bigl[b_t^\pi\mid x_t=x\bigr], \nn
\end{align}
for \(\rho_t^\pi\)-a.e. \(x\). Consequently, \(\bar{u}^\pi\) is the \(L^2(\dd t\otimes \optmeas^\pi)\)-orthogonal projection
of the hidden-label drift \(b^\pi\) onto the class of feedbacks depending only on the current state, that is,
\begin{align}\label{eq:var:charac}
\bar u^\pi\in\arg\min_v \E^{\optmeas^\pi}\left[\int_0^1\|b_t^\pi-v_t(x_t)\|^2\,\dd t\right], 
\end{align}
where the minimization is over measurable feedbacks \((t,x)\mapsto v_t(x)\) for which the expression is finite.
\end{proposition}
\end{myframe}

\begin{proof}
Let \(\mathcal{G}_t \Let \sigma(Z)\vee \mathcal{F}_t^x\) be the smallest sigma algebra containing both \(\sigma(Z)\) and \(\mathcal{F}_t^x\), where $F_t^x$ as in
\eqref{eqn:Fx}.
%
For notational convenience, define the labeled drift \(u_t^Z(x_t)\Let u_t^{ij}(x_t)\) on the event \(\aset[]{Z=(i,j)}.\) Under \(\liftedmeas\), conditional on \(Z=(i,j)\), the canonical process has law
\(\probmeas^{ij}\). 
Hence, in the enlarged filtration \(\mathcal G_t\), the process \(M_t \Let x_t-x_0-\int_0^t u_s^Z(x_s)\,\dd s\)
is a continuous \(\liftedmeas\)-local martingale with quadratic variation \(\langle M\rangle_t=\eps t I_d.\)
This is the standard semimartingale decomposition of a diffusion in its natural enlarged filtration; see \cite[Chapter II, \S 1 and \S 6]{ref:Protter:StochInt} and \cite[Chapter 3, \S 2--4]{ref:Shreve:Karatzas}.

Define the optional projection \cite[Chapter 7]{ref:Protter:StochInt}, \cite[Chapter 10]{ref:LiptserShiryaev:Vol1}, \cite[Chapter 2] {ref:BainCrisan:Filtering} of the labeled drift onto the observed filtration by
\begin{align}
    b_t^\pi\Let \E^{\liftedmeas}\bigl[u_t^Z(x_t)\mid \mathcal F_t^x\bigr] =  \sum_{i\in\indexsource}\sum_{j\in\indexsink}
    \liftedmeas\bigl(Z=(i,j)\mid\mathcal F_t^X\bigr)u_t^{ij}(x_t). \nn
\end{align}
We next verify that \(b_t^\pi\) is indeed the drift of the projected process with respect to \((\mathcal F_t^x)\). 
Define
\begin{align}
    \widetilde M_t \Let x_t-x_0-\int_0^t b_s^\pi\,\dd s. \nn
\end{align}
Let \(0\leq s\leq t\) and \(A\in\mathcal F_s^x\). Since \(A\in\mathcal{G}_s\), the \((\mathcal{G}_t)\)-local martingale property of \(M\), after localization if necessary, yields
\begin{align}
    \E^{\liftedmeas} \left[\mathbf 1_A\left( x_t-x_s-\int_s^t u_r^Z(x_r)\,\dd r \right) \right] =0. \nn
\end{align}
Moreover, by the definition of \(b_r^\pi\) as the conditional expectation of \(u_r^Z(x_r)\) given \(\mathcal F_r^x\), and since
\(A\in\mathcal{F}_s^x\subseteq\mathcal{F}_r^x\) for \(r\in \lcrc{s}{t}\), Fubini's theorem \cite[Theorem 2.37]{ref:FolReal-99} and the tower property of expectation give
\begin{align}
    \E^{\liftedmeas}\left[\mathbf 1_A\int_s^t u_r^Z(x_r)\,\dd r
    \right] =\E^{\liftedmeas}\left[ \mathbf 1_A\int_s^t b_r^\pi\,\dd r \right]. \nn 
\end{align}
Therefore, \(\E^{\liftedmeas}\bigl[\mathbf 1_A\bigl(\widetilde M_t-\widetilde M_s\bigr)\bigr]=0.\)
Hence, \(\widetilde{M}\) is an \((\mathcal{F}_t^x)\)-local martingale under the projected law \(\optmeas^\pi\). 
Consequently,
\begin{align}
    x_t = x_0+\int_0^t b_s^\pi\,\dd s+ \widetilde{M}_t, \nn
\end{align}
is the \((\mathcal{F}_t^x)\)-semimartingale decomposition of the projected process.

Since \(\widetilde{M}\) differs from \(x\) only by a finite-variation process, subtracting the drift term does not change quadratic variation. 
Moreover, under each conditional
law \(\probmeas^{ij}\), the coordinate process has diffusion coefficient \(\sqrt{\eps}I_d\). 
Hence, under the projected law \(\optmeas^\pi\), the matrix-valued quadratic variation of the continuous local martingale \(\widetilde M\) is \(\langle\widetilde M\rangle_t=\varepsilon t I_d\) for all \(t\in[0,1]\); see \cite[Chapter II, \S 6]{ref:Protter:StochInt}.
%
%
Therefore, by the multidimensional L\'{e}vy characterization of Brownian motion, there exists an \((\mathcal F_t^{x})\)-Brownian motion \((\widetilde w_t)_{t\in\lcrc{0}{1}]}\) under
\(\optmeas^\pi\) such that \(\widetilde{M}_t=\sqrt{\eps}\,\widetilde w_t;\)
see \cite[Chapter 3, \S 3]{ref:Shreve:Karatzas} or
\cite[Chapter IV, \S 3]{ref:RevuzYor}. 
Thus,
\begin{align}\label{eq:decomp:markov}
\dd x_t=b_t^\pi\,\dd t+\sqrt{\eps}\,\dd\widetilde w_t,
\end{align}
with respect to the observed filtration \((\mathcal F_t^{x})\).

We now identify the Markov projection of the adapted hidden-label drift \(b_t^{\pi}\) appearing in \eqref{eq:decomp:markov}.
Since \(\optmeas^\pi\) is the path-space marginal of \(\liftedmeas\), conditional expectations of path-measurable quantities may be computed under \(\liftedmeas\). Using the tower property:
\begin{align}
\E^{\optmeas^\pi}\bigl[b_t^\pi\mid x_t=x\bigr]
    =\E^{\liftedmeas}
    \left[\E^{\liftedmeas}\bigl[u_t^Z(x_t)\mid\mathcal F_t^x\bigr]
        \mid x_t=x \right] =\E^{\liftedmeas}\bigl[u_t^Z(x_t)\mid x_t=x\bigr].     \nn
\end{align} 
By the definition of the posterior weights, \(\liftedmeas\bigl(Z=(i,j)\mid x_t=x\bigr)=\gamma_{ij}^\pi(t,x).\) Therefore,
\begin{align}
\E^{\optmeas^\pi}\bigl[b_t^\pi\mid x_t=x\bigr]=
    \sum_{i\in\indexsource}\sum_{j\in\indexsink}
    \gamma_{ij}^\pi(t,x)u_t^{ij}(x)=\bar u_t^\pi(x). \nn
\end{align}

Finally, the variational characterization \eqref{eq:var:charac}
follows from the conditional-expectation projection theorem in \(L^2\): among all square-integrable feedbacks \(v_t(x_t)\), the unique \(L^2\)-orthogonal projection of \(b_t^\pi\) onto the closed subspace of functions measurable with respect to \(x_t\) is \(\E^{\optmeas^\pi}\bigl[b_t^\pi\mid x_t\bigr]=\bar u_t^\pi(x_t).\)
See, for example, \cite[Ch.~1, Sec.~2]{ref:Shreve:Karatzas}, \cite{ref:AA:VSB:Control:Diff,ref:VSB:Diff:1}. This proves the result.
\end{proof}

Proposition \ref{prop:markov:proj} clarifies the status of the posterior-averaged feedback.
As emphasized before, it is \emph{not} introduced as an independent ansatz. 
Rather, it is the \emph{Markovian projection}, in the conditional-expectation sense, of the non-Markov hidden-label drift associated with \(\optmeas^\pi\). 
Thus, the state-feedback \(\bar u^\pi\) is obtained by replacing the full path-posterior label information \(\Gamma_{ij}^\pi(t)\) with the current-state posterior \(\gamma_{ij}^\pi(t,x_t)\). We now state the final result of this article.

\begin{myframe}
\begin{theorem}[Entropy decomposition through Markovization]\label{thrm:gaps:all:together}
Recall that \(\optmeas^\pi=(\proj_\Omega)_\#\liftedmeas\), and let \(\widehat \optmeas^\pi\) be the Markov path law generated by \(\bar u^\pi\). 
Then,
\begin{align}\label{thrm3:first:identity}
\eps \KL(\optmeas^\pi\|\mathsf{R}^{\rho_0})=\eps \KL(\widehat \optmeas^\pi\|\mathsf{R}^{\rho_0})+\eps \KL(\optmeas^\pi\|\widehat \optmeas^\pi). 
\end{align}
Moreover,
\begin{align}
\eps \KL(\optmeas^\pi\|\widehat \optmeas^\pi)=\frac{1}{2}\E^{\optmeas^\pi}\left[\int_0^1\|b_t^\pi-\bar u_t^\pi(x_t)\|^2\,\dd t\right].\nn
\end{align}
Combining this with the projection-gap identity from \S\ref{subsec:gap:quantification} gives the three-level decomposition
\begin{align}\label{eq:three:lev:decomp}
\eps \KL(\liftedmeas\|\bar{\mathsf{R}}^\eta)
= \eps \KL(\widehat \optmeas^\pi\|\mathsf{R}^{\rho_0}) + \eps \KL(\optmeas^\pi\|\widehat \optmeas^\pi) + \eps \gap(\pi,\eta).
\end{align}
\end{theorem}
\end{myframe}

\begin{proof}
We first state the absolute continuity and integrability facts needed for the applications of Girsanov's theorem \cite[Chapter 3, \S 3.5]{ref:Shreve:Karatzas}. 
Since \(\optmeas^\pi=\sum_{i\in \indexsource}\sum_{j\in \indexsink}\pi_{ij}\probmeas^{ij}\) and \(\mathsf{R}^{\rho_0}=\sum_{i\in \indexsource}\alpha_i^0\mathsf{R}^i,\)
and \(\alpha_i^0>0\) for all \(i\in \indexsource\), every \(\mathsf R^{\rho_0}\)-null set is an \(\mathsf R^i\)-null set for each \(i\in \indexsource\). Since \(\probmeas^{ij}\ll \mathsf{R}^i\), it follows that \(\optmeas^\pi\ll \mathsf{R}^{\rho_0}.\) 
Furthermore, under the lifted law \(\liftedmeas\), conditioned on \(Z=(i,j)\), the path law is \(\mathsf P^{ij}\).
Hence,
\begin{align}
\E^{\liftedmeas}\left[\int_0^1\|u_t^Z(x_t)\|^2\,\dd t\right]=\sum_{i\in \indexsource}\sum_{j\in \indexsink}\pi_{ij}
\E^{\mathsf P^{ij}}\left[\int_0^1\|u_t^{ij}(x_t)\|^2\,\dd t\right]
=2\sum_{i\in \indexsource}\sum_{j\in \indexsink}\pi_{ij}C_{ij}<+\infty. 
\end{align}
By Proposition~\ref{prop:markov:proj}, \(b_t^\pi=\E^{\liftedmeas}\left[u_t^Z(x_t)\mid\mathcal{F}_t^x\right].\) Therefore, Jensen's inequality gives
\begin{align}
\E^{\optmeas^\pi}\left[\int_0^1\|b_t^\pi\|^2\,\dd t\right]\le \E^{\liftedmeas}\left[\int_0^1\|u_t^Z(x_t)\|^2\,\dd t\right]<+\infty. 
\end{align}
Thus, \(\optmeas^\pi\) has adapted drift \(b_t^\pi\) relative to the Brownian reference \(\mathsf R^{\rho_0}\), with finite energy. 
From the finite-energy form of Girsanov's theorem, we have that
\begin{align}
\eps \KL(\optmeas^\pi\|\mathsf{R}^{\rho_0})=\frac{1}{2}\E^{\optmeas^\pi}\left[\int_0^1\|b_t^\pi\|^2\,\dd t\right]. 
\end{align}

Similarly, since \(\widehat \optmeas^\pi\) is the law of the Markov diffusion \(\dd x_t=\bar u_t^\pi(x_t)\,\dd t+\sqrt{\eps}\,\dd w_t\) with \(x_0\sim\rho_0,\) and the corresponding Markov feedback has finite quadratic energy, another application of the finite-energy form of Girsanov's theorem gives
\begin{align}\label{eq:energy:bound:1}
\eps \KL(\widehat \optmeas^\pi\|\mathsf{R}^{\rho_0})=\frac{1}{2}\E^{\widehat \optmeas^\pi}\left[\int_0^1\|\bar u_t^\pi(x_t)\|^2\,\dd t\right]. 
\end{align}
Since both \(\widehat \optmeas^\pi\) and \(\optmeas^\pi\) have one-time marginals \(\rho_t^\pi\), we obtain
\begin{align}\label{eq:energy:bound:2}
\eps \KL(\widehat \optmeas^\pi\|\mathsf{R}^{\rho_0})
&=\frac{1}{2}\E^{\widehat \optmeas^\pi}\left[\int_0^1\|\bar u_t^\pi(x_t)\|^2\,\dd t\right]=\int_0^1\int_{\R[d]}\frac{1}{2}\rho_t^\pi(x)\|\bar u_t^\pi(x)\|^2\,\dd x\,\dd t\nn\\
&=\frac{1}{2}\E^{\optmeas^\pi}\left[\int_0^1\|\bar u_t^\pi(x_t)\|^2\,\dd t\right].
\end{align}
We now compare \(\optmeas^\pi\) and \(\widehat{\optmeas}^\pi\). 
Under \(\optmeas^\pi\), the adapted drift is \(b_t^\pi\), while under \(\widehat \optmeas^\pi\), the adapted drift is \(\bar u_t^\pi(x_t)\). 
The two laws have the same initial law and the same nondegenerate diffusion coefficient. Moreover, the energy identities \eqref{eq:energy:bound:1} and \eqref{eq:energy:bound:2} imply 
\begin{align}\label{eq:energy:bound:3}
\E^{\optmeas^\pi}\left[\int_0^1\|b_t^\pi-\bar u_t^\pi(x_t)\|^2\,\dd t\right]<+\infty, 
\end{align}
because \(\|b_t^\pi-\bar u_t^\pi(x_t)\|^2\le 2\|b_t^\pi\|^2+2\|\bar u_t^\pi(x_t)\|^2,\) and the two terms on the right-hand side are integrable under \(\optmeas^\pi\) by \eqref{eq:energy:bound:1} and \eqref{eq:energy:bound:2}, respectively. 
The finite-energy estimate \eqref{eq:energy:bound:3} implies, by Girsanov's theorem, that \(\optmeas^\pi \ll \widehat{\optmeas}^\pi\).
Therefore, the standard finite-energy
relative-entropy formula~\cite[Chapter 3, Section 5]{ref:Shreve:Karatzas}
for two diffusion laws with the same initial law and the same nondegenerate diffusion coefficient gives
\begin{align}
    \eps \KL(\optmeas^\pi\|\widehat \optmeas^\pi)=\frac{1}{2}\E^{\optmeas^\pi}\left[\int_0^1\|b_t^\pi-\bar u_t^\pi(x_t)\|^2\,\dd t\right]. 
\end{align}

Finally, we show the first decomposition \eqref{thrm3:first:identity}. 
By Proposition \ref{prop:markov:proj}, \(\bar u_t^\pi(x_t)=\E^{\optmeas^\pi}\bigl[b_t^\pi\mid x_t\bigr].\) Hence, the orthogonality property of conditional expectation gives, for a.e. \(t\in\lcrc{0}{1}\),
\begin{align}
\E^{\optmeas^\pi}\left[\bigl(b_t^\pi-\bar u_t^\pi(x_t)\bigr)\cdot \bar u_t^\pi(x_t)\right]=0. 
\end{align}
Consequently, we have \(\E^{\optmeas^\pi}\|b_t^\pi\|^2=\E^{\optmeas^\pi}\|\bar u_t^\pi(x_t)\|^2+\E^{\optmeas^\pi}\|b_t^\pi-\bar u_t^\pi(x_t)\|^2,\) for a.e. \(t\in \lcrc{0}{1}\). 
Integrating with respect to time and using the three relative-entropy identities \eqref{eq:energy:bound:1}--\eqref{eq:energy:bound:3}
gives
\begin{align}\label{eq:markov:1}
\eps \KL(\optmeas^\pi\|\mathsf{R}^{\rho_0})=\eps \KL(\widehat \optmeas^\pi\|\mathsf{R}^{\rho_0})+\eps \KL(\optmeas^\pi\|\widehat \optmeas^\pi).
\end{align}
Finally, the projection-gap identity from \S\ref{subsec:gap:quantification} gives
\begin{align}\label{eq:markov:2}
\eps \KL(\liftedmeas\|\bar{\mathsf R}^\eta)=\eps \KL(\optmeas^\pi\|\mathsf R^{\rho_0})+\eps\gap(\pi,\eta).
\end{align}
Using the Markovization identity \eqref{eq:markov:1} into the scaled projection-gap identity \eqref{eq:markov:2} 
yields the desired result \eqref{eq:three:lev:decomp}.
\end{proof}

\begin{remark}[When does the Markovization loss vanish?]
The term \(\KL(\optmeas^\pi\|\widehat \optmeas^\pi)\) vanishes if and only if \(\optmeas^\pi=\widehat \optmeas^\pi\). 
By the Girsanov representation above, this is equivalent to \(b_t^\pi=\bar u_t^\pi(x_t)\) for \(\dd t\otimes \optmeas^\pi\)-a.e. \((t,\omega)\). 
In other words, the hidden-label drift must already be determined by the current state, i.e., the additional label information contained in the observed past path must not improve the prediction of the current drift beyond what is already known from \(x_t\). 
The condition \(b_t^\pi=\bar u_t^\pi(x_t)\) for \(\dd t\otimes \optmeas^\pi\)-a.e. \((t,\omega)\), 
 fails, in general, because the past trajectory may contain information about which component pair \((i,j)\) was selected, and this information is not usually contained in the current state alone. 
Thus, \(\optmeas^\pi\) is generally non-Markov, whereas \(\widehat \optmeas^\pi\) is Markov by construction. 
The condition can, nevertheless, hold in useful cases. 
For example, if all active pairwise bridges have the same feedback drift, say \(x \mapsto u_t^{ij}(x) \Let u_t^\circ(x)\) for all \((i,j)\in\supp\pi,\) then both the hidden-label drift \(b_t^\pi\) and the posterior-averaged drift \(\bar u_t^\pi(x_t)\) reduce to \(u_t^\circ(x_t)\), independently of the posterior label probabilities, and hence \(\KL(\optmeas^\pi\|\widehat\optmeas^\pi)=0\). 
This situation includes, for instance, a common open-loop feedforward steering law \(x \mapsto u_t^{ij}(x) \Let v_t\) for all active pairs, identical affine covariance-steering feedbacks \(x \mapsto u_t^{ij}(x) \Let A_t x+c_t\), a rigid translation of all active Gaussian components by the same displacement vector, or a symmetric multi-agent formation-control setting in which several label assignments induce the same instantaneous averaged velocity field.
\end{remark}

\subsection{A dual certificate for the original density-control problem}\label{subsec:dual_suboptimality}

The aim of this section is to compare the projected feedback with the optimum of
Problem~\ref{prob:density-control}.
Recall the SDE \eqref{eq:dynamics}; let us define
\begin{equation}\label{eq:original_control_value}
\mathcal{J}_{\rm SB}\Let
\inf_{u\in\mathcal{U}(\rho_0,\rho_1)}\mathcal{J}(u),
\quad\text{where }\,
\mathcal{J}(u)\Let
\E^{\optmeas^u}\left[
\int_0^1\frac{1}{2}\norm{u_t(x_t)}^2\,\dd t
\right].
\end{equation}
By Corollary~\ref{cor:projected_drift_admissible}, this infimum
is finite, and \(\mathcal{J}(\bar u^\pi)=\mathcal{J}_{\rm proj}(\pi)\)
for every \(\pi\in\Pi(\alpha^0,\alpha^1)\). 
We will provide a dual certificate following \cite{ref:Leonard:SB,ref:YC:TG:MP:SB-and-OT,ref:SOC:SB:YC}. To this end, for \(s>0\), let
\begin{equation}\label{eq:dual_heat_operator}
k_s(x,y)\Let (2\pi\eps s)^{-d/2}
\exp\left(-\frac{\norm{y-x}^2}{2\eps s}\right), \quad
(\mathsf{H}_sg)(x)\Let \int_{\R[d]}k_s(x,y)g(y)\,\dd y,
\end{equation}
and set \(\mathsf{H}_0 g \Let g\).
The (heat) kernel \(k_s\) is the transition density of the Brownian
reference in \eqref{eq:uncontrolled}.

Moreover, let \(\Gamma(\rho_0,\rho_1)\) denote the set of probability measures on \(\R[d]\times\R[d]\) with marginals \(\rho_0(x)\,\dd x\) and \(\rho_1(y)\,\dd y\).
The endpoint law of the Brownian reference \(\mathsf R^{\rho_0}\), defined by \eqref{eq:uncontrolled} with \(\nu=\rho_0\), is \(r_\eps(\dd x,\dd y)\Let\rho_0(x)k_1(x,y)\,\dd x\,\dd y.\) Note that \(r_\eps\) is the endpoint projection
\((x_0,x_1)_\#\mathsf R^{\rho_0}\).
The associated static Schr\"{o}dinger problem is
\begin{equation}\label{eq:dual_static_problem}
\inf_{\lambda\in\Gamma(\rho_0,\rho_1)}\eps\KL(\lambda\|r_\eps).
\end{equation}
See \cite[Section~2, Definition~2.2]{ref:Leonard:SB} and \cite[Proposition~2.3 and Lemma~2.4]{ref:Leonard:SB}, for more details.

\begin{myframe}
\begin{theorem}[Dual certificate and suboptimality]
\label{thm:dual_suboptimality}
Consider the mapping \(g:\R[d]\to(0,+\infty)\) and define the following quantities 
\begin{align*}
h_g(t,x)\Let(\mathsf{H}_{1-t}g)(x),\,\,\Phi_g(t,x)\Let\eps\log h_g(t,x),\,\,\text{and}\,\,v_t^g(x)\Let\nabla_x\Phi_g(t,x).   
\end{align*}
Assume that \(h_g\) is finite and strictly positive on
\(\lcrc{0}{1}\times\R[d]\), belongs to
\(C(\lcrc{0}{1}\times\R[d])
\cap C^{1,2}([0,1)\times\R[d])\), and satisfies
\(\partial_t h_g+\frac{\eps}{2}\Delta_xh_g=0\)
on \([0,1)\times\R[d]\).
Moreover, assume there exists \(K_g<+\infty\) such that
\begin{equation}\label{eq:dual_growth_assumptions}
|\Phi_g(t,x)|\le K_g(1+\norm{x}^2),
\quad
\norm{v_t^g(x)}\le K_g(1+\norm{x})
\end{equation}
for all \(t\in[0,1)\) and \(x\in\R[d]\).
Define
\begin{equation}\label{eq:dual_certificate_value}
\mathcal{D}(g)\Let
\eps\int_{\R[d]}\rho_1(y)\log g(y)\,\dd y
-\eps\int_{\R[d]}\rho_0(x)
\log(\mathsf{H}_1g)(x)\,\dd x.
\end{equation}
Then both integrals are finite, and the following statements hold.

\begin{enumerate}[
label={\textup{(\ref{thm:dual_suboptimality}-\alph*)}},
leftmargin=*, widest=c, align=left]

\item\label{thm:dual_suboptimality:static}
For the static Schr\"{o}dinger problem with reference \(r_\eps\),
the potentials
\begin{equation}\label{eq:dual_endpoint_potentials}
\varphi_g(x)\Let-\eps\log(\mathsf{H}_1g)(x),
\quad \psi_g(y)\Let\eps\log g(y),
\end{equation}
give the dual value \(\mathcal{D}(g)\). Moreover, for every \(u\in\mathcal{U}(\rho_0,\rho_1)\),
\begin{equation}\label{eq:dual_exact_control_residual}
\mathcal{J}(u)-\mathcal{D}(g) =
\frac{1}{2}\E^{\optmeas^u}\left[
\int_0^1\norm{u_t(x_t)-v_t^g(x_t)}^2\,\dd t
\right].
\end{equation}
In particular, \(\mathcal{D}(g)\le\mathcal{J}_{\rm SB}\).

\item\label{thm:dual_suboptimality:bound}
For every \(\pi\in\Pi(\alpha^0,\alpha^1)\), define
\(\mathcal B_g(\pi)
\Let\mathcal{J}_{\rm proj}(\pi)-\mathcal{D}(g)\).
Then,
\begin{equation}\label{eq:dual_projected_suboptimality}
0\le
\mathcal{J}_{\rm proj}(\pi)-\mathcal{J}_{\rm SB}
\le\mathcal{B}_g(\pi),
\end{equation}
where,
\begin{equation}\label{eq:dual_projected_residual}
\mathcal{B}_g(\pi) = \frac{1}{2}\int_0^1\int_{\R[d]}
\rho_t^\pi(x) \norm{\bar u_t^\pi(x)-v_t^g(x)}^2 \,\dd x\,\dd t.
\end{equation}

\end{enumerate}
\end{theorem}
\end{myframe}

\begin{proof}
By \eqref{eq:dual_growth_assumptions} and continuity at \(t=1\), both \(\log g\) and \(\log\mathsf{H}_1g\) have at most quadratic growth. 
Since \eqref{eq:mixture_endpoints} implies finite second moments of \(\rho_0,\rho_1\), both integrals in \eqref{eq:dual_certificate_value} are finite.

Let \(\varphi,\psi:\R[d]\to\R\) be endpoint
potentials associated with the initial and terminal
marginal constraints, respectively.
For pairs for which the expression below is finite,
the dual objective associated with
\eqref{eq:dual_static_problem} is
\begin{align}\label{eq:dual_static_functional}
\mathsf{D}_s(\varphi,\psi)&\Let\int_{\R[d]}\varphi(x)\rho_0(x)\,\dd x+\int_{\R[d]}\psi(y)\rho_1(y)\,\dd y \nn\\
&\quad-\eps\log\int_{\R[d]\times\R[d]}\exp\left(\frac{\varphi(x)+\psi(y)}{\eps}\right)r_\eps(\dd x,\dd y).
\end{align}
See \cite[Section~2, formulation~(D)]{ref:Leonard:SB}.
We evaluate this objective at the pair determined by
the function \(g\) in Theorem~\ref{thm:dual_suboptimality},
namely,
\begin{equation*}
\varphi_g(x)\Let-\eps\log(\mathsf H_1g)(x),\quad \psi_g(y)\Let\eps\log g(y).
\end{equation*}
These are the potentials specified in
\eqref{eq:dual_endpoint_potentials}.
Since \(\varphi_g(x)=-\Phi_g(0,x)\) and
\(\psi_g(x)=\Phi_g(1,x)\), the quadratic-growth bound
on \(\Phi_g\) in \eqref{eq:dual_growth_assumptions},
extended to \(t=1\) by continuity, gives
\begin{equation*}
|\varphi_g(x)|\le K_g(1+\norm{x}^2),\quad |\psi_g(x)|\le K_g(1+\norm{x}^2).
\end{equation*}
The finite second moments of \(\rho_0,\rho_1\),
which follow from \eqref{eq:mixture_endpoints},
therefore ensure that both marginal integrals
in \eqref{eq:dual_static_functional} are absolutely finite.
Moreover, the definitions of \(r_\eps\)
and \(\mathsf H_1\) give
\begin{equation}\label{eq:dual_partition_normalization}
\int_{\R[d]\times\R[d]}\exp\left(\frac{\varphi_g(x)+\psi_g(y)}{\eps}\right)r_\eps(\dd x,\dd y)=\int_{\R[d]}\rho_0(x)\frac{(\mathsf H_1g)(x)}{(\mathsf H_1g)(x)}\,\dd x=1.
\end{equation}
Thus, the logarithmic term in
\eqref{eq:dual_static_functional} vanishes, and the two
remaining marginal integrals give
\(\mathsf{D}_s(\varphi_g,\psi_g)=\mathcal{D}(g)\),
with \(\mathcal{D}(g)\) defined in
\eqref{eq:dual_certificate_value}.

The backward heat equation and the definition \(\Phi_g=\eps\log h_g\) imply
\begin{equation}\label{eq:dual_hjb_identity}
\partial_t\Phi_g+\frac{\eps}{2}\Delta_x\Phi_g+\frac12\norm{\nabla_x\Phi_g}^2=0
\quad\text{on }[0,1)\times\R[d].
\end{equation}
See \cite[Section~4, Eq.~(4.3)]{ref:Leonard:SB}, with the diffusion scaling adapted to \eqref{eq:dynamics}.
Now, fix \(u\in\mathcal U(\rho_0,\rho_1)\) for Problem~\ref{prob:density-control}. All expectations below are under its controlled law \(\optmeas^u\), and \(w_t\) is the Brownian motion in \eqref{eq:dynamics}. The integral form of that equation, finite control energy, and Doob's \(L^2\) maximal inequality \cite[Chapter~1, Theorem~3.8(iv)]{ref:Shreve:Karatzas} yield
\[
\E^{\optmeas^u}\left[\sup_{0\le t\le1}\norm{x_t}^2\right]<+\infty.
\]
The growth bound \eqref{eq:dual_growth_assumptions}, consequently, gives \(\E^{\optmeas^u}\int_0^1\norm{v_t^g(x_t)}^2\,\dd t<+\infty\). 
Thus,
\[
M_T\Let\sqrt{\eps}\int_0^T v_t^g(x_t) \, \dd w_t,\qquad 0\le T<1,
\]
is a square-integrable martingale with \(\E^{\optmeas^u}[M_T]=0\); see \cite[Chapter~3, Proposition~2.10, Eqs.~(2.12)--(2.14)]{ref:Shreve:Karatzas}.

For \(T<1\), the time-dependent multidimensional It\^o formula \cite[Chapter~3, Theorem~3.6]{ref:Shreve:Karatzas}, together with \eqref{eq:dual_hjb_identity}, gives
\begin{equation}\label{eq:dual_finite_time_verification}
\Phi_g(T,x_T)-\Phi_g(0,x_0)=\int_0^T\left(u_t(x_t)\cdot v_t^g(x_t)-\frac12\norm{v_t^g(x_t)}^2\right)\dd t+M_T.
\end{equation}
Taking expectations removes \(M_T\). 
Let \(T\uparrow1\).
Continuity then implies that \(\Phi_g(T,x_T)\to\Phi_g(1,x_1)\) almost surely, and \eqref{eq:dual_growth_assumptions} bounds its absolute value by the integrable random variable \(K_g(1+\sup_{0\le t\le1}\norm{x_t}^2)\). 
The time integrand is also absolutely integrable, since
\[
\left|u_t(x_t)\cdot v_t^g(x_t)-\frac12\norm{v_t^g(x_t)}^2\right|\le\frac12\norm{u_t(x_t)}^2+\norm{v_t^g(x_t)}^2.
\]
Dominated convergence \cite[Theorem~2.24]{ref:FolReal-99} therefore yields
\begin{equation}\label{eq:dual_endpoint_verification}
\mathcal{D}(g)=\E^{\optmeas^u}[\Phi_g(1,x_1)-\Phi_g(0,x_0)]=\E^{\optmeas^u}\left[\int_0^1\left(u_t(x_t)\cdot v_t^g(x_t)-\frac12\norm{v_t^g(x_t)}^2\right)\dd t\right].
\end{equation}
The first equality follows from the endpoint constraints in Problem~\ref{prob:density-control}, the identities \(\Phi_g(1,\cdot)=\psi_g\), \(\Phi_g(0,\cdot)=-\varphi_g\), and \eqref{eq:dual_certificate_value}.
Subtracting \eqref{eq:dual_endpoint_verification} from the control cost \(\mathcal{J}(u)\) in Problem~\ref{prob:density-control} and completing the square proves \eqref{eq:dual_exact_control_residual}. 
Taking the infimum over admissible controls gives \(\mathcal{D}(g)\le\mathcal{J}_{\rm SB}\).

Finally, Corollary~\ref{cor:projected_drift_admissible} shows that \(\bar u^\pi\) in \eqref{eq:projected_drift} is admissible for Problem~\ref{prob:density-control}. Its law \(\widehat{\optmeas}^\pi\) has the time-marginal densities \eqref{eq:proj:density}, and its cost is \(\mathcal{J}_{\rm proj}(\pi)\), as defined in Proposition~\ref{prop:projected_feasibility_energy}. Applying \eqref{eq:dual_exact_control_residual} with \(u=\bar u^\pi\) and using these marginals and Tonelli's theorem \cite[Theorem~2.37(a)]{ref:FolReal-99} proves \eqref{eq:dual_projected_residual}. 
Admissibility also gives \(\mathcal{J}_{\rm SB}\le\mathcal{J}_{\rm proj}(\pi)\). Combining this with \(\mathcal{D}(g)\le\mathcal{J}_{\rm SB}\) proves \eqref{eq:dual_projected_suboptimality}.
\end{proof}

To evaluate the lower bound
\(\mathcal{D}(g)\le\mathcal{J}_{\rm SB}\) proved in
Theorem~\ref{thm:dual_suboptimality},
we construct \(g(\cdot)\) from the component bridges as follows.
Specifically, 
for each \((i,j)\), let \(g_{ij}(\cdot)>0\) be the terminal
Schr\"{o}dinger potential of \(\probmeas^{ij}\), normalized by
\(g_{ij}(m_j^1)=1\). The terminal potential \(g_{ij}(\cdot)\) satisfies \(u_t^{ij}=\eps\nabla_x\log(\mathsf{H}_{1-t}g_{ij})\);
see \cite[Section~4, Eq.~(4.2)]{ref:Leonard:SB}. For \(\pi\in\Pi(\alpha^0,\alpha^1)\), choose
\begin{equation}\label{eq:assembled_terminal_potential}
y \mapsto g_\pi(y)\Let\sum_{i,j}\pi_{ij}g_{ij}(y).
\end{equation}
The Gaussian component heat transforms are finite, strictly
positive, continuous up to \(t=1\), and smooth solutions of
the backward heat equation for \(t<1\).
The finite weighted sum of \((\mathsf{H}_{1-t}g_{ij})(x)\), which is, \(h_{g_\pi}(t,x) \Let \sum_{i,j}\pi_{ij}(\mathsf{H}_{1-t}g_{ij})(x)\)
inherits these properties. Moreover,
\begin{equation*}
x \mapsto v_t^{g_\pi}(x) \Let \frac{\sum_{i,j}\pi_{ij}(\mathsf{H}_{1-t}g_{ij})(x)u_t^{ij}(x)}{\sum_{i,j}\pi_{ij}(\mathsf{H}_{1-t}g_{ij})(x)}.
\end{equation*}
Proposition~\ref{prop:pairwise_gaussian_bridge_formulas}
and the boundedness argument in the proof of
Lemma~\ref{lem:projected_drift_regularity} give a uniform
linear-growth bound for the component drifts.
The same bound therefore holds for \(v^{g_\pi}\).
Integrating \(\nabla_x\Phi_{g_\pi}=v^{g_\pi}\) in space,
and using the boundedness of the continuous function
\(t\mapsto\Phi_{g_\pi}(t,0)\) on \(\lcrc{0}{1}\),
gives the quadratic-growth bound for \(\Phi_{g_\pi}\).
Thus, \(g_\pi\) satisfies the hypotheses of
Theorem~\ref{thm:dual_suboptimality}.
Evaluating \eqref{eq:dual_certificate_value} with this
choice gives the lower bound \(\mathcal{D}(g_\pi)\le
\mathcal{J}_{\rm SB}\), and
\eqref{eq:dual_projected_suboptimality} bounds the
suboptimality by
\(\mathcal{J}_{\rm proj}(\pi)-\mathcal{D}(g_\pi)\).

%% file: sections/4-Numerics.tex
\section{Numerical experiments}\label{sec:numexp}


In this section, we present several numerical examples to illustrate that the lifted formulation helps in three concrete ways: \highlight{(a)} it reduces the continuous multimodal steering problem to closed-form Gaussian component bridges together with the finite-dimensional coupling problem \eqref{eq:pi_opt}; \highlight{(b)} it produces an explicit projected Markov feedback through \eqref{eq:proj:density}--\eqref{eq:projected_drift}; \highlight{(c)} it exposes the component-level assignment structure through the optimized coupling \(\pi^\star\), which is not directly visible in a grid-based unlabeled Schr\"{o}dinger bridge computation. 
The lifting gives the following structured computational recipe:\footnote{All numerical simulations were performed in MATLAB R2022a on a 2021 MacBook Pro with an Apple M1 processor.}
\begin{myframe}
\vspace{-2mm}
\[
 \bigl(p_i^0,p_j^1\bigr) \mapsto 
\bigl(\probmeas^{ij},C_{ij}\bigr) \implies 
 C_{ij} \mapsto
\pi^\star \implies
 \bigl(\pi^\star,\probmeas^{ij}\bigr) \mapsto 
\bigl(\rho_t^{\pi^\star}(\cdot),\bar u_t^{\pi^\star}(\cdot)\bigr).
\]    
\end{myframe}


\subsection{A one-dimensional mixture-to-mixture example}
\label{subsec:numerics-1d}

\input{sections/NumExp1}


\subsection{Two-dimensional examples}
\label{subsec:numerics-2d}

We next consider two-dimensional Gaussian-mixture examples. We provide three illustrations. 
First, an example where both endpoint distributions are mixtures with three components. 
We will see that the optimized coupling \(\pi^\star \in \Pi(\alpha^0,\alpha^1)\) is nontrivial. 
That is, it has to decide how much mass from each source component goes to each target component, balancing pairwise costs \(C_{ij}\) against the entropy term.
The second example, 
has only one source component. 
Thus, it 
verifies our developments with intuitive numerical visualizations. 
The final example 
illustrates the developments through a shape-transformation example.

\input{sections/NumExp2}

%% file: sections/NumExp1.tex
We first consider a one-dimensional example in order to illustrate the lifted construction, the optimized component coupling, and the projected Markov drift. We take \(\eps=0.35\) and prescribe the initial and terminal Gaussian mixtures as
\[
    \rho_0 \Let 0.65\,\mathcal N(-3,0.45^2)
    + 0.35\,\mathcal N(2,0.60^2),
\]
and
\[
    \rho_1 \Let 0.35\,\mathcal N(-1.5,0.55^2)
    + 0.65\,\mathcal N(3.5,0.45^2).
\]
Thus \(N_1=N_2=2\), with mixture weights \(\alpha^0 \Let (0.65,0.35)^\top\) and \(\alpha^1 \Let (0.35,0.65)^\top\). 
The endpoint densities are shown in Figure~\ref{fig:oned_endpoint_densities}. Notice that the initial distribution has more mass in its left component, while the terminal distribution has more mass in its right component. 
Therefore, any feasible mixture-to-mixture steering must split part of the left initial component and send it to the right terminal component. For each component pair \((i,j)\), we compute the pairwise Gaussian Schrödinger bridge using Proposition~\ref{prop:pairwise_gaussian_bridge_formulas}. The resulting kinetic-energy costs \(C_{ij}\), was found to be
\[
    [C_{ij}]_{(i,j)} \Let C = \begin{pmatrix}
    1.1545 & 21.1967\\
    6.1869 & 1.2415
    \end{pmatrix}.
\]
As expected, the diagonal assignments are energetically cheap, whereas the assignment from the left initial component to the right terminal component has a high cost, and the assignment from the right initial component to the left terminal component is also relatively expensive.

We choose the prior coupling to be \(\eta^{\rm prod}_{ij}\Let \alpha_i^0\alpha_j^1\). Solving the lifted finite-dimensional problem by Sinkhorn scaling (with 100 Sinkhorn iterations) yields
\[
    \pi^\star = \begin{pmatrix} 0.35 & 0.30\\  0 & 0.35
    \end{pmatrix},
\]
where entries below \(10^{-5}\) are displayed as zero. The row and column sums satisfy \(\sum_j\pi_{ij}^\star=\alpha_i^0,\) and \(\sum_i\pi_{ij}^\star=\alpha_j^1,\) as required. The interpretation is clear: the first initial component has mass \(0.65\), but the first terminal component requires only mass \(0.35\). 
Hence, approximately \(0.35\) units of mass are assigned from component \(1\) to component \(1\), while the remaining \(0.30\) units are assigned from component \(1\) to component \(2\). 
The second initial component is assigned almost entirely to the second terminal component. 

\begin{figure}[t]
\centering
\subfloat[Endpoint densities]{%
    \includegraphics[scale=0.3]{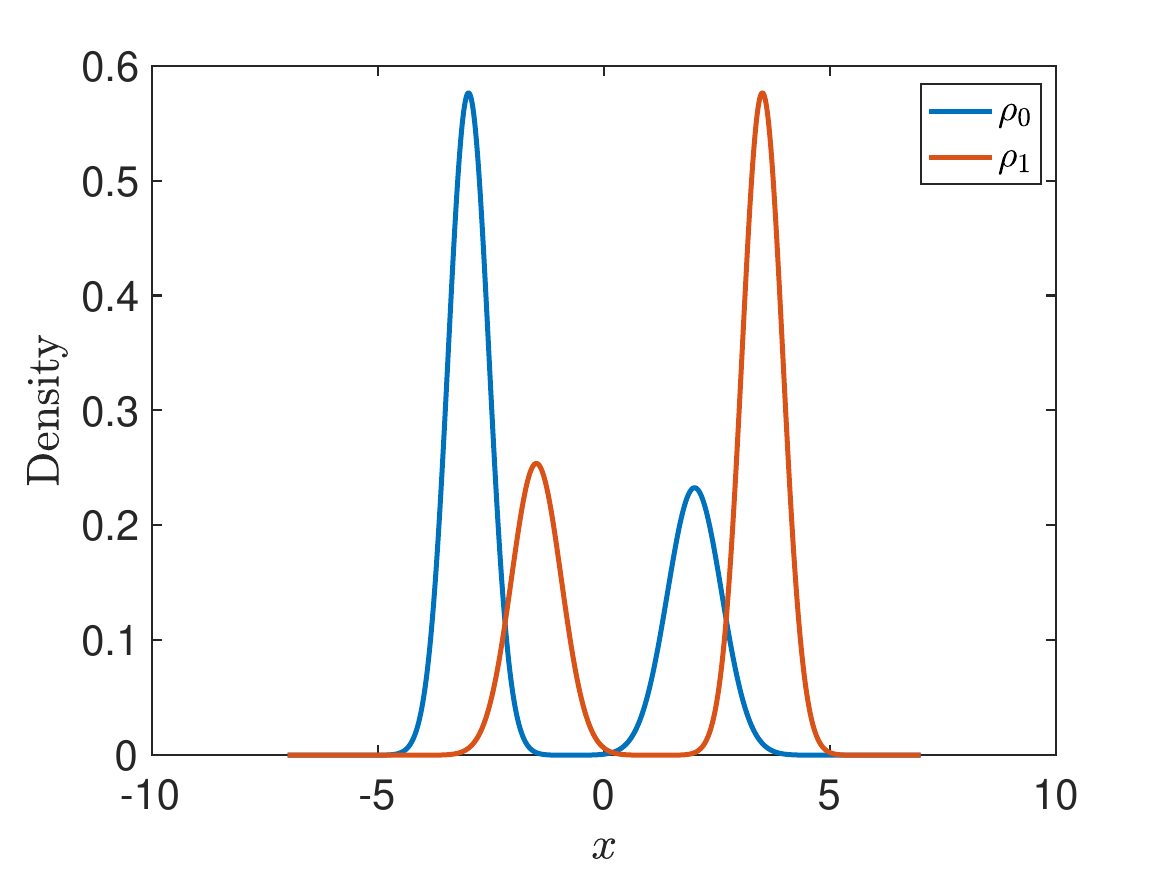}
    \label{fig:oned_endpoint_densities}}
\hspace{0.01\textwidth}
\subfloat[Sample paths]{%
    \hspace{0mm}\includegraphics[scale=0.37]{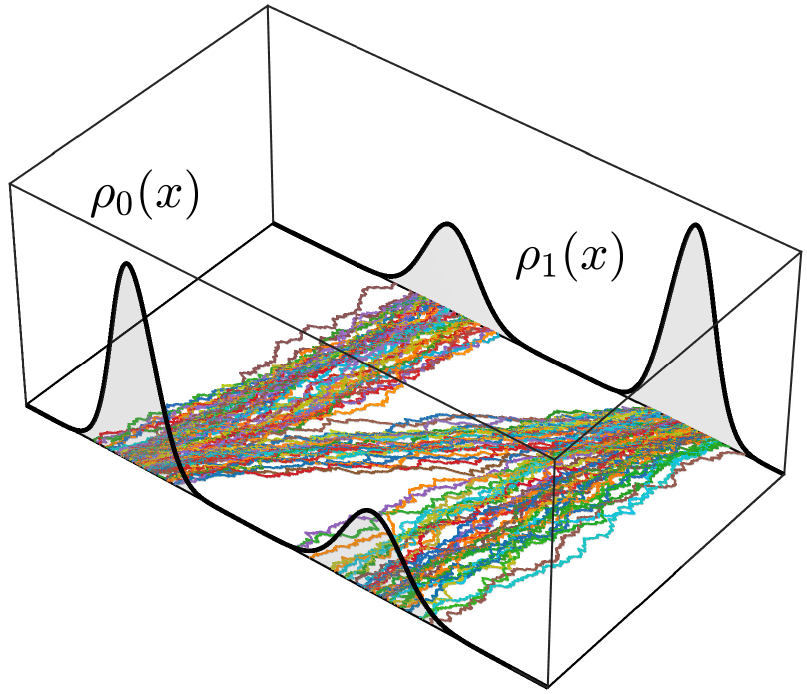}
    \label{fig:oned_sample_paths}}\\[-0.3em]
\subfloat[Projected marginal flow]{%
    \includegraphics[scale=0.3]{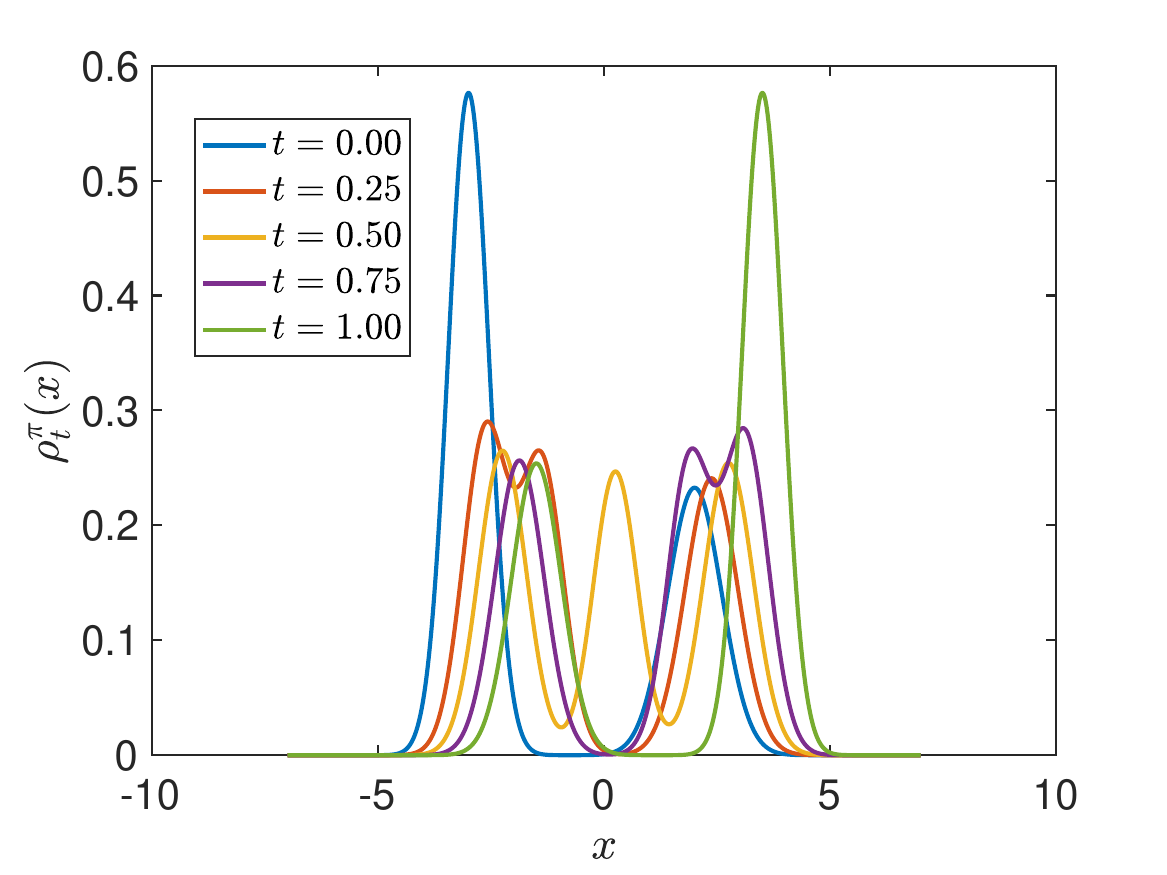}
    \label{fig:oned_density_flow}}
\hspace{0.01\textwidth}
\subfloat[Terminal validation]{%
    \includegraphics[scale=0.3]{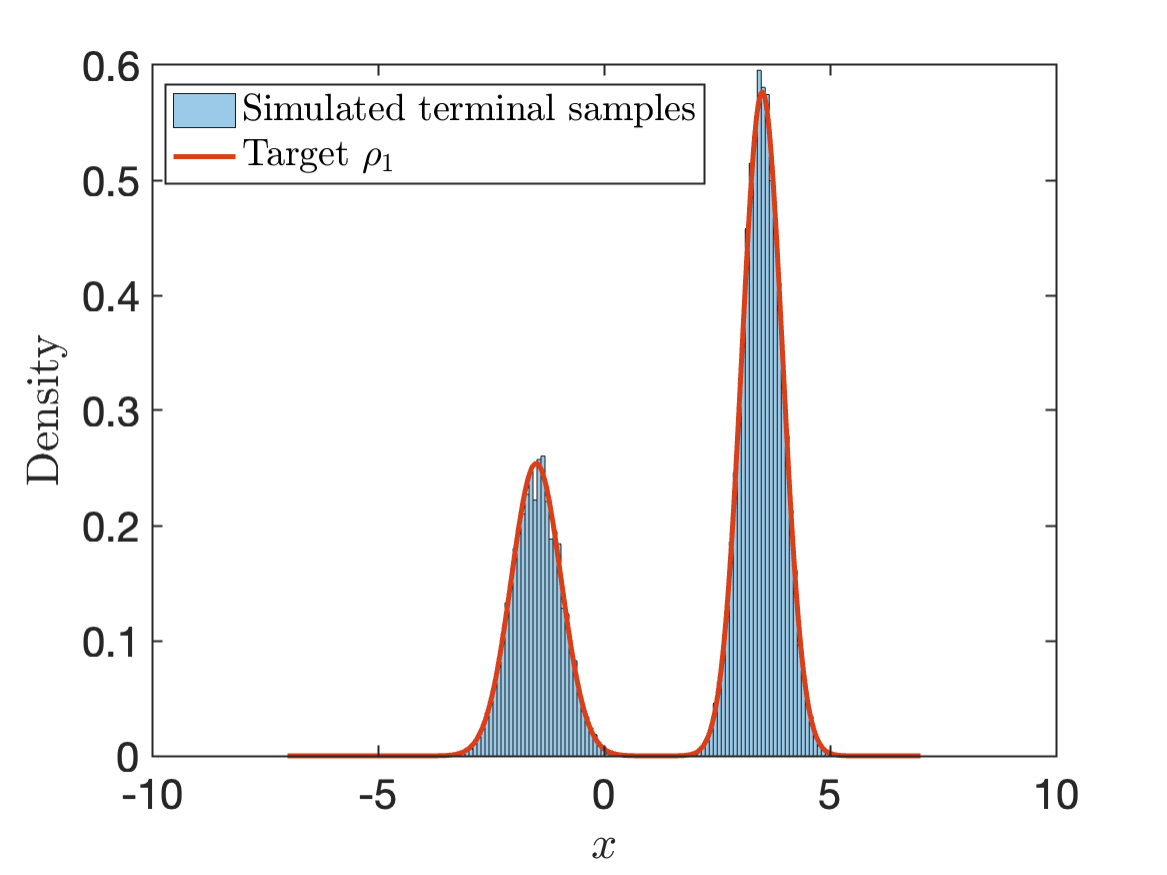}
    \label{fig:oned_terminal_validation}}
\caption{One-dimensional Gaussian-mixture Schrödinger bridge experiment.}
\label{fig:oned_results}
\end{figure}
 Given \(\pi^\star\), the lifted marginal flow is \(x \mapsto \rho_t^{\pi^\star}(x) = \sum_{i,j}\pi_{ij}^\star \rho_t^{ij}(x).\) The projected Markov drift is then defined by posterior averaging \(x \mapsto \bar u_t^{\pi^\star}(x) = \sum_{i,j}
    \gamma_{ij}^{\pi^\star}(t,x)u_t^{ij}(x),\) with \(
    \gamma_{ij}^{\pi^\star}(t,x) =  {\pi_{ij}^\star\rho_t^{ij}(x)}/
    {\rho_t^{\pi^\star}(x)}.\)
Figure~\ref{fig:oned_density_flow} shows the analytic density flow \(\rho_t^{\pi^\star}\) at several slices of times. 
The plot confirms the endpoint identities \(\rho_0^{\pi^\star}=\rho_0,\) and \(\rho_1^{\pi^\star}=\rho_1,\) which are guaranteed by Proposition \ref{prop:endpoints_auto}. 
The intermediate densities reveal the three dominant transport channels corresponding to the non-negligible entries of \(\pi^\star\): mass from component \(1\) to component \(1\), mass from component \(1\) to component \(2\), and mass from component \(2\) to component \(2\). The missing channel from component \(2\) to component \(1\) is consistent with the zero value of \(\pi_{21}^\star\).

To verify the projected drift construction, we simulated the Markov diffusion \(\dd x_t=\bar u_t^{\pi^\star}(x_t)\,\dd t+\sqrt{\eps}\,\dd w_t,\) with \(x_0\sim\rho_0\) using \(800\) Euler-Maruyama steps. The resulting sample paths are shown in 
Figure~\ref{fig:oned_sample_paths}. 
The trajectories visibly split into the same three channels identified by the optimized coupling. 
In particular, a portion of the left initial component moves to the left terminal component, another portion moves to the right terminal component, and the right initial component moves primarily to the right terminal component. 
This provides a trajectory-level visualization of the Markovized marginal flow induced by the coupling matrix \(\pi^\star\). 
Figure~\ref{fig:oned_terminal_validation} compares the terminal empirical distribution obtained by simulating the projected drift with the prescribed target density \(\rho_1\). 
The close agreement verifies numerically that the projected drift generates the desired terminal marginal. 
This is the simulation-level counterpart of the Fokker--Planck feasibility statement proved in Proposition \ref{prop:projected_feasibility_energy}.

We also compared our lifted construction with a direct numerical
approximation of the unlabeled Schr\"{o}dinger bridge for the same one-dimensional Gaussian-mixture endpoints.
The direct method discretizes the one-dimensional state space. Specifically, it first chooses a bounded computational interval \(D \Let [y_{\min},y_{\max}]\subset \R\) containing the effective support of the endpoint densities \(\rho_0\) and \(\rho_1\), and then introduces a uniform spatial grid \(\{y_k\}_{k=1}^M\subset D\).
The continuous endpoint densities \(\rho_0\) and \(\rho_1\) are then replaced by normalized probability vectors \(r_0,r_1\in\Delta_M\), obtained by evaluating/integrating the densities over the grid cells. 
With \(K_{k\ell}\) denoting the Brownian heat kernel between grid points \(y_k\) and \(y_\ell\), the method solves the finite-dimensional entropy-projection problem
\[
\min_{\Gamma\in\Pi(r_0,r_1)}\sum_{k,\ell}\Gamma_{k\ell}\log\frac{\Gamma_{k\ell}}{K_{k\ell}},
\]
using Sinkhorn scaling, following the Schr\"{o}dinger--Fortet--Sinkhorn viewpoint~\cite{ref:SOC:SB:YC}.
We found that the direct-SB energy, in the same kinetic-energy scaling used throughout the paper, is \(J_{\mathrm{direct}}=6.231.\)
or \(100\) Sinkhorn iterations, the marginal residuals were \(\norm{\gamma \mathbf 1-r_0}_1 = 4\times 10^{-13} ,\) and \(\norm{\gamma^\top \mathbf 1-r_1}_1 = 6\times 10^{-14}\), where \(r_0,r_1\) are the discretized endpoint marginals and \(\gamma\) is the computed endpoint coupling. 
This means that, among all state-space controlled diffusions steering the full mixture \(\rho_0\) to the full mixture \(\rho_1\), the minimum energy is approximately \(6.231\).

Using the projected density \(\rho_t^{\pi^\star}(\cdot)\) and the feedback \(\bar u_t^{\pi^\star}(\cdot)\), we have
\[
J_{\rm proj}\Let \int_0^1\int_{\R}\frac{1}{2}\rho_t^{\pi^\star}(x) \,| \bar{u}_t^{\pi^\star}(x)|^2
\,\dd x \, \dd t,
\]
which numerically is \(J_{\rm proj}=6.437.\) The relative gap is
\[
\frac{J_{\rm proj}-J_{\rm direct}}{J_{\rm direct}}=\frac{0.2061072060}{6.2313514143}\approx 0.03308.
\]
Thus, the advantage of our method should be understood as follows: it provides an explicit, structured, low-dimensional, interpretable, feedback-realizable construction. The direct numerical Schr\"{o}dinger bridge is a grid-based reference method. In one dimension, it is easy to implement. 
However, in dimension \(d\), if one uses \(M\) grid points per coordinate, the endpoint grid has \(M^d\) points, and the endpoint coupling matrix has roughly \(M^d\times M^d=M^{2d}\) entries. Thus, even a moderately fine two- or three-dimensional grid becomes expensive quickly. Our approach replaces this massive endpoint coupling by a component coupling of size \(N_1\times N_2\); see Table \ref{tab:metadata_ex_1} for details. 
\begin{table}[]
\centering
\caption{Numerical statistics. The computation time (in secs) is median run times over 30 repetitions.}
\begin{tblr}{l c c c}
\hline[2pt]
\SetRow{azure9}
Method  & Size & Energy & Time \\
\hline
Direct SB & \(601 \times 601 \) & 6.231 & \(4.4 \times 10^{-1}\) \\
Ours &  \(2 \times 2\) & 6.437  & \(4.4 \times 10^{-4}\) \\
\hline[2pt]
\end{tblr}
\label{tab:metadata_ex_1}
\end{table}

%% file: sections/NumExp2.tex
\subsubsection{Example 1}\label{subsub:example:1} We take \(d=2\), terminal time \(T=1\), and diffusion parameter \(\eps = 0.3.\) The initial and terminal distributions are three-component Gaussian mixtures
\[
    x \mapsto \rho_0(x) \Let \sum_{i=1}^3 \alpha_i^0\,\mathcal N(x;m_i^0,\Sigma_i^0), \quad
    x \mapsto \rho_1(x) \Let \sum_{j=1}^3 \alpha_j^1\,\mathcal N(x;m_j^1,\Sigma_j^1),
\]
with mixture weights \(\alpha^0 \Let (0.40,0.35,0.25)^\top,\alpha^1 \Let (0.25,0.45,0.30)^\top.\) The component means are \(m_1^0 \Let (-4,-2),\) \(m_2^0 \Let (-4,1.8),\) and \(m_3^0 \Let (-1,0),\) and
\(m_1^1 \Let (2.5,-2.5),m_2^1 \Let (4,0.4),\) and \(m_3^1 \Let (2.5,2.7).\) 
The covariance matrices are
\[
    \Sigma_1^0 \Let \begin{pmatrix}
    0.35 & 0.08\\
    0.08 & 0.28 \end{pmatrix},
    \qquad \Sigma_2^0 \Let \begin{pmatrix}
    0.30 & -0.06\\
    -0.06 & 0.45 \end{pmatrix},
    \qquad
    \Sigma_3^0 \Let \begin{pmatrix}
    0.48 & 0\\
    0 & 0.35 \end{pmatrix},
\]
and
\[
    \Sigma_1^1 \Let  \begin{pmatrix}
    0.40 & -0.05\\
    -0.05 & 0.30 \end{pmatrix},
    \qquad
    \Sigma_2^1 \Let  \begin{pmatrix}
    0.35 & 0.07\\
    0.07 & 0.42 \end{pmatrix},
    \qquad
    \Sigma_3^1 \Let \begin{pmatrix}
    0.32 & -0.04\\
    -0.04 & 0.36 \end{pmatrix}.
\]
For each pair \((i,j)\), we compute the pairwise Gaussian Schr\"{o}dinger bridge from the initial component \(\mathcal N(m_i^0,\Sigma_i^0)\) to the terminal component \(\mathcal N(m_j^1,\Sigma_j^1)\). The corresponding kinetic cost is 
\[
    C
    \approx
    \begin{pmatrix}
    21.31 & 34.92 & 32.24\\
    30.45 & 33.05 & 21.61\\
    9.34  & 12.66 & 9.87
    \end{pmatrix}.
\]
For this example, the third initial component has comparatively low transport cost to all terminal components, especially to the first and third terminal components. By contrast, transporting the first initial component to the second or third terminal component and the second initial component to the first or second terminal component are substantially more expensive.

\begin{figure}[h]
\centering
\subfloat{%
    \includegraphics[scale=0.24]{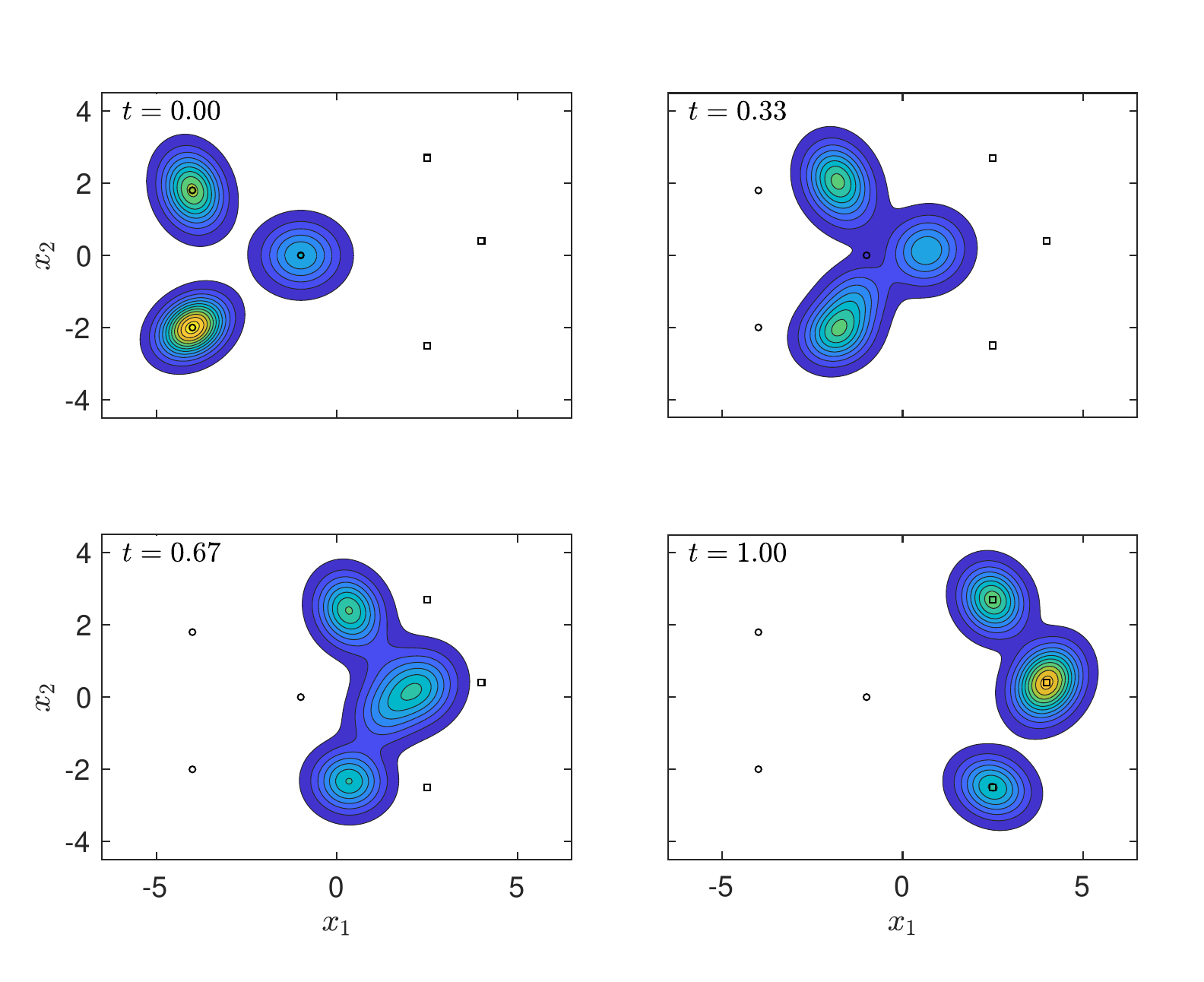}
    \label{fig:three_mode_density_flow}}
\subfloat{%
    \includegraphics[scale=0.27]{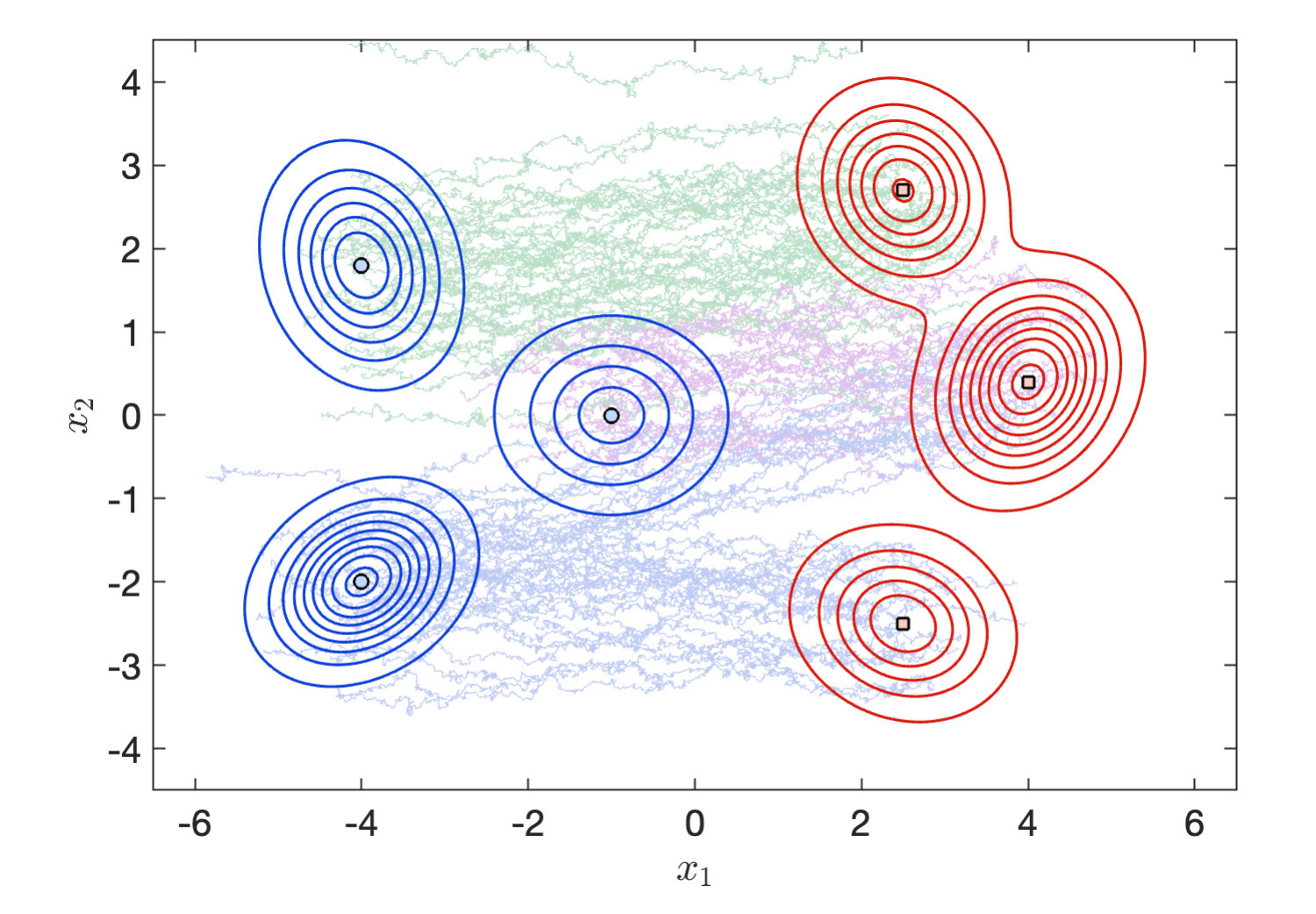}
    \label{fig:three_mode_density_sample_paths}}
\caption{The evolution of the projected density \(\rho_t^{\pi^{\star}}(\cdot)\) at several time-slices (left subfigure) and sample paths of the Markov diffusion under Euler-Maruyama steps (right subfigure).}
\label{fig:threemode_results}
\end{figure}

We choose the prior coupling to be the independent coupling \(\eta_{ij}=\alpha_i^0\alpha_j^1\) and solved the finite-dimensional lifted problem by Sinkhorn scaling to obtain \(\pi^{\star}\), as follows
\[
    \pi^\star \approx \begin{pmatrix}
    0.25 & 0.15 & 0\\
    0 & 0.05 & 0.30\\
    0 & 0.25 & 0
    \end{pmatrix},
\]
where entries below \(10^{-7}\) are displayed as zero. The row and column sums satisfy \(\sum_j\pi_{ij}^\star=\alpha_i^0,\) \(\sum_i\pi_{ij}^\star=\alpha_j^1.\) 
The transport contribution is approximately \(\sum_{i,j}\pi_{ij}^\star C_{ij}
\approx 21.87,\) and the discrete entropy term is approximately \(0.659\). 
Thus, the finite-dimensional lifted objective is approximately
\(\sum_{i,j}\pi_{ij}^\star C_{ij}+\eps\,\KL(\pi^\star\|\eta)
    \approx 22.06.\)

Given \(\pi^\star\), using the lifted marginal density \(\rho_t^{\pi^\star}\), and the projected Markov drift \(\bar u_t^{\pi^\star}\), we then simulated the Markov diffusion
\[
    \dd x_t=\bar u_t^{\pi^\star}(x_t)\,\dd t+\sqrt{\eps}\,\dd w_t, \quad  x_0\sim\rho_0.
\]
In the simulation, \(30,000\) particles were propagated using \(1,500\) Euler--Maruyama steps, and \(100\) sample paths were retained for visualization.

The right-hand subfigure in Figure~\ref{fig:threemode_results} shows sample paths generated by the projected Markov drift. The paths are colored according to their initial component, while the endpoint density contours are drawn on top. The figure illustrates the Eulerian projected-drift construction at the trajectory level. 
At the level of one-time marginals, the observed particle flow is consistent with the optimized coupling: density initially concentrated near each source component is transported toward the terminal regions favored by the large entries of \(\pi^\star\). 
The left-hand subfigure in Figure~\ref{fig:threemode_results} shows the analytic projected density \(\rho_t^{\pi^\star}\) at four times,
\(t=0,\tfrac13,\tfrac23,\) and \(t=1\). The panel at \(t=0\) recovers the prescribed initial mixture \(\rho_0\), and the panel at \(t=1\) recovers the prescribed terminal mixture \(\rho_1\). The intermediate panels show the splitting and recombination of density induced by the optimized coupling. This figure provides a direct numerical verification of the theoretical endpoint identities \(\rho_0^{\pi^\star}=\rho_0,\) and \(\rho_1^{\pi^\star}=\rho_1.\) We also report similar numerical statistics as before; see Table~\ref{tab:metadata_ex_2}.

\begin{table}[]
\centering
\caption{Numerical statistics. The computation time (in secs) is the median run times over 20 repetitions.}
\begin{tblr}{l c c c}
\hline[2pt]
\SetRow{azure9}
Method  & Size & Energy & Time \\
\hline
Direct SB & \(3111 \times 3111 \) & 21.197 & \(1.5\) \\
Ours &  \(3 \times 3\) & 21.267  & \(1.3 \times 10^{-2}\) \\
\hline[2pt]
\end{tblr}
\label{tab:metadata_ex_2}
\end{table}

\subsubsection{Example 2: An eight-mode splitting} \label{subsub:example:2} 

We next present a two-dimensional radial example designed to visualize the projected Markov drift when a single unimodal initial distribution is steered to a highly multimodal terminal distribution. 
This example illustrates how the posterior-averaged projected drift splits one initial cloud into several terminal clusters. 
We take \(d=2, T=1, \eps=0.08.\) The initial distribution is a single Gaussian centered at the origin, \(\rho_0(x)=\mathcal N(x;m_0,\Sigma_0),\) with \(m_0 \Let (0,0), \Sigma_0 \Let 0.15 I_2.\) 
The terminal distribution is an eight-component Gaussian mixture whose component means are arranged uniformly on a circle of radius \(r=5.1\), i.e., \(x \mapsto \rho_1(x) \Let \frac18\sum_{j=1}^8 \mathcal N(x;m_j^1,\Sigma_j^1),\) where
\[
    m_j^1 \Let r(\cos\theta_j,\sin\theta_j)\, \text{ and }\,
    \theta_j \Let \frac{2\pi(j-1)}{8} \quad\text{for }j=1,\dots,8.
\]
The terminal covariance matrices are anisotropic with their major axes
aligned with the radial directions. 
Specifically, if
\[
    R_j \Let \begin{pmatrix}
    \cos\theta_j & -\sin\theta_j\\
    \sin\theta_j & \cos\theta_j
    \end{pmatrix}, \quad\text{then} \quad  \Sigma_j^1 =
    R_j \begin{pmatrix}
    0.16 & 0\\
    0 & 0.10 \end{pmatrix} R_j^\top.
\]
Thus, the terminal density consists of eight localized Gaussian modes distributed around a circle. 
We carried out all other numerical steps as before and thus omitted the details. 
\begin{figure}[h]
\centering
\subfloat{%
    \includegraphics[scale=0.30]{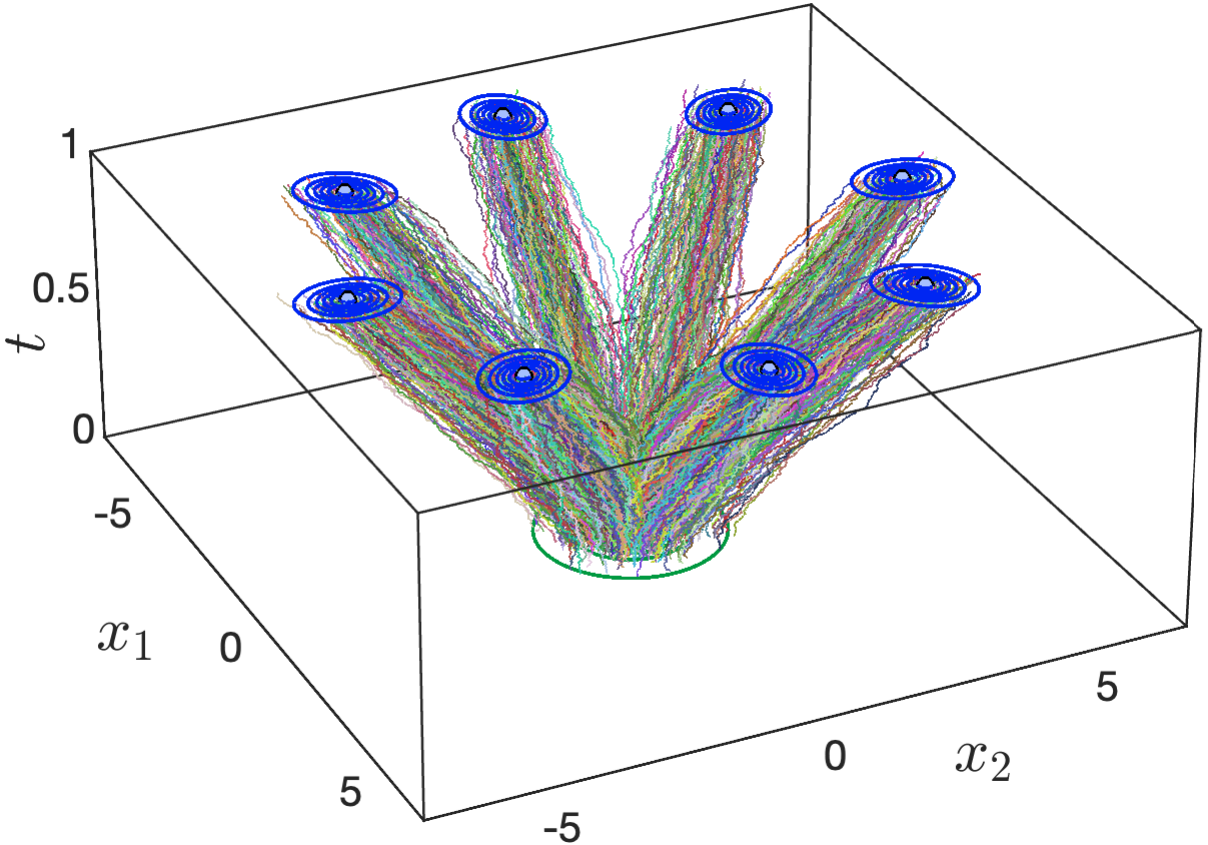}
    \label{fig:eight_mode_sample_paths}}
\hspace{0.01\textwidth}
\subfloat{%
    \includegraphics[scale=0.2]{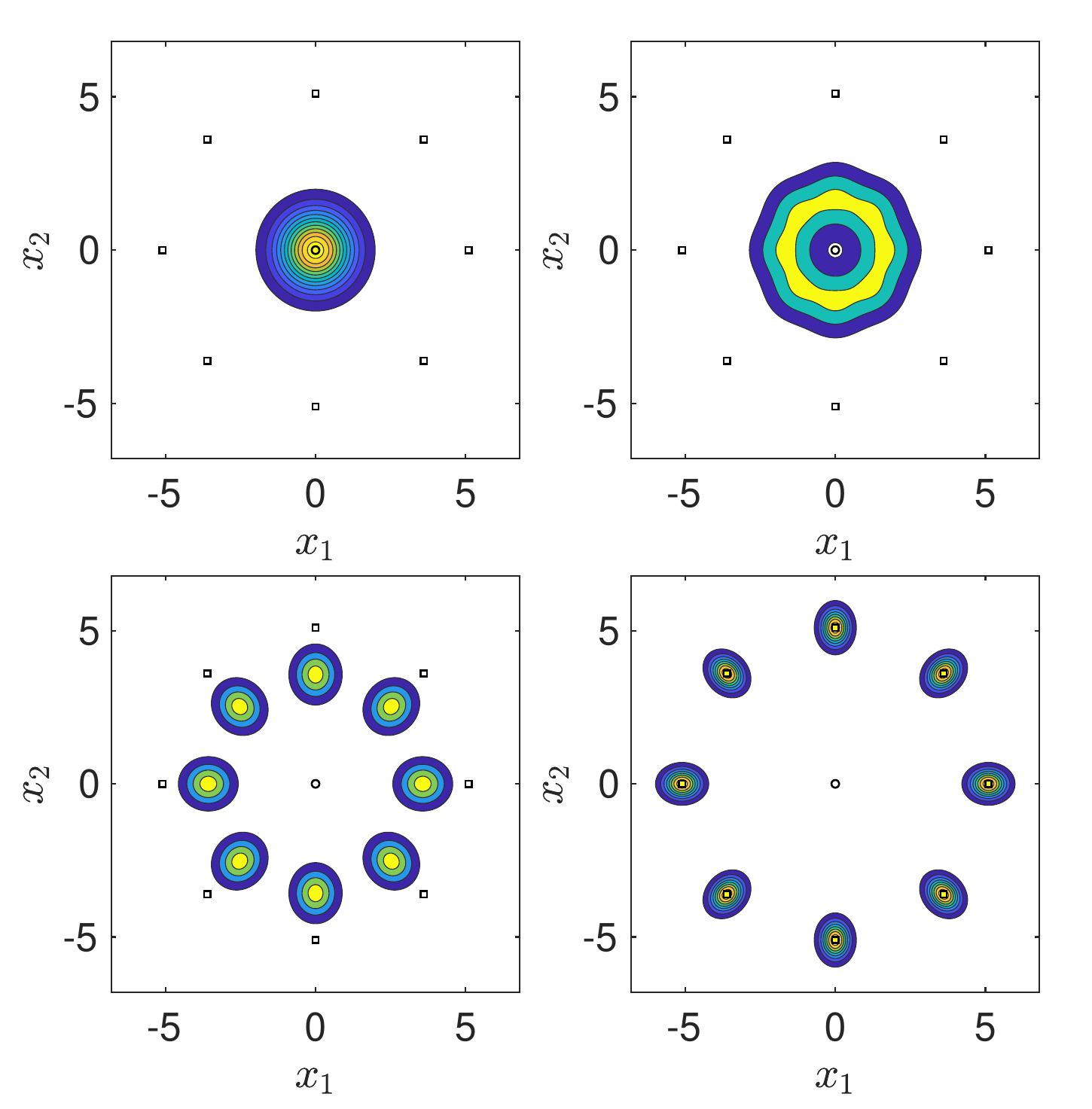}
    \label{fig:eight_mode_density_flow}}
\caption{The evolution of the sample paths of the Markov diffusion under Euler-Maruyama steps (left subfigure) and the corresponding density evolution (right subfigure) at different time-slices.}
\label{fig:eight_mode_shape_paths}
\end{figure}
The left subfigure in Figure \ref{fig:eight_mode_shape_paths} shows a three-dimensional space-time visualization in \((x_1,x_2,t)\). 
The green contour lies in the plane \(t=0\), representing the initial Gaussian. 
The blue contours lie in the plane \(t=1\), representing the eight terminal Gaussian components. 
The right subfigure in Figure~\ref{fig:eight_mode_shape_paths} shows the analytic projected density \(\rho_t^{\pi^\star}=\tfrac{1}{8}\sum_{j=1}^8 \rho_t^{1j}\) at several intermediate times. 
At \(t=0\), the density is the initial Gaussian at the origin. At \(t=1\), the density becomes the eight-component terminal mixture. The intermediate panels show how the initially unimodal density spreads and separates into multiple branches. This example demonstrates that the projected Markov drift is capable of producing a highly multimodal terminal distribution from a single unimodal initial distribution.

 \subsubsection{Example 3 (Shape matching)}\label{subsubsec:shape:matching}
\begin{figure}[h]
\centering
\subfloat{%
    \includegraphics[scale=0.17]{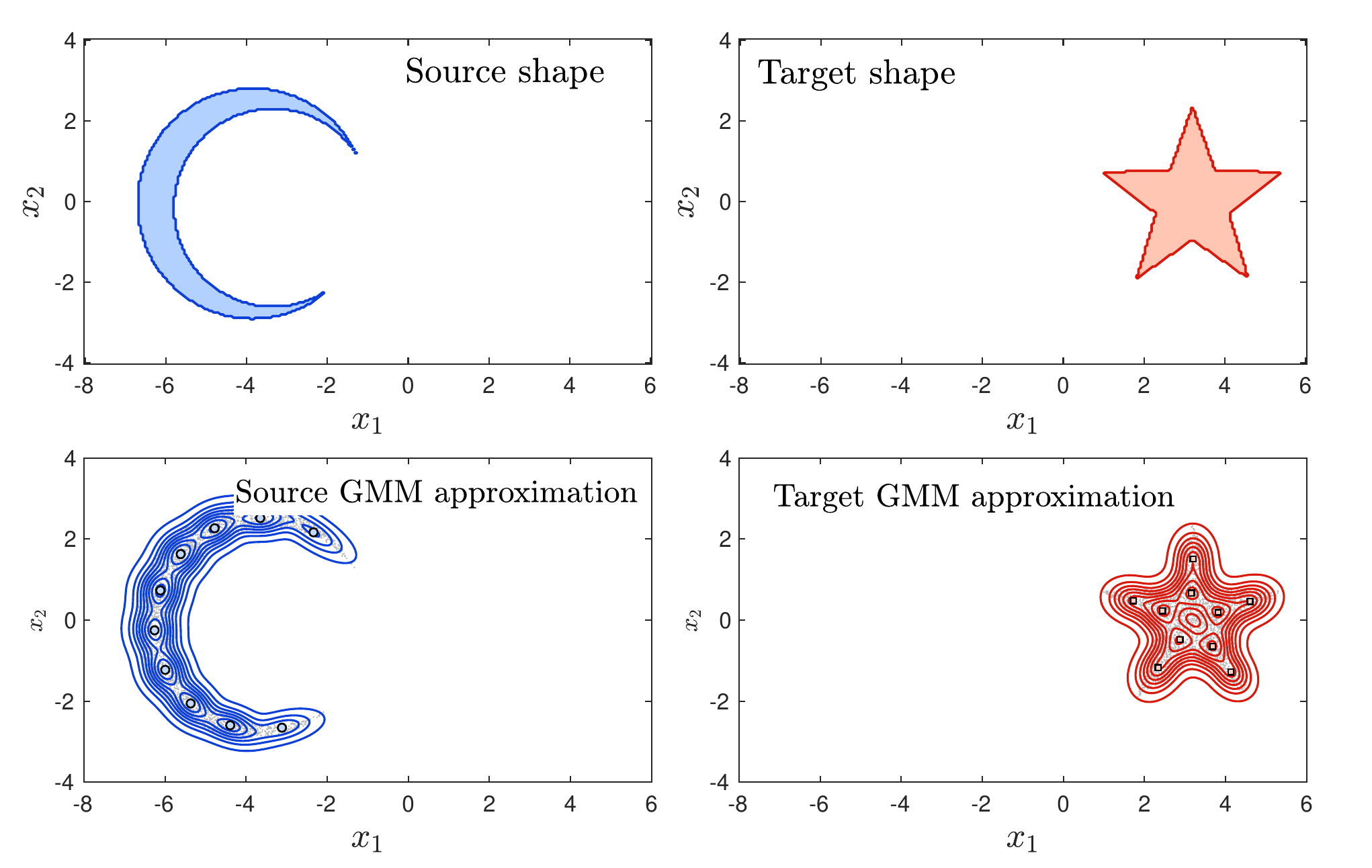}
    \label{fig:source:target:shape}}
\subfloat{%
    \hspace{-3mm}\includegraphics[scale=0.2]{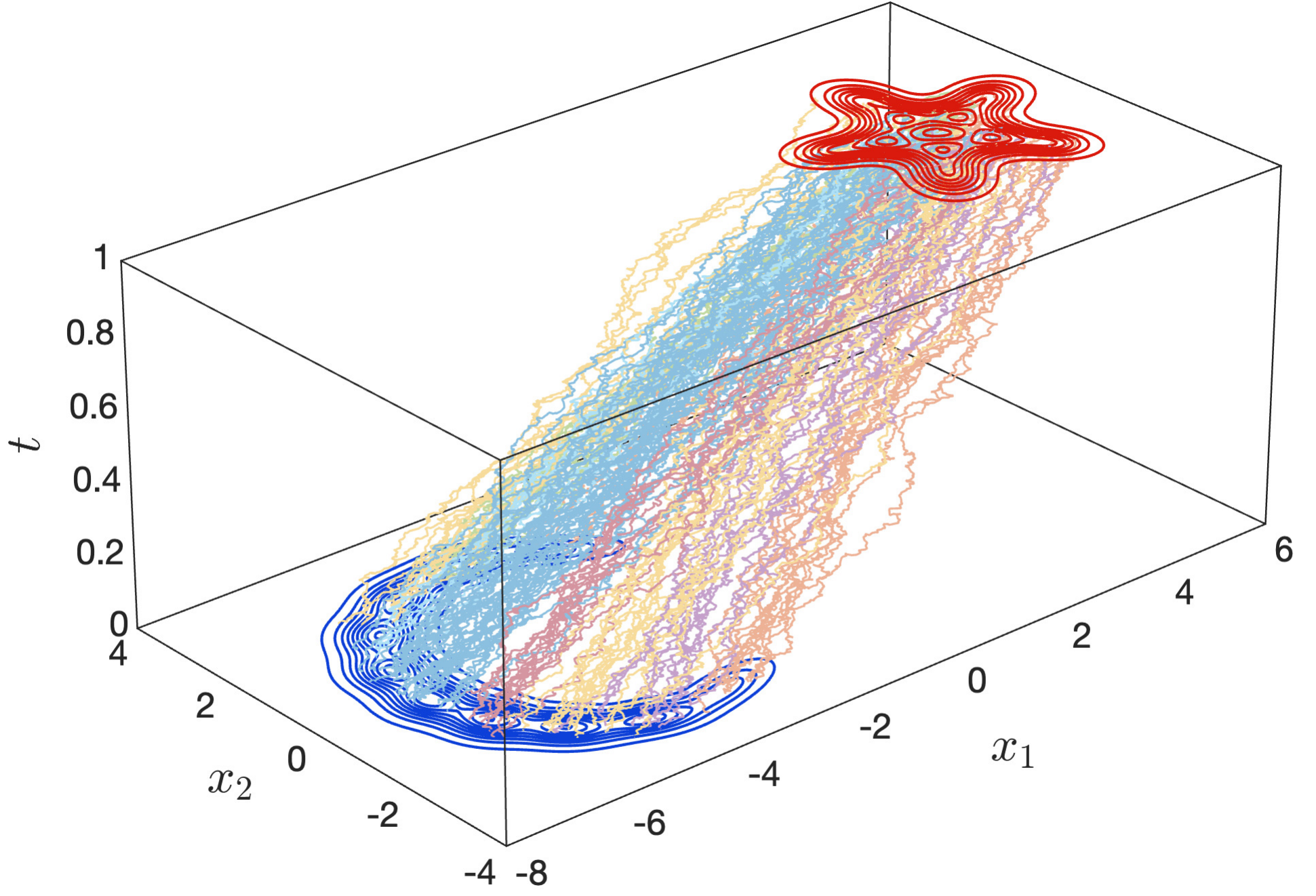}
    \label{fig:three_mode_density_sample_paths}}
\caption{The crescent source and star target shapes and their GMM approximations (left subfigure). The sample paths of the Markov diffusion under Euler-Maruyama steps (right subfigure).} 
\label{fig:shape_match_path_shape}
\end{figure}
As a final visual illustration, we consider an image-inspired density steering problem. A crescent-shaped source silhouette and a star-shaped target silhouette are first converted into point clouds, and each point cloud is approximated by a finite Gaussian mixture; see Figure \ref{fig:shape_match_path_shape}. 
More precisely, they were approximated by \(10\) Gaussian components with mixture weights as follows
\begin{align*}
\alpha^0 & = (0.13,0.05,0.10,0.07,0.09,0.13,0.12,0.12,0.04,0.10)^{\top} \nn\\ \alpha^1 &= (0.12,0.07,0.07,0.11,0.11,0.12,0.08,0.09,0.08,0.12)^\top.\nn
\end{align*}
Our lifted construction is then applied to these fitted Gaussian-mixture endpoints. The right-hand subfigure in Figure \ref{fig:shape_match_path_shape} shows sample paths generated by the projected Markov drift \(\bar u^{\pi^\star}\), illustrating the stochastic transport from the source region to the target region. 
Figure \ref{fig:shape_density} shows the corresponding analytic projected density evolution \(x \mapsto \rho_t^{\pi^\star}(x)\) which deforms the crescent-shaped density into the star-shaped density. 
\begin{figure}[h!]
\centering
    \includegraphics[scale=0.33]{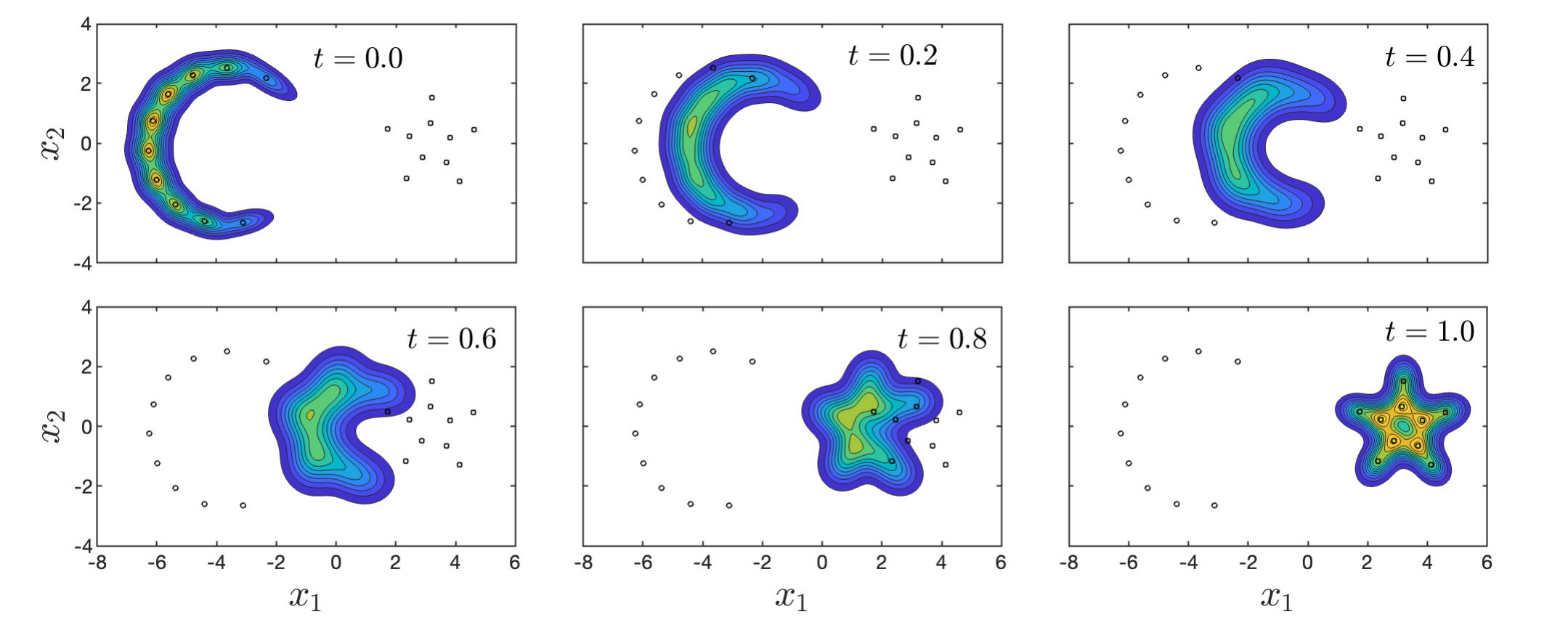}
    \label{fig:three_mode_density_sample_paths}
 \caption{The evolution of the projected density \(\rho_t^{\pi^{\star}}(\cdot)\) for the shape matching problem.}
\label{fig:shape_density}
\end{figure}

To illustrate the role of the prior component coupling \(\eta\) (as discussed in Remark \ref{rem:eta:interp}), we additionally perform a prior-sensitivity experiment for this example. We compare three admissible prior couplings \(\eta\). 
The first is the neutral product prior \(\eta^{\rm prod}\Let \alpha^0(\alpha^1)^\top\) (as employed in the previous examples).
The second is a diagonal-favoring prior, and the third is a rotated prior. Since the source and target weights are not uniform, the diagonal and rotated priors cannot be simple permutation matrices. 
Instead, we first construct two couplings \(\tilde \eta^{\rm diag},\tilde \eta^{\rm rot}\in\Pi(\alpha^0,\alpha^1)\), where \(\tilde \eta^{\rm diag}\) places as much mass as possible along the diagonal \(i=j\), while \(\tilde \eta^{\rm rot}\) places as much mass as possible along the shifted diagonal \(j=i+3\pmod{10}\). 
Any remaining mass is assigned greedily to satisfy the prescribed row and column sums. We then blend these structured couplings with the product prior to generate
\[
\eta^{\rm diag}\Let (1-\theta)\eta^{\rm prod}+\theta\tilde \eta^{\rm diag},\qquad \eta^{\rm rot}\Let (1-\theta)\eta^{\rm prod}+\theta\tilde \eta^{\rm rot},
\]
with \(\theta=0.90\). 
This blending guarantees strict positivity of all entries, while preserving the marginals \(\alpha^0\) and \(\alpha^1\).
For each choice \(\eta\in\{\eta^{\rm prod},\eta^{\rm diag},\eta^{\rm rot}\}\), we solve
\[
\pi^\star_\eta\in\argmin_{\pi\in\Pi(\alpha^0,\alpha^1)}\left\{\sum_{i,j}\pi_{ij}C_{ij}+\eps\KL(\pi\Vert\eta)\right\}.
\]
\begin{figure}
    \centering
    \includegraphics[scale=0.6]{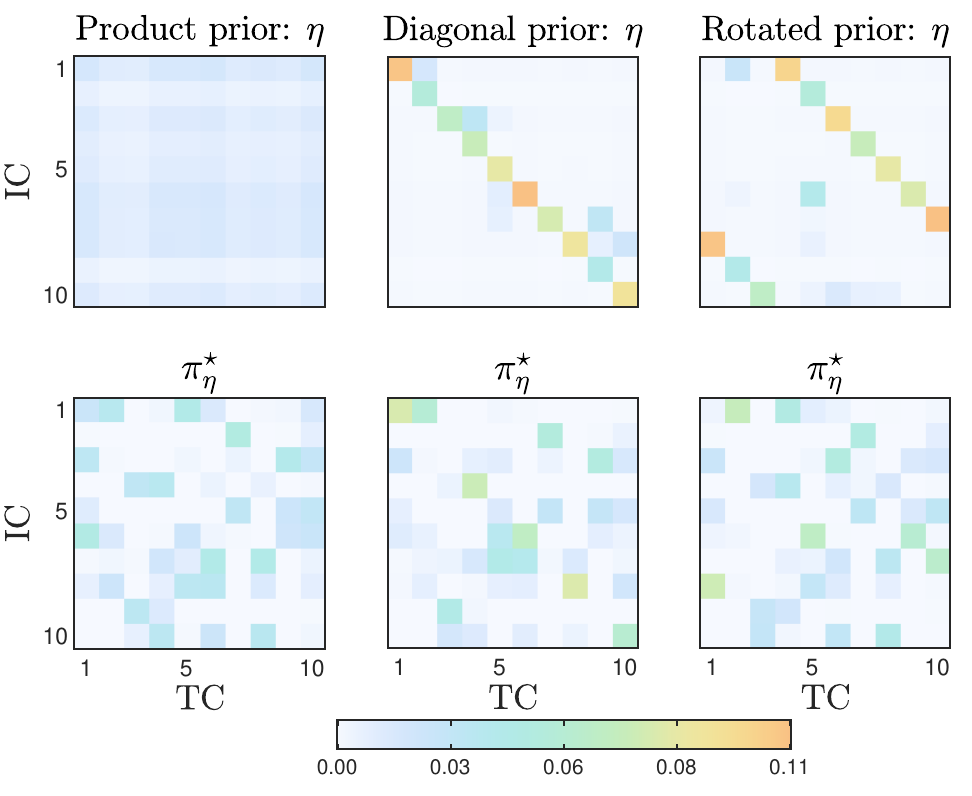}
  \caption{Effect of the prior component coupling in the shape-matching example. The top row shows the prescribed prior couplings \(\eta\), and the bottom row shows the corresponding optimized couplings \(\pi^\star_\eta\). The product prior is neutral, the diagonal-favoring prior encourages \(i\mapsto i\) assignments, and the rotated prior encourages the shifted assignment \(j=i+3\pmod{10}\). All entries represent fractions of total mass assigned from an initial component to a terminal component. Here IC denotes initial component and TC denotes terminal component.}
    \label{fig:silhouette_eta_prior_ablation}
\end{figure}
Figure~\ref{fig:silhouette_eta_prior_ablation} shows that the optimized component assignment depends visibly on the prior coupling. Consequently, as we discussed before, \(\eta\) is \emph{not merely an interpretability device}, it acts as a modeling input in the lifted assignment layer.

%% file: sections/Conclusion.tex
\section{Concluding remarks}\label{sec:diss:conc}

In this article, we developed a lifted Schr\"odinger bridge construction for Gaussian-mixture endpoints by augmenting trajectories with component labels, thereby reducing the mixture-to-mixture problem to explicit Gaussian component bridges and a finite-dimensional entropic coupling problem. 
Projecting the lifted law back to the original state space yields a feasible feedback construction that matches the prescribed endpoints, admits a certified kinetic-energy upper bound, and is numerically shown to achieve performance very close to direct grid-based Schr\"odinger bridge solvers, while being numerically faster and retaining a much more interpretable and low-dimensional component-level structure.

Several extensions of this work are natural. First, the Brownian Gaussian component bridges can be replaced by linear-system covariance-steering bridges, or, more generally, by numerically computed nonlinear-prior Schr\"odinger bridges. 
In that case, the lifted architecture remains the same, but the component quantities \((\rho_t^{ij},u_t^{ij},C_{ij})\) are computed numerically rather than in closed form. Second, state constraints can be incorporated at the component level by replacing each pairwise bridge with a constrained Schr\"odinger bridge or chance-constrained covariance-steering subproblem; the finite-dimensional assignment layer is then preserved, while the costs \(C_{ij}\) encode constrained transfer geometry. 
Third, unconstrained linear mean-field Schr\"{o}dinger bridges with Gaussian-mixture endpoints can be handled by combining deterministic mean steering with a lifted centered Gaussian-mixture bridge. Stochastic maximum principle \cite{red:SOC:VP:I,ref:SOC:VP:II,ref:adj:match:soc,ref:SOC:VP:IV}, unbalanced optimal transport \cite{ref:Chen:UOT:2,ref:UnOTDenCon:Extended:TAC}, and optimal stopping \cite{ref:OptStop:AS:SG:AP:PT,ref:AE:TTG:SB:Stop:Time} based viewpoints would also be interesting. Finally, a sharper analysis of the projection gap between the lifted and unlabeled path laws would further clarify when the lifted construction is close to the true unlabeled optimum.

\section{Acknowledgment}
The authors thank an anonymous reviewer for suggesting several relevant references and pointing out inconsistencies in their citation. 
During the preparation of this manuscript, the author(s) used OpenAI’s ChatGPT (model 5.6) to improve grammar and English usage and to assist in generating colored visualizations for some plots. The author(s) take full responsibility for all remaining errors.

%% file: sections/appendix.tex
\section{Appendix}\label{appen:proofs}

\begin{proof}[Proof of Proposition \ref{prop:pairwise_gaussian_bridge_formulas}]
Let \((x_t)_{t\in\lcrc{0}{1}}\) denote the canonical process under \(\probmeas^{ij}\). 
Since \(\probmeas^{ij}\) is a Brownian reciprocal process with Gaussian endpoint law, conditional on \((x_0,x_1)\), the path is a Brownian bridge. Therefore, for every \(t\in\lcrc{0}{1}\), \(x_t=(1-t)x_0+t x_1+\sqrt{\eps}\,B_t^{\mathrm{br}},\) where \((B_t^{\mathrm{br}})_{t\in\lcrc{0}{1}}\) is a standard Brownian bridge, independent of \((x_0,x_1)\), satisfying \(\E[B_t^{\mathrm{br}}]=0,\) and \(\operatorname{Cov}(B_t^{\mathrm{br}})=t(1-t)I.\) 
Taking expectations gives
\[
\E[x_t]=(1-t)\E[x_0]+t\,\E[x_1]=(1-t)m_i^0+t\,m_j^1,
\]
which proves the formula for \(m_t^{ij}\). 
Using independence of \(B_t^{\mathrm{br}}\) from \((x_0,x_1)\) and \(\operatorname{Cov}(x_0,x_1)=\Sigma_{01}^{ij}\), we obtain \(\operatorname{Cov}(x_t)=\operatorname{Cov}\big((1-t)x_0+t x_1\big)+\eps\,\operatorname{Cov}(B_t^{\mathrm{br}}).\) 
Hence,
\[
\Sigma_t^{ij}=(1-t)^2\Sigma_i^0+t^2\Sigma_j^1+t(1-t)\big(\Sigma_{01}^{ij}+(\Sigma_{01}^{ij})^{\top}\big)+\eps\,t(1-t)I,
\]
as desired. In particular, \(\Sigma_t^{ij}\in\mathbb S_{++}^d\) for every \(t\in\lcrc{0}{1}\), so \((\Sigma_t^{ij})^{-1}\) is well defined. 
Next, we derive the drift. 
For a Brownian bridge with random endpoints, the drift is~\cite[Chapter 5, Section 6]{ref:Shreve:Karatzas}
\[
u_t^{ij}(x)=\frac{1}{1-t}\big(\E[x_1\mid x_t=x]-x\big) \quad\text{for } t\in [0,1).
\]
Since \((x_t,x_1)\) is jointly Gaussian, the conditional expectation is affine, and thus, we get \(\E[x_1\mid x_t=x]=m_j^1+\operatorname{Cov}(x_1,x_t)(\Sigma_t^{ij})^{-1}(x-m_t^{ij}).\) 
Moreover, \(\operatorname{Cov}(x_1,x_t)=(1-t)(\Sigma_{01}^{ij})^{\top}+t\,\Sigma_j^1.\) 
Therefore,
\[
u_t^{ij}(x)=\frac{m_j^1-m_t^{ij}}{1-t}+\frac{\operatorname{Cov}(x_1,x_t)-\Sigma_t^{ij}}{1-t}(\Sigma_t^{ij})^{-1}(x-m_t^{ij}).
\]
Since \(m_j^1-m_t^{ij}=(1-t)(m_j^1-m_i^0),\) the constant term becomes \(c_{ij}=m_j^1-m_i^0.\) 
Also,
\[
\operatorname{Cov}(x_1,x_t)-\Sigma_t^{ij}=(1-t)\Big((1-t)\big((\Sigma_{01}^{ij})^{\top}-\Sigma_i^0\big)+t\big(\Sigma_j^1-\Sigma_{01}^{ij}-\eps I\big)\Big),
\]
and hence,
\[
\frac{\operatorname{Cov}(x_1,x_t)-\Sigma_t^{ij}}{1-t}=(1-t)\big((\Sigma_{01}^{ij})^{\top}-\Sigma_i^0\big)+t\big(\Sigma_j^1-\Sigma_{01}^{ij}-\eps I\big)=S_t^{ij}.
\]
Thus, we have \(u_t^{ij}(x)=S_t^{ij}(\Sigma_t^{ij})^{-1}(x-m_t^{ij})+c_{ij},\) which proves the affine drift formula in Proposition \ref{prop:GSB:2}.

Finally, by definition of the pairwise cost and the Girsanov identity, we have that
\[
C_{ij}=\int_0^1\int_{\mathbb R^d}\frac12\,\rho_t^{ij}(x)\,\|u_t^{ij}(x)\|^2 \ \dd x \, \dd t.
\]
Since under \(\rho_t^{ij}\) one has \(x_t-m_t^{ij}\sim\mathcal N(0,\Sigma_t^{ij})\), and \(u_t^{ij}(x_t)=A_t^{ij}(x_t-m_t^{ij})+c_{ij},\) the cross term vanishes after expectation, because \(\E[x_t-m_t^{ij}]=0\). 
Hence, \(\E\big[\|u_t^{ij}(x_t)\|^2\big]=\operatorname{tr}\!\Big(A_t^{ij}\Sigma_t^{ij}(A_t^{ij})^{\top}\Big)+\|c_{ij}\|^2.\) 
Therefore,
\[
C_{ij}=\frac12\int_0^1\left[\operatorname{tr}\!\Big(A_t^{ij}\Sigma_t^{ij}(A_t^{ij})^{\top}\Big)+\|c_{ij}\|^2\right] \dd t.
\]
Since \(c_{ij}=m_j^1-m_i^0\) is constant in time, this becomes
\[
C_{ij}=\frac12\|m_j^1-m_i^0\|^2+\frac12\int_0^1\operatorname{tr}\!\Big(A_t^{ij}\Sigma_t^{ij}(A_t^{ij})^{\top}\Big)\,\dd t,
\]
thus completing the proof.
\end{proof}

%% file: refs.bib
@book {ref:Protter:StochInt,
    AUTHOR = {P. E. Protter},
     TITLE = {Stochastic {I}ntegration and {D}ifferential {E}quations},
    SERIES = {Applications of Mathematics (New York)},
    VOLUME = {21},
   EDITION = {Second},
 PUBLISHER = {Springer-Verlag, Berlin},
      YEAR = {2004},
     PAGES = {xiv+415},
NOTE = {doi: \url{https://doi.org/10.1007/978-3-662-10061-5}}
}

@book {ref:LiptserShiryaev:Vol1,
    AUTHOR = {R. S. Liptser and A. N. Shiryaev},
     TITLE = {Statistics of {R}andom {P}rocesses. {I}. {G}eneral {T}heory},
    SERIES = {Stochastic Modelling and Applied Probability, Applications of Mathematics (New York)},
    VOLUME = {5},
   EDITION = {expanded},
 PUBLISHER = {Springer-Verlag, Berlin},
      YEAR = {2001},
     PAGES = {xvi+427},
   NOTE = {doi: \url{https://doi.org/10.1007/978-3-662-13043-8}}
}

@book {ref:Dupuis:LargeDev,
    AUTHOR = {P. Dupuis and R. S. Ellis},
     TITLE = {A {W}eak {C}onvergence {A}pproach to the {T}heory of {L}arge {D}eviations},
    SERIES = {Wiley Series in Probability and Statistics: Probability and
              Statistics},
 PUBLISHER = {John Wiley \& Sons, Inc., New York},
      YEAR = {1997},
     PAGES = {xviii+479},
       NOTE = {doi: \url{https://doi.org/10.1002/9781118165904}},
}

@book {ref:Rud-Analysis,
    AUTHOR = {W. Rudin},
     TITLE = {{P}rinciples of {M}athematical {A}nalysis},
    SERIES = {International Series in Pure and Applied Mathematics},
   EDITION = {Third},
PUBLISHER = {McGraw-Hill Book Co., New York-Auckland-D\"usseldorf},
      YEAR = {1976},
     PAGES = {x+342},
   MRCLASS = {26-02},
  NOTE = {doi: \url{https://doi.org/10.1017/S0013091500008889}}
}

@book{ref:FolReal-99,
  title={Real {A}nalysis: {M}odern {T}echniques and their {A}pplications},
  author={G. B. Folland},
  year={1999},
  publisher={John Wiley \& Sons},
  NOTE = {URL: \url{https://tinyurl.com/nycd2y2y}},
}

@book {ref:RevuzYor,
    AUTHOR = {D. Revuz and M. Yor},
     TITLE = {Continuous {M}artingales and {B}rownian {M}otion},
    SERIES = {Grundlehren der mathematischen Wissenschaften [Fundamental
              Principles of Mathematical Sciences]},
    VOLUME = {293},
   EDITION = {Third},
 PUBLISHER = {Springer-Verlag, Berlin},
      YEAR = {1999},
     PAGES = {xiv+602},
       NOTE = {doi: \url{https://doi.org/10.1007/978-3-662-06400-9}},
}

@book {ref:BainCrisan:Filtering,
    AUTHOR = {A. Bain and D. Crisan},
     TITLE = {Fundamentals of {S}tochastic {F}iltering},
    SERIES = {Stochastic Modelling and Applied Probability},
    VOLUME = {60},
 PUBLISHER = {Springer, New York},
      YEAR = {2009},
     PAGES = {xiv+390},
       NOTE = {doi: \url{https://doi.org/10.1007/978-0-387-76896-0}},
}

@book {ref:Stoc:Control:book,
    AUTHOR = {J. Yong and X. Y. Zhou},
     TITLE = {Stochastic {C}ontrols: {H}amiltonian {S}ystems and {HJB} {E}quations},
    SERIES = {Applications of Mathematics (New York)},
    VOLUME = {43},
 PUBLISHER = {Springer-Verlag, New York},
      YEAR = {1999},
     PAGES = {xxii+438},
       NOTE = {doi: \url{https://doi.org/10.1007/978-1-4612-1466-3}},
}

@article{ref:GoWithTheFlow,
  title={Go with the flow: Fast diffusion for {G}aussian mixture models},
  author={G. Rapakoulias and A. R. Pedram and F. Liu and L. Zhu and P. Tsiotras},
  journal={Advances in Neural Information Processing Systems},
  volume={38},
  pages={92879--92916},
  year={2026},
  NOTE = {URL: \url{https://openreview.net/forum?id=bmznY5wYXH}},
}

@misc{ref:UnOTDenCon:Extended:TAC,
      title={Globally Solving Unbalanced Optimal Transport and Density Control
for {G}aussian Distributions}, 
      author={H. Nakashima and S. Ganguly and K. Kashima},
      year={2026},
      eprint={2605.04246},
      archivePrefix={arXiv},
      primaryClass={math.OC},
      NOTE = {doi: \url{https://doi.org/10.48550/arXiv.2605.04246}}, 
}

@article{ref:mikami2006duality,
  title={Duality theorem for the stochastic optimal control problem},
  author={T. Mikami and M. Thieullen},
  journal={Stochastic processes and their applications},
  volume={116},
  number={12},
  pages={1815--1835},
  year={2006},
  publisher={Elsevier},
  NOTE = {doi: \url{https://doi.org/10.1016/j.spa.2006.04.014}}
}

@article{ref:mikami2008optimal,
  title={Optimal transportation problem by stochastic optimal control},
  author={T. Mikami and M. Thieullen},
  journal={SIAM Journal on Control and Optimization},
  volume={47},
  number={3},
  pages={1127--1139},
  year={2008},
  publisher={SIAM},
  NOTE = {doi: \url{https://doi.org/10.1137/050631264}}
}

@book {ref:Mikami:OT:book,
    AUTHOR = {T. Mikami},
     TITLE = {Stochastic {O}ptimal {T}ransportation --- {S}tochastic {C}ontrol with
              {F}ixed {M}arginals},
    SERIES = {SpringerBriefs in Mathematics},
 PUBLISHER = {Springer, Singapore},
      YEAR = {2021},
     PAGES = {xi+121},
       NOTE = {doi: \url{https://doi.org/10.1007/978-981-16-1754-6}},
}

@article{ref:GR:AP:PS:MF-SB,
  title={Steering Large Agent Populations Using Mean-Field {S}chr{\"o}dinger Bridges With {G}aussian Mixture Models},
  author={G. Rapakoulias and A. R. Pedram and P. Tsiotras},
  journal={IEEE Control Systems Letters},
  year={2025},
  publisher={IEEE},
  NOTE = {doi: \url{https://doi.org/10.1109/LCSYS.2025.3581859}}
}

@book {ref:Shreve:Karatzas,
    AUTHOR = {I. Karatzas and S. E. Shreve},
     TITLE = {Brownian {M}otion and {S}tochastic {C}alculus},
    SERIES = {Graduate Texts in Mathematics},
    VOLUME = {113},
   EDITION = {Second},
 PUBLISHER = {Springer-Verlag, New York},
      YEAR = {1991},
     PAGES = {xxiv+470},
      ISBN = {0-387-97655-8},
       NOTE = {doi: \url{https://doi.org/10.1007/978-1-4612-0949-2}},
}

@inproceedings{ref:cattiaux1994minimization,
  title={Minimization of the {K}ullback information of diffusion processes},
  author={P. Cattiaux and C. L{\'e}onard},
  booktitle={Annales de l'IHP Probabilit{\'e}s et statistiques},
  volume={30},
  number={1},
  pages={83--132},
  year={1994},
  NOTE = {URL: \url{https://www.numdam.org/item/AIHPB_1994__30_1_83_0/}}
}

@article {ref:Leonard:SB,
    AUTHOR = {C. L\'eonard},
     TITLE = {A survey of the {S}chr\"odinger problem and some of its
              connections with optimal transport},
   JOURNAL = {Discrete and Continuous Dynamical Systems},
    VOLUME = {34},
      YEAR = {2014},
    NUMBER = {4},
     PAGES = {1533--1574},
      ISSN = {1078-0947,1553-5231},
   MRCLASS = {60J25 (46N10 60F10)},
MRREVIEWER = {Nicolas\ Juillet},
        NOTE = {doi: \url{https://doi.org/10.3934/dcds.2014.34.1533}},
}

@inproceedings{ref:SB:1,
author = {J. Garg and X. Zhang and Q. Zhou},
title = {Soft-constrained {S}chr\"odinger Bridge: a Stochastic Control Approach},
year = {2024},
publisher = {PMLR},
booktitle = {Proceedings of the 27th International Conference on 
             Artificial Intelligence and Statistics (AISTATS)},
pages = {4429–4437},
volume = {238},
address = {Valencia, Spain},
NOTE = {URL: \url{https://proceedings.mlr.press/v238/garg24a.html}}
}

@inproceedings{ref:SB:3,
author = {N. Gushchin and S. Kholkin and E. Burnaev and A. Korotin},
title = {Light and Optimal {S}chr\"odinger Bridge Matching},
year = {2024},
publisher = {JMLR.org},
booktitle = {Proceedings of the 41st International Conference on Machine Learning},
number = {680},
address = {Vienna, Austria},
pages = {17100--17122},
NOTE = {URL: \url
{https://openreview.net/forum?id=EWJn6hfZ4J}}
}

@ARTICLE{ref:SB:4,
  author={A. Eldesoukey and O. M. Miangolarra and T. T. Georgiou},
  journal={IEEE Control Systems Letters}, 
  title={An Excursion Onto {S}chrödinger’s Bridges: Stochastic Flows With Spatio-Temporal Marginals}, 
  year={2024},
  volume={8},
  number={},
  pages={1138-1143},
  NOTE = {doi: \url{https://doi.org/10.1109/LCSYS.2024.3409107}}
}

@article{ref:Chen:UOT:2,
author = {Y. Chen and T. T. Georgiou and M. Pavon},
title = {Optimal Survival Strategies for Diffusive Flows: A {S}chrödinger Bridge Approach to Unbalanced Transport},
journal = {SIAM Review},
volume = {67},
number = {3},
pages = {579-604},
year = {2025},
NOTE = {doi: \url{https://doi.org/10.1137/25M176581X}},
}

@article{ref:SOC:SB:YC,
  title={Stochastic control liaisons: {R}ichard {S}inkhorn meets {G}aspard {M}onge on a {S}chr\"{o}dinger bridge},
  author={Y. Chen and T. T. Georgiou and M. Pavon},
  journal={SIAM Review},
  volume={63},
  number={2},
  pages={249--313},
  year={2021},
  publisher={SIAM},
  NOTE = {doi: \url{https://doi.org/10.1137/20M1339982}}
}

@article{ref:KI:KK:Disc:SB,
    author = {K. Ito and K. Kashima},
    title = {Maximum Entropy Optimal Density Control of Discrete-Time Linear Systems and {S}chr\"{o}dinger Bridges},
    journal = {IEEE Transactions on Automatic Control},
    year = {2023},
    pages = {1536--1551},
    NOTE = {doi: \url{https://doi.org/10.1109/TAC.2023.3305319}}
}

@article{ref:caluya2021wasserstein,
  title={Wasserstein proximal algorithms for the {S}chr{\"o}dinger bridge problem: Density control with nonlinear drift},
  author={K. F. Caluya and A. Halder},
  journal={IEEE Transactions on Automatic Control},
  volume={67},
  number={3},
  pages={1163--1178},
  year={2021},
  publisher={IEEE},
NOTE = {doi: \url{https://doi.org/10.1109/TAC.2021.3060704}},
}

@article{ref:YC:TG:MP:SB-and-OT,
  title={On the relation between optimal transport and {S}chr{\"o}dinger bridges: A stochastic control viewpoint},
  author={Y. Chen and T. T. Georgiou and M. Pavon},
  journal={Journal of Optimization Theory and Applications},
  volume={169},
  number={2},
  pages={671--691},
  year={2016},
  publisher={Springer},
  NOTE = {doi: \url{https://doi.org/10.1007/s10957-015-0803-z}},
}

@misc{ref:GP:MC:OT:book,
      title={Computational {O}ptimal {T}ransport}, 
      author={G. Peyré and M. Cuturi},
      year={2020},
      eprint={1803.00567},
      archivePrefix={arXiv},
      primaryClass={stat.ML},
      NOTE = {doi: \url{https://doi.org/10.48550/arXiv.1803.00567}}, 
}

@misc{ref:OT:Diff:GP,
      title={{Optimal and Diffusion Transports in Machine Learning}}, 
      author={G. Peyré},
      year={2025},
      eprint={2512.06797},
      archivePrefix={arXiv},
      primaryClass={math.OC},
      NOTE = {doi: \url{https://doi.org/10.48550/arXiv.2512.06797}}, 
}

@InProceedings{ref:SB:closed:form,
  title = 	 {The {S}chr\"{o}dinger Bridge between {G}aussian Measures has a Closed Form},
  author =       {C. Bunne and Y.-P. Hsieh and M. Cuturi and A. Krause},
  booktitle = 	 {Proceedings of The 26th International Conference on Artificial Intelligence and Statistics},
  pages = 	 {5802--5833},
  year = 	 {2023},
  editor = 	 {Ruiz, Francisco and Dy, Jennifer and van de Meent, Jan-Willem},
  volume = 	 {206},
  series = 	 {Proceedings of Machine Learning Research},
  month = 	 {25--27 Apr},
  publisher =    {PMLR},
  NOTE = 	 {URL: \url{https://proceedings.mlr.press/v206/bunne23a.html}}
}

@inproceedings{liu2022deep,
 author = {G.-H. Liu and T. Chen and O. So and E. A. Theodorou},
 booktitle = {Advances in Neural Information Processing Systems},
 address={Louisiana, LA},
 pages = {9374--9388},
 publisher = {Curran Associates, Inc.},
 title = {Deep Generalized {Schr\"odinger} Bridge},
 volume = {35},
 year = {2022},
 NOTE = {URL: \url{https://openreview.net/forum?id=fp33Nsh0O5}}
}

@inproceedings{liu2024generalized,
  title={Generalized {S}chr{\"o}dinger bridge matching},
  author={G.-H. Liu and Y. Lipman and M. Nickel and B. Karrer and E. A. Theodorou and R. T. Q. Chen},
  booktitle={International Conference on Learning Representations},
  year={2024},
  address={Vienna, Austria},
  NOTE = {URL: \url{https://openreview.net/forum?id=SoismgeX7z}}
}

@article{ruthotto2021introduction,
author = {L. Ruthotto and E. Haber},
title = {An introduction to deep generative modeling},
journal = {GAMM-Mitteilungen},
volume = {44},
number = {2},
pages = {e202100008},
NOTE = {doi: \url{https://doi.org/10.1002/gamm.202100008}},
year = {2021}
}

@inproceedings{song2020score,
  title={Score-Based Generative Modeling through Stochastic Differential Equations},
  author={Y. Song and J. Sohl-Dickstein and D. P. Kingma and A. Kumar and S. Ermon and B. Poole},
  booktitle={International Conference on Learning Representations},
  year={2021},
  address={held virtually},
  NOTE={URL: \url{https://openreview.net/forum?id=PxTIG12RRHS}}
}

@inproceedings{lipman2022flow,
title={Flow Matching for Generative Modeling},
author={Y. Lipman and R. T. Q. Chen and H. Ben-Hamu and M. Nickel and M. Le},
booktitle={The Eleventh International Conference on Learning Representations},
year={2023},
NOTE = {URL: \url{https://openreview.net/forum?id=PqvMRDCJT9t}}
}

@ARTICLE{ref:SB:halder:exact,
  author={A. M. H. Teter and W. Wang and A. Halder},
  journal={IEEE Transactions on Automatic Control}, 
  title={Schr\"{o}dinger Bridge With Quadratic State Cost is Exactly Solvable}, 
  year={2026},
  volume={71},
  number={5},
  pages={2903-2917},
  NOTE = {doi: \url{https://doi.org/10.1109/TAC.2025.3631521}}
  }

@misc{ref:adj:match:soc,
      title={Adjoint Matching through the Lens of the Stochastic Maximum Principle in Optimal Control}, 
      author={C. Domingo-Enrich and J. Han},
      year={2026},
      eprint={2604.08580},
      archivePrefix={arXiv},
      primaryClass={math.OC},
      NOTE={URL: \url{https://arxiv.org/abs/2604.08580}}, 
}

@article{chen2015optimal,
  title={Optimal Steering of a Linear Stochastic System to a Final Probability Distribution, {Part} {I}},
  author={Chen, Y. and Georgiou, T. and Pavon, M.},
  journal={IEEE Transactions on Automatic Control},
  volume={61},
  number={5},
  pages={1158--1169},
  year={2015},
NOTE = {doi: \url{https://doi.org/10.1109/TAC.2015.2457784}}
}

@ARTICLE{chenoptimalp3,
  author={Y. Chen and T. T. Georgiou and M. Pavon},
  journal={IEEE Transactions on Automatic Control}, 
  title={Optimal Steering of a Linear Stochastic System to a Final Probability Distribution—{P}art {III}}, 
  year={2018},
  volume={63},
  number={9},
  pages={3112-3118},
  NOTE = {https://doi.org/10.1109/TAC.2018.2791362}
  }

@article{chen2015optimal2,
  title={Optimal steering of a linear stochastic system to a final probability distribution, {P}art {II}},
  author={Y. Chen and T. T. Georgiou and M. Pavon},
  journal={IEEE Transactions on Automatic Control},
  volume={61},
  number={5},
  pages={1170--1180},
  year={2015},
  NOTE = {doi: \url{https://doi.org/10.1109/TAC.2015.2457791}}
}

@inproceedings{de2021diffusion,
  title={Diffusion {S}chr\"odinger bridge with applications to score-based generative modeling},
  author={V. De Bortoli and J. Thornton and J. Heng and A. Doucet},
  booktitle={Advances in Neural Information Processing Systems},
  volume={34},
  pages={17695--17709},
  year={2021},
  address={held virtually},
  NOTE = {URL: \url{https://openreview.net/forum?id=9BnCwiXB0ty}}
}

@article{tong2023simulation,
  title={Simulation-free {S}chr\"{o}dinger bridges via score and flow matching},
  author={A. Tong and N. Malkin and K. Fatras and L. Atanackovic and Y. Zhang and G. Huguet and G. Wolf and Y. Bengio},
  volume = {238},
  journal={Proceedings of the 27th International
Conference on Artificial Intelligence and Statistics (AISTATS) 2024, Valencia, Spain.},
  year={2023},
  NOTE = {doi: \url{https://doi.org/10.48550/arXiv.2307.03672}}
}

@article{follmer1988random,
  title={Random fields and diffusion processes},
  author={H. F{\"o}llmer},
  journal={Lecture Notes in Mathematics},
  volume={1362},
  pages={101--204},
  year={1988},
  NOTE = {doi: \url{https://doi.org/10.1007/BFb0086180}}
}

@inproceedings{shi2023diffusion,
  title={Diffusion {S}chr{\"o}dinger bridge matching},
  author={Y. Shi and V. De Bortoli and A. Campbell and A. Doucet},
  booktitle={Advances in Neural Information Processing Systems},
  volume={36},
  pages = {62183--62223},
  publisher = {Curran Associates, Inc.},
  year={2023},
  NOTE = {URL: \url{https://openreview.net/forum?id=qy07OHsJT5}}
}

@inproceedings{ref:dai:pra:1990markov,
  title={On the {M}arkov processes of {S}chr{\"o}dinger, the {F}eynman-{K}ac formula and stochastic control},
  author={P. D. Pra and M. Pavon},
  booktitle={Realization and Modelling in System Theory: Proceedings of the International Symposium MTNS-89, Volume I},
  pages={497--504},
  year={1990},
  organization={Springer},
  NOTE = {doi: \url{https://doi.org/10.1007/978-1-4612-3462-3_55}}
}

@article{ref:dai:pra:1991stochastic,
  title={A stochastic control approach to reciprocal diffusion processes},
  author={P. D. Pra},
  journal={Applied Mathematics and Optimization},
  volume={23},
  number={1},
  pages={313--329},
  year={1991},
  publisher={Springer},
  NOTE = {doi: \url{https://doi.org/10.1007/BF01442404}}
}

@article{mallasto2022entropy,
  title={Entropy-regularized 2-{W}asserstein distance between {G}aussian measures},
  author={A. Mallasto and A. Gerolin and H. Q. Minh},
  journal={Information Geometry},
  volume={5},
  number={1},
  pages={289--323},
  year={2022},
  publisher={Springer},
  NOTE = {doi: \url{https://doi.org/10.1007/s41884-021-00052-8}}
}

@inproceedings{red:SOC:VP:I,
 author = {C. Domingo-Enrich and J. Han and B. Amos and J. Bruna and R. T. Q. Chen},
 booktitle = {Advances in Neural Information Processing Systems},
 editor = {A. Globerson and L. Mackey and D. Belgrave and A. Fan and U. Paquet and J. Tomczak and C. Zhang},
 pages = {112459--112504},
 publisher = {Curran Associates, Inc.},
 title = {Stochastic Optimal Control Matching},
 NOTE = {doi: \url{https://proceedings.neurips.cc/paper_files/paper/2024/file/cc32ec39a5073f61d38c338d963df30d-Paper-Conference.pdf}},
 volume = {37},
 year = {2024}
}

@inproceedings{ref:SOC:VP:II,
title={Adjoint Matching: Fine-tuning Flow and Diffusion Generative Models with Memoryless Stochastic Optimal Control},
author={C. Domingo-Enrich and M. Drozdzal and B. Karrer and R. T. Q. Chen},
booktitle={The Thirteenth International Conference on Learning Representations},
year={2025},
NOTE = {doi: \url{https://openreview.net/forum?id=xQBRrtQM8u}}
}

@misc{ref:SOC:VP:IV,
      title={Flow Matching for Measure Transport and Feedback Stabilization of Control-Affine Systems}, 
      author={K. Elamvazhuthi},
      year={2026},
      eprint={2510.02706},
      archivePrefix={arXiv},
      primaryClass={math.OC},
      NOTE = {doi: \url{https://doi.org/10.48550/arXiv.2510.02706}}, 
}

@INPROCEEDINGS{ref:LQR:SB,
  author={M. Lambert},
  booktitle={IEEE 64th Conference on Decision and Control (CDC)}, 
  title={The {LQR}-{S}chr\"{o}dinger Bridge}, 
  year={2025},
  volume={},
  number={},
  pages={3149-3156},
  NOTE = {doi: \url{https://doi.org/10.1109/CDC57313.2025.11312252}}
  }

@article{ref:VSB:Diff:1,
author = {V. S. Borkar},
title = {{Controlled Diffusion Processes}},
volume = {2},
journal = {Probability Surveys},
publisher = {Institute of Mathematical Statistics and Bernoulli Society},
pages = {213 -- 244},
year = {2005},
NOTE = {doi: \url{https://doi.org/10.1214/154957805100000131}}
}

@book{ref:AA:VSB:Control:Diff, 
place={Cambridge}, 
series={Encyclopedia of Mathematics and its Applications}, 
title={Ergodic {C}ontrol of {D}iffusion {P}rocesses}, 
publisher={Cambridge University Press},
author={A. Arapostathis and V. S. Borkar and M. K. Ghosh},
year={2011},
collection={Encyclopedia of Mathematics and its Applications},
NOTE = {doi: \url{https://doi.org/10.1017/CBO9781139003605}},
}

@article{delon_wasserstein-type_2020,
  title={A {W}asserstein-type distance in the space of {G}aussian mixture models},
  author={J. Delon and A. Desolneux},
  journal={SIAM Journal on Imaging Sciences},
  volume={13},
  number={2},
  pages={936--970},
  year={2020},
  publisher={SIAM},
  NOTE = {doi: \url{https://doi.org/10.1137/19M1301047}}
}

@InProceedings{mei_flow_2024,
  title = 	 {Flow matching for stochastic linear control systems},
  author =       {Y. Mei and M. Al-Jarrah and A. Taghvaei and Y. Chen},
  booktitle = 	 {Proceedings of the 7th Annual Learning for Dynamics \& Control Conference},
  pages = 	 {484--496},
  year = 	 {2025},
  volume = 	 {283},
  month = 	 {04--06 Jun},
  publisher =    {PMLR},
  NOTE = 	 {URL: \url{https://proceedings.mlr.press/v283/mei25a.html}}
}

@article{chen2018optimal,
  title={Optimal transport for {G}aussian mixture models},
  author={Y. Chen and T. T. Georgiou and A. Tannenbaum},
  journal={IEEE Access},
  volume={7},
  pages={6269--6278},
  year={2018},
  publisher={IEEE},
  NOTE = {doi: \url{https://doi.org/10.1109/ACCESS.2018.2889838}}
}

@ARTICLE{chen2016optimal,
  author={Y. Chen and T. T. Georgiou and M. Pavon},
  journal={IEEE Transactions on Automatic Control}, 
  title={Optimal Transport Over a Linear Dynamical System}, 
  year={2017},
  volume={62},
  number={5},
  pages={2137-2152},
  NOTE ={doi: \url{https://doi.org/10.1109/TAC.2016.2602103}},
  }

@misc{ref:OptStop:AS:SG:AP:PT,
      title={Nonlinear Stochastic Optimal Control and Optimal Stopping using the {F}okker--{P}lanck Transformation}, 
      author={A. Selim and S. Ganguly and A. Pakniyat and P. Tsiotras},
      year={2026},
      eprint={2604.12153},
      archivePrefix={arXiv},
      primaryClass={math.OC},
      NOTE ={doi: \url{https://doi.org/10.48550/arXiv.2604.12153}}, 
}

@ARTICLE{ref:SB:original:I,
  author={E. Schr\"{o}dinger},
  journal={Sitzungsberichte der Preuss. Phys. Math. Klasse}, 
  title={Über die umkehrung der naturgesetze}, 
  year={1931},
  volume={10},
  pages={144--153},
  NOTE ={URL: \url{https://leonard.perso.math.cnrs.fr/papers/1931-Schroedinger-Ueber%20die%20Umkehrung%20der%20Naturgesetze.pdf}},
  }

@article{ref:SB:original:II,
     author = {E. Schr\"{o}dinger},
     title = {Sur la th\'{e}orie relativiste de l'\'electron et l'interpr\'etation de la m\'ecanique quantique},
     journal = {Annales de l'institut Henri Poincar\'e},
     pages = {269--310},
     year = {1932},
     publisher = {INSTITUT HENRI POINCAR\'E ET LES PRESSES UNIVERSITAIRES DE FRANCE},
     volume = {2},
     number = {4},
     zbl = {0004.42505},
     language = {fr},
     NOTE = {URL: \url{https://www.numdam.org/item/AIHP_1932__2_4_269_0/}}
}

@misc{ref:AH:Lie:Group:SB,
      title={Schr\"{o}dinger Bridge Over A Compact Connected {L}ie Group}, 
      author={H. Mahmood and A. Halder and A. Akhtar},
      year={2026},
      eprint={2603.14049},
      archivePrefix={arXiv},
      primaryClass={math.OC},
      NOTE={URL: \url{https://arxiv.org/abs/2603.14049}}, 
}

@misc{ref:Adu:SB:Average:Sys,
      title={Schr\"{o}dinger Bridge over Averaged Systems}, 
      author={D. O. Adu and Y. Chen},
      year={2024},
      eprint={2412.03294},
      archivePrefix={arXiv},
      primaryClass={math.OC},
      NOTE={URL: \url{https://arxiv.org/abs/2412.03294}}, 
}

@misc{ref:Adu:SB:SubRiemann,
      title={From {S}chr\"{o}dinger Bridge to Optimal Transport over Sub-{R}iemannian Manifolds}, 
      author={D. O. Adu and K. Elamvazhuthi and B. Gharesifard},
      year={2026},
      eprint={2605.11429},
      archivePrefix={arXiv},
      primaryClass={math.OC},
      NOTE = {doi: \url{https://doi.org/10.48550/arXiv.2605.11429}}, 
}

@book{ref:santambrogio2023course,
  title={A {C}ourse in the {C}alculus of {V}ariations: {O}ptimization, {R}egularity, and {M}odeling},
  author={F. Santambrogio},
  year={2023},
NOTE = {doi: \url{https://doi.org/10.1007/978-3-031-45036-5}},
pages={338+XXI},
  publisher={Universitext, Springer Nature}
}

@book {ref:Santam:OT:book,
    AUTHOR = {F. Santambrogio},
     TITLE = {Optimal {T}ransport for {A}pplied {M}athematicians: {C}alculus of {V}ariations, {PDE}s, and {M}odeling},
    SERIES = {Progress in Nonlinear Differential Equations and their
              Applications},
    VOLUME = {87},
 PUBLISHER = {Birkh\"auser/Springer, Cham},
      YEAR = {2015},
     PAGES = {xxvii+353},
       NOTE = {doi: \url{https://doi.org/10.1007/978-3-319-20828-2}},
}

@book{ref:OT:book:Villani,
    AUTHOR = {C. Villani},
     TITLE = {Topics in {O}ptimal {T}ransportation},
    SERIES = {Graduate Studies in Mathematics},
    VOLUME = {58},
 PUBLISHER = {American Mathematical Society, Providence, RI},
      YEAR = {2003},
     PAGES = {xvi+370},
     NOTE = {doi: \url{https://doi.org/10.1090/gsm/058}},
}

@ARTICLE{ref:AE:TTG:SB:Stop:Time,
  author={A. Eldesoukey and T. T. Georgiou},
  journal={IEEE Transactions on Automatic Control}, 
  title={Schr\"{o}dinger's Control and Estimation Paradigm With Spatio-Temporal Distributions on Graphs}, 
  year={2025},
  volume={70},
  number={4},
  pages={2466-2478},
  NOTE = {doi: \url{10.1109/TAC.2024.3485537}},
  }
